\documentclass[12pt,twoside]{article}
\usepackage[hmargin=0.8in,vmargin=0.9in]{geometry}
\usepackage{fancyhdr}
\usepackage{graphicx}
\usepackage{amssymb}
\usepackage{amsmath}
\usepackage{amsthm}
\usepackage{amsfonts}
\usepackage{mathrsfs}
\usepackage{mathtools}
\usepackage{bm}
\usepackage{color}
\usepackage{setspace}
\usepackage{exscale}
\usepackage{relsize}
\usepackage{float}
\usepackage{cite}
\usepackage{hyperref}
\usepackage{cleveref}
\usepackage{nicefrac}
\usepackage{placeins}
\usepackage{booktabs}
\DeclareMathAlphabet{\mathbbold}{U}{bbold}{m}{n}

\usepackage{booktabs,multirow} 
\usepackage{array} 
\usepackage{paralist} 

\usepackage{sectsty}

\sectionfont{\fontsize{14}{15}\selectfont}

\theoremstyle{plain}                    
\newtheorem{thm}{Theorem}[section]

\newtheorem{rmk}[thm]{Remark}

\usepackage{tikz}
\usetikzlibrary{positioning}
\usepackage{xcolor}

\allowdisplaybreaks[1]

\numberwithin{equation}{section}
\numberwithin{figure}{section}
\numberwithin{table}{section}

\newcommand*\xbar[1]{%
	\hbox{%
		\vbox{%
			\hrule height 0.5pt 
			\kern0.4ex
			\hbox{%
				\kern-0.05em
				\ensuremath{#1}%
				\kern-0.00em
			}%
		}%
	}%
}

\usepackage[utf8]{inputenc}
\usepackage[english]{babel}
\usepackage{empheq}
\usepackage{subcaption} 
\usepackage{epsfig}
\usepackage{algorithm}
\usepackage{algpseudocode}
\usepackage{bbm}
\usepackage{sectsty}

\newcommand{\bb}[1]{{\boldsymbol{#1}}} 
\newcommand{\mc}[1]{{\mathcal{#1}}} 

\def\R{\mathbb{R}}

\newcommand{\norm}[1]{{\left\| #1 \right\|}}
\newcommand{\p}[1]{{\left( #1 \right)}}
\newcommand{\set}[1]{{\left\{ #1 \right\}}}
\newcommand{\br}[1]{{\left[ #1 \right]}}
\newcommand{\abs}[1]{{\left| #1 \right|}}

\def\eps{\varepsilon}

\newcommand{\pd}[3][]{\frac{\partial^{#1} #2}{\partial #3^{#1}}}
\newcommand{\ddt}[1]{\frac{{\rm d} #1}{{\rm d} t}}

\newcommand{\jph}{{j+\frac{1}{2}}}
\newcommand{\jmh}{{j-\frac{1}{2}}}
\newcommand{\jpf}{{j+\frac{1}{4}}}
\newcommand{\jmf}{{j-\frac{1}{4}}}
\newcommand{\dx}{\Delta x}

\newcommand{\dt}{\Delta t}

\graphicspath{{Figures/}}

\title{High-order positivity-preserving well-balanced schemes: combining flux globalization with modified Patankar deferred correction methods}
\author{Philipp \"Offner%
	\thanks{Institute of Mathematics, Clausthal University of Technology,  Erzstraße 1, 38678 Clausthal-Zellerfeld, Germany; {\tt philipp.oeffner@tu-clausthal.de}}, 
	Michael Redle%
	\thanks{Institute of Mathematics, Clausthal University of Technology, Erzstraße 1, 38678 Clausthal-Zellerfeld, Germany; {\tt michael.redle@tu-clausthal.de}}, and %
	Davide Torlo%
	\thanks{Dipartimento di Matematica Guido Castelnuovo, Universit\`a di Roma La Sapienza, p.le Aldo Moro 5, Rome, 00185, Italy; {\tt davide.torlo@uniroma1.it}}}
\date{\today}
\numberwithin{equation}{section}
\begin{document}
\maketitle

\begin{abstract}

	Positivity-preserving schemes for hyperbolic balance laws are typically built from the Zhang-Shu limiter and strong-stability-preserving Runge-Kutta (SSP-RK) time stepping, at the cost of a restrictive time-step condition.
	The modified Patankar deferred correction (mPDeC) method removes this restriction, offering high-order accuracy, conservation, and unconditional positivity.
	However, its solution-dependent Patankar weights make the time update nonlinear, so, unlike with RK methods, preserving (steady and moving) equilibria, i.e., well-balancing (WB), is no longer straightforward.
	In this paper, we construct high-order, WB, positivity-preserving mPDeC schemes by combining mPDeC with a flux-globalization approach.
	We demonstrate the idea on the one-dimensional shallow water equations, using a WB A-WENO spatial discretization, and introduce a modified local characteristic decomposition that remains stable in the presence of dry states.
	Numerical experiments on one-dimensional benchmarks confirm that the method satisfies the desired properties.
\end{abstract}

\noindent\textbf{Keywords:} Non-conservative hyperbolic systems of nonlinear PDEs, well-balanced schemes, positivity-preserving schemes, modified Patankar Deferred Correction time evolution, A-WENO, shallow water equations.
\\

\noindent\textbf{AMS subject classification:} 35B09, 35L65, 65L06, 65M06, 76M20, 86A05

\section{Introduction}\label{sec1}

The focus of this paper is on the development of high-order well-balanced positivity-preserving methods for nonconservative hyperbolic systems of balance laws.
The one-spatial-dimension (1-D) case reads
\begin{equation} \label{e:noncons}
	\bb{U}_t + \bm{F}(\bb U)_x = B(\bm U) \bm U_x,
\end{equation}
where $x$ and $t$ denote the spatial and temporal independent variables, respectively, $\bb{U}\in\R^d$ denotes the vector of evolved quantities, $\bm{F}(\bb{U})$ is the physical flux, and $B(\bb{U})$ represents the matrix of the nonconservative term. 

One such system that fits the nonconservative form \eqref{e:noncons} are the shallow water (SW) equations.
Also known as the Saint-Venant equations, this system describes the evolution of hydrostatic free-surface flows under the influence of gravity and are widely used in applications including river floods and tsunami modeling. 
This nonlinear hyperbolic system of PDEs is derived under the assumption that the characteristic wavelength is large in comparison to the water depth. 
In this work, we consider the SW equations with time-independent bathymetry and without friction:
\begin{equation}\label{e:sw_vec}
	\bm U = \begin{pmatrix}
		h\\ q \\ Z
	\end{pmatrix}, \qquad
	\bm F(\bm U) = \begin{pmatrix}
		q\\ \frac{q^2}{h} + \frac{1}{2}gh^2 \\ 0
	\end{pmatrix}, \qquad
	B(\bm U) = \begin{pmatrix}
		0 & 0 & 0\\
		0 & 0 & -gh \\ 
		0 & 0 & 0
	\end{pmatrix},
\end{equation}
where $h(x,t)$ denotes the fluid thickness, $u(x,t)$ is the velocity, $q \equiv hu$ is the discharge, $Z(x)$ denotes the time-independent bottom topography, and $g$ is the gravitational constant.

Many, if not all, of the numerical approximations of the SW system share the same overarching goal---to provide reliable and physically meaningful simulations at reasonable computational cost.
In this context, high-order methods are particularly attractive, as they can achieve high accuracy on relatively coarse discretizations.
Furthermore, a supplementary strategy for addressing reliability and efficiency is to incorporate structural properties of the underlying PDE system into the discretization directly. 
These structure-preserving methods not only can improve the discretization, but additionally uphold the implicit physical constraints seen within the governing PDE system.
For example, for the SW equations, important properties include positivity of the fluid thickness $h$ and preservation of equilibrium states. 
Schemes addressing these properties are referred to as positivity-preserving and well-balanced (WB) schemes, respectively.

Structure-preserving schemes for the SW system thus remain an active area of research, with a wide variety of approaches developed for different applications and settings.
In the context of SW WB and/or positivity-preserving approaches, there exists many recent efforts in this direction alone---including related developments in finite volume (FV) methods
\cite{berthon2022very, 
	castro-pares-jsc20, 
	Ciallella23WB, 
	michel2017well, 
	kurganov2018finite 
	}, 
continuous and discontinuous finite element methods
\cite{ranocha2017shallow, 
	mantri2024fully, 
	Guerrero22WB, 
	hennes2022, 
	Xing13WB 
	},
among other directions
\cite{abgrall_2022, 
	mario1, 
	Chu25WB5 
	}.
The present paper aims to build upon these developments for SW in a high-order setting; in doing so, we must address the following challenges. 

First, to ensure the WB property and preserve the lake-at-rest ($q \equiv 0$) and moving water ($q \equiv \textrm{const.}\neq 0$) steady state solutions of \eqref{e:sw_vec}, the scheme must exactly balance the flux derivative and the nonconservative terms.
This is not a straightforward computation, as commonly-used shock-capturing discretizations may result in spurious oscillations that are orders of magnitude larger than the original perturbation wave.
Using a very fine mesh might suppress these oscillations---however, this approach drastically increases computation time, and thus often times proves impractical. 
Rather, WB schemes are able to discretely preserve the steady state solutions, and can capture small perturbations of such equilibria even on coarse meshes.
Additionally, they can mitigate small error accumulation between the flux derivative and the source terms, which is particularly important for long-time simulations. 
One popular numerical approach to ensure the WB property is referred to as the flux-globalization approach (see, e.g., \cite{Gascon01Const,castro-pares-jsc20,Chertock18WB,Kurganov23WBPCCU,Ciallella23WB}), where one takes advantage of the quasi-conservative form of \eqref{e:noncons} and reconstructs/interpolates the global fluxes, or the so-called equilibrium variables. This technique has also been extended to capture multi-dimensional steady states \cite{barsukow2024structure,barsukow2026genuinely,barsukow2026stationarity,ziggaf2026arbitrary}.

To ensure positivity of the SW system, we must preserve the positivity constraint $h \geq 0$. 
This is vital for many applications of SW, such as flood simulations and simulations with shorelines.
Moreover, if this constraint is violated numerically, the solution exits the physically relevant regime and the numerical simulation crashes.
One approach to preserve the positivity of the fluid thickness in the context of high-order weighted essentially non-oscillatory (WENO) schemes (see, e.g., \cite{shu1998essentially}) has been through the development of positivity-preserving limiters \cite{zhang2011maximum,perthame1996positivity}.
However, if updating in time using classical explicit schemes such as SSP-RK \cite{gottlieb2001strong}, these limiters impose a restrictive CFL condition. 
Recent strategies have since overcome this limitation, as unconditional positivity-preserving high-order time-stepping techniques based on the production-destruction formulation and the modified Patankar nonlinear weighting technique (see, e.g., \cite{meister2019,burchard2003,torlo2020}) have been proposed \cite{ciallella2025,bender2025, arXiv:2602.14392, ciallella2022}. These techniques allow to remove the restrictive CFL condition at a price of a very sparse linear system to be solved at each time step.

In this paper, we incorporate the flux-globalization WB approach within the arbitrary high-order production-destruction framework of the unconditionally positive modified Patankar deferred correction (mPDeC) time discretization. 
While the framework allows for different combinations of flux-globalization techniques with mPDeC, we demonstrate its effectiveness using the flux-globalization-based WB alternative-WENO (A-WENO) scheme introduced in \cite{Chu25WB5}.
Furthermore, we introduce a modification of the local characteristic decomposition of equilibrium variables, introduced in \cite{chu24local}, that improves the robustness of the approach in the presence of dry states. 
The proposed scheme is thus positivity-preserving, and exactly WB in the sense that it preserves the physically relevant wet/dry lake-at-rest and moving water equilibria of the SW system. 
We demonstrate the performance through a series of numerical experiments, where the results confirm that the proposed approach combines high-order accuracy with robustness, positivity preservation, and the exactly WB property---while avoiding spurious oscillations.

The paper is organized as follows. 
In \S\ref{sec2.1} and \S\ref{sec2.2}, we present the WB A-WENO path-conservative central-upwind (PCCU) scheme for general nonconservative hyperbolic systems of balance laws \eqref{e:noncons}.
\S\ref{sec2.3} follows with this spatial discretization applied directly to the SW system \eqref{e:sw_vec}.
The mPDeC temporal discretization is then presented for the production-destruction framework in \S\ref{sec3}. 
In \S\ref{sec4}, we follow with the construction of the unconditionally positive flux-globalization-based WB scheme using mPDeC. 
We present the numerical results in \S\ref{sec5}, and make our concluding remarks in \S\ref{sec6}.

\section{An overview of the WB A-WENO PCCU scheme}\label{sec2}
For general hyperbolic balance laws, we discretize in space following the WB A-WENO PCCU scheme. 
Recently introduced in \cite{Chu25WB5}, this scheme is a high-order extension of the second-order WB PCCU scheme, presented in \cite{Kurganov23WBPCCU}, which discretely preserves all 1-D steady states of systems in the form \eqref{e:noncons} via the flux-globalization technique. 
We provide an overview to the high-order WB A-WENO PCCU scheme in this section.

As shown in \cite{Kurganov23WBPCCU}, the steady state solutions of nonconservative systems \eqref{e:noncons} satisfy
\begin{equation} \label{e:ss}
	\bm{F}(\bb U)_x - B(\bm U) \bm U_x = M(\bm U) \bm E(\bm U)_x = \bm 0,
\end{equation}
where $M(\bm U) \in \R^{d \times d}$ and $\bm E$ denotes the vector of equilibrium variables, which are constant at steady states. 
To preserve these steady states, we follow \cite{Kurganov23WBPCCU} and apply the flux-globalization approach; introduced in, e.g., \cite{Caselles09Flux,Chertock18WB,Donat11Hybrid,Gascon01Const,Martinez11Hybrid}; in which they use the quasi-conservative equivalent of \eqref{e:noncons}:
\begin{equation} \label{e:quasicons}
	\bb{U}_t + \bb{K}(\bb U)_x = \bb 0,
\end{equation}
where $\bm K(\bm U)$ is the global flux defined by
\begin{equation} \label{e:k_r}
	\bb{K}(\bb U) = \bm F(\bm U) - \bm R(\bm U), \qquad
	\bm R (\bm U) = \int\limits_{\widehat x}^x \br{B(\bm U(\xi, t)) \bm U_\xi(\xi,t))}\ {\rm d}\xi,
\end{equation}
for arbitrary number $\widehat x$. 

\subsection{The second-order WB PCCU scheme}\label{sec2.1}

For a second-order method, \cite{Kurganov23WBPCCU} evolves the quasi-conservative form \eqref{e:quasicons} using a FV approach. 
We therefore introduce a uniform FV mesh $C_j : = [x_\jmh,x_\jph]$ of size $\dx \equiv x_\jph-x_\jmh$ with cell center $x_j = (x_\jph + x_\jmh)/2$ for $1 \leq j \leq N$.
Assume that at time level $t$, the approximate solution, realized in terms of its cell averages
\begin{equation*} \label{e:cellavg}
	\xbar{\bb U}_j(t) :\approx \frac{1}{\dx}\int\limits_{C_j} \bm U (x, t)\ {\rm d}x, 
\end{equation*}
is available. 
For the sake of brevity, we will suppress the  time-dependence of all indexed quantities for the remainder of the paper. 
The cell averages are then advanced in time by evolving the following semi-discretization:
\begin{equation*}\label{e:fv_semi_disc}
	\ddt{}\xbar{\bb U}_j = -\frac{1}{\dx}\br{\bb{\mc K}_{\jph}^{\rm FV} - \bb{\mc K}_{\jmh}^{\rm FV}}, 
\end{equation*}
where $\bb{\mc K}_{\jph}^{\rm FV}$, $0 \leq j \leq N$ denote the WB PCCU numerical fluxes, which read
\begin{equation} \label{e:fvcu_flux}
	\bb{\mc K}_{\jph}^{\rm FV} = \frac{a_\jph^+ \bb K(\bb U_j^{\rm E}) - a_\jph^- \bb K(\bb U_{j+1}^{\rm W})}{a_\jph^+ - a_\jph^-} 
	+ \frac{a_\jph^+  a_\jph^-}{a_\jph^+ - a_\jph^-}\br{\widehat{\bb U}_{j+1}^{\rm W} - \widehat{\bb U}_{j}^{\rm E}}.
\end{equation}
Here, $\bm U_j^{\rm E}$ and $\bm U_j^{\rm W}$ represent the cell interface values of cell $C_j$ at $x_\jph$ and $x_\jmh$, respectively, while $\widehat{\bm U}_j^{\rm E}$ and $\widehat{\bm U}_j^{\rm W}$ are similar quantities computed in a slightly different way that is needed to guarantee the preservation of equilibrium states.
Additionally, $a_{\jph}^\pm$ denote the one-sided local speeds of propagation, which can be approximated using the eigenvalues $\lambda_1 < \dots < \lambda_d$ of the Jacobian matrix $\mc{A}(\bm U) := \pd{\bm F}{\bm U} - B(\bm U)$: 
\begin{equation*}\label{e:wavespeeds}
	\begin{aligned}
		a_\jph^+ &= \max\set{\lambda_d(\mc{A}(\bb U_j^{\rm E})),\lambda_d(\mc{A}(\bb U_{j+1}^{\rm W})), 0}, \\
		a_\jph^- &= \min\set{\lambda_1(\mc{A}(\bb U_j^{\rm E})),\lambda_1(\mc{A}(\bb U_{j+1}^{\rm W})), 0};
	\end{aligned}
\end{equation*}
and $\bm K (\bm U)$ is again the global flux \eqref{e:k_r}. 
$\bm K(\bm U)$ must be evaluated with an appropriate quadrature due to the global integral $\bm R(\bm U)$, which (i) should maintain the WB property; and (ii) since the integrand of $\bm R$ is a nonconservative product, should employ the path conservative technique. 
Following \cite{Kurganov23WBPCCU}, we compute the global flux evaluations 
\begin{equation}\label{e:approxK}
	\bm K(\bm U_j^{\rm E,W}) = \bm F(\bm U_j^{\rm E,W}) - \bm R(\bm U_j^{\rm E,W})
\end{equation}
as follows.

To preserve the WB property, we must employ a sound discretization of the cell interface values $\bm U_j^{\rm E,W}$ and $\widehat{\bm U}_j^{\rm E,W}$, and of the quadrature for global integral $\bm R(\bm U_j^{\rm E,W})$. 
To this end, the cell interface values $\bm U_j^{\rm E,W}$ are computed by (i) converting to equilibrium variables $\bm E_j := \bm E(\xbar{\bm U}_j) $; (ii) applying an appropriate reconstruction to $\bm E_j$ to obtain the equilibrium cell interface values $\bm E_j^{\rm E,W}$; and (iii) obtaining the corresponding $\bm U_j^{\rm E,W}$ values by solving the (nonlinear system of) equations 
\begin{equation}\label{e:nonlin}
	\bm E(\bm U_j^{\rm E,W}) = \bm E_j^{\rm E, W},
\end{equation}
for $\bm U_j^{\rm E,W}$. 
For the second-order approach, \cite{Kurganov23WBPCCU} uses a standard piecewise linear reconstruction of $\bm E_j$ with slopes computed using the nonlinear generalized minmod limiter (see, e.g., \cite{Nessyahu90NonOsc,Sweby84High}). 
We refrain from presenting the second-order reconstruction details for the sake of brevity, as we present the high-order computation of these point values in \S \ref{sec2.2}. 

Once equipped with the point values $\bm U_j^{\rm E,W}$, we then compute the quadrature for $\bm R(\bm U_j^{\rm E,W})$. 
To this end, we follow \cite{Chu25WB5,Kurganov23WBPCCU} and select $\widehat{x} = x_{\frac{1}{2}}$ (the left domain boundary), set $\bm R(\bm U_0^{\rm E}) = \bm 0$, then evaluate recursively using
\begin{equation}\label{e:R_recurs}
	\bm R(\bm U_j^{\rm W}) = \bm R(\bm U_{j-1}^{\rm E}) + \bm B_{\bm \Psi, \jmh}, \qquad
	\bm R(\bm U_j^{\rm E}) = \bm R(\bm U_{j}^{\rm W}) + \bm B_{j}, 
\end{equation} 
for $1 \leq j \leq n$. 
Here, $\bm B_{\bm \Psi, \jph}$ and $\bm B_j$ are computed using appropriate quadratures for the integrals:
\begin{equation}\label{e:Bpsi_Bj}
	\bm B_{\bm \Psi, \jph} = \int\limits_0^1 B(\bm \Psi_\jph(s))
	\bm \Psi_\jph'(s) \ {\rm d}s, \qquad 
	\bm B_j = \int\limits_{C_j} B(\bm U) \bm U_x \ {\rm d}x, 
\end{equation}
where $\bm \Psi_\jph(s) = \bm \Psi_\jph(s;\bm U_{j}^{\rm E}, \bm U_{j+1}^{\rm W})$ represents a sufficiently smooth path connecting the left state $\bm U_{j}^{\rm E}$ and right state $\bm U_{j+1}^{\rm W}$ of the interface at $x_\jph$; that is, 
\begin{equation*}\label{e:path}
	\begin{aligned}
		&\bm \Psi_\jph : [0,1]\times \R^d\times \R^d \rightarrow \R^d, \\
		&\bm \Psi_\jph(0;\bm U_{j}^{\rm E}, \bm U_{j+1}^{\rm W}) = \bm U_{j}^{\rm E} \qquad 
		\bm \Psi_\jph(1;\bm U_{j}^{\rm E}, \bm U_{j+1}^{\rm W}) = \bm U_{j+1}^{\rm W}.
	\end{aligned}
\end{equation*}
To preserve the WB property of the scheme, we follow \cite{Kurganov23WBPCCU} and select a linear segment path that connects the left and right equilibrium states: 
\begin{equation*}\label{e:Epath}
	\bm E_\jph(s)
	 = \bm E_{j}^{\rm E} + s(\bm E_{j+1}^{\rm W} - \bm E_{j}^{\rm E}),
	 \qquad s\in [0,1].
\end{equation*}

To discretize the integrals in \eqref{e:Bpsi_Bj}, we substitute in the relation in \eqref{e:ss} for $B(\bm U) \bm U_x$, yielding
\begin{equation}\label{e:Bs_subbed}
	\begin{aligned}
		\bm B_{\bm \Psi, \jph} &= \bm F(\bm U_{j+1}^{\rm W}) - \bm F(\bm U_{j}^{\rm E}) - \int\limits_0^1 M(\bm \Psi_\jph(s))
		\bm E_\jph'(s) \ {\rm d}s, \\ 
		\bm B_j &= \bm F(\bm U_{j}^{\rm E}) - \bm F(\bm U_{j}^{\rm W}) -\int\limits_{C_j} M(\bm U) \bm E(\bm U)_x \ {\rm d}x.
	\end{aligned}
\end{equation}
We again refer the reader to \cite{Kurganov23WBPCCU} for further details on the second-order discretization, where they approximate these integrals using trapezoidal rule; we will discuss the higher-order approximation in \S\ref{sec2.2}. 

Finally, we obtain $\widehat{\bm U}_j^{\rm E,W}$ by solving a slight variation of the (nonlinear system of) equations in \eqref{e:nonlin}:
\begin{equation}\label{e:hatnl}
		\widehat{\bm E}(\widehat{\bm U}_j^{\rm E,W}) = \bm E_j^{\rm E, W}.
\end{equation}
These slight modifications $\widehat{\bm E}$ of the functions $\bm E$ are selected in way such that (i) $\widehat{\bm U}_j^{\rm E} = \widehat{\bm U}_{j+1}^{\rm W}$ as long as ${\bm E}_j^{\rm E} = {\bm E}_{j+1}^{\rm W}$, which is to ensure the preservation of equilibrium states; and (ii) $\widehat{\bm U}_j^{\rm E,W}\approx\bm U_j^{\rm E, W}$ to keep the selected order of accuracy for smooth solutions.

\subsection{The flux-globalization-based WB A-WENO PCCU scheme}\label{sec2.2}
We now present the high-order extension of the WB PCCU scheme introduced in \cite{Chu25WB5}, in which the vast majority of adjustments are made to the numerical fluxes \eqref{e:fvcu_flux}, the quadratures in \eqref{e:Bs_subbed}, and within the reconstruction. We present the fifth-order extension, but this approach can be extended to arbitrary order analogously.

In the high order A-WENO framework, we follow \cite{Chu25WB5} and instead now evolve point values:
\begin{equation}\label{e:fd_semi_disc}
	\ddt{}\bb U_j = -\frac{1}{\dx}\br{\bb{\mc K}_{\jph} - \bb{\mc K}_{\jmh}}, 
\end{equation}
where $\bm U_j \approx \bm U(x_j,t)$ and $\bb{\mc K}_{\jph}$ are the fifth-order A-WENO numerical fluxes:
\begin{equation}\label{e:aweno_fluxes}
	\bb{\mc K}_{\jph} = \bb{\mc K}_{\jph}^{\rm FV} 
	- \frac{(\dx)^2}{24} (\bb K_{xx})_\jph - \frac{7(\dx)^4}{5760} (\bb K_{xxxx})_\jph.
\end{equation}
Here, $\bb{\mc K}_{\jph}^{\rm FV} $ denotes the second-order WB PCCU numerical flux \eqref{e:fvcu_flux}, whose computation we presented in \S\ref{sec2.1}, and $(\bb K_{xx})_\jph \approx \bm K(\bm U(x_\jph, t))_{xx}$ and $(\bb K_{xxxx})_\jph \approx \bm K(\bm U(x_\jph, t))_{xxxx}$ are the higher order correction terms, computed using finite-difference approximations. 
We follow \cite{Chu25Efficient} and compute these correction terms using the second-order numerical fluxes:
\begin{equation} \label{e:corrections}
	\begin{aligned}
		(\bb K_{xx})_\jph & = \frac{1}{12 (\dx)^2} \br{-\bb{\mc K}_{j-\frac{3}{2}}^{\rm FV} +16\bb{\mc K}_{\jmh}^{\rm FV} -30 \bb{\mc K}_{\jph}^{\rm FV} +16 \bb{\mc K}_{j+\frac{3}{2}}^{\rm FV} - \bb{\mc K}_{j+\frac{5}{2}}^{\rm FV}},\\[4pt]
		(\bb K_{xxxx})_\jph & = \frac{1}{(\dx)^4} \br{\bb{\mc K}_{j-\frac{3}{2}}^{\rm FV} -4\bb{\mc K}_{\jmh}^{\rm FV} +6 \bb{\mc K}_{\jph}^{\rm FV} -4 \bb{\mc K}_{j+\frac{3}{2}}^{\rm FV} + \bb{\mc K}_{j+\frac{5}{2}}^{\rm FV}}.
	\end{aligned}
\end{equation}
Obtaining the correction terms this way allows one to use already-computed values, as opposed to the alternative, which would have to additionally evaluate the physical flux at cell centers. 

To ensure a fifth-order accurate scheme that preserves the WB property, we apply the affine-invariant WENO-Z interpolation; see, e.g., \cite{Don22SI_WENO,Li23SI_WENO,Wang22AI_WENO}, and a detailed review in Appendix \ref{secA}; to the equilibrium variables, yielding point values $\bm E_j^{\rm E,W}$ and $\bm E_{j\pm \frac{1}{4}}$. 
The quarter-values $\bm E_{j\pm \frac{1}{4}}$ will be needed in the quadrature step.
Furthermore, to reduce oscillations, we additionally use the local characteristic decomposition of the equilibrium variables, as recently introduced in \cite{chu24local}.
To this end, we first multiply the system \eqref{e:noncons} by $\pd{\bm E}{\bm U}$ to obtain an evolution equation for the equilibrium variables:
\begin{equation*}\label{e:eq_evol}
	\bm E(\bm U)_t + C(\bm U) \bm E(\bm U)_x = \bm 0,
\end{equation*}
where $C(\bm U) = \pd{\bm E}{\bm U} M(\bm U)$.
Then we evaluate matrix $C$ at the grid points $x_j$, thus yielding constant matrices $C_j := C(\bm U_j)$.
These can further be diagonalized using $C_j = Q_j \Lambda_j Q_j^{-1}$, where $\Lambda_j$ is the diagonal matrix of eigenvalues of $C_j$ and the columns of $Q_j$ are the corresponding eigenvectors. 
As a result, one can obtain the local characteristic equilibrium variables in the neighborhood of $x = x_j$:
\begin{equation}\label{e:lcv}
	\bb{\Gamma}_\ell = Q_j^{-1} \bm E_\ell, \qquad j-2 \leq \ell \leq j+2.
\end{equation} 
We then use the affine-invariant WENO-Z (presented in Appendix \ref{secA}), applied in a component-wise manner, to obtain $\bb{\Gamma}_j^{\rm E,W}$ and $\bb{\Gamma}_{j\pm \frac{1}{4}}$, and finally obtain the equilibrium variable point values using 
\begin{equation}\label{e:Epv}
	\bm E_j^{\rm E,W} = Q_j\bb{\Gamma}_j^{\rm E,W}, \qquad \bm E_{j\pm \frac{1}{4}} = Q_j\bb{\Gamma}_{j\pm \frac{1}{4}}. 
\end{equation}

Equipped now with these point values, we can obtain $\bm U_j^{\rm E,W}$ and $\bm U_{j\pm \frac{1}{4}}$ by solving the nonlinear system in \eqref{e:nonlin}, in turn providing everything needed to approximate the integrals in \eqref{e:Bs_subbed}. 
For the integral connecting the left and right states, we use trapezoidal quadrature:
\begin{equation}\label{e:Bpsi_HO}
	\bm B_{\bm \Psi, \jph} = \bm F(\bm U_{j+1}^{\rm W}) - \bm F(\bm U_{j}^{\rm E}) - \frac{1}{2}\br{M(\bm U_{j+1}^{\rm W}) + M(\bm U_{j}^{\rm E})} \p{\bm E_{j+1}^{\rm W}-\bm E_{j}^{\rm E}}. 
\end{equation}
For the integral across the ``cell", we follow, e.g., \cite{Chu22Fifth}, and compute via five-point Newton-Cotes approach. 
That is, to compute $\bm B_j$, we (i) construct a fourth-degree polynomial interpolation of $M(\bm U)$ and $\bm E$ using the obtained point values $\bm U_j$, $\bm U_j^{\rm E,W}$, $\bm U_{j\pm \frac{1}{4}}$ and $\bm E_j$, $\bm E_j^{\rm E,W}$, $\bm E_{j\pm \frac{1}{4}}$, respectively; (ii) obtain $\bm E_x$ by differentiating the polynomial interpolation; and (iii) substitute these into \eqref{e:Bs_subbed} and integrate exactly to obtain
\begin{equation}\label{e:Bj_HO}
	\begin{aligned}
		\bm B_j \approx& \bm F(\bm U_j^{\rm E}) - \bm F(\bm U_j^{\rm W})\\
		 &- \frac{1}{1890}\Bigg(
		107 \br{M(\bm U_j^{\rm E}) \bm E_j^{\rm W} - M(\bm U_j^{\rm W}) \bm E_j^{\rm E}} \\
		& \hspace{1.4cm} +2112\set{\br{M(\bm U_{\jmf}) - M(\bm U_\jpf)} \bm E_j+ \br{\bm E_\jpf - \bm E_\jmf}M(\bm U_j)}\\
		&\hspace{1.4cm} +804 \set{\br{M(\bm U_j^{\rm E}) - M(\bm U_j^{\rm W})} \bm E_j+ \br{\bm E_j^{\rm W} - \bm E_j^{\rm E}}M(\bm U_j)}\\
		&\hspace{1.4cm}+1024\br{M(\bm U_\jpf)\bm E_\jmf - M(\bm U_\jmf)E_\jpf}
		+945 \br{M(\bm U_j^{\rm E})\bm E_j^{\rm E} - M(\bm U_j^{\rm W})\bm E_j^{\rm W}} \\
		& \hspace{1.4cm} + 1472\br{M(\bm U_\jpf)\bm E_j^{\rm E} - M(\bm U_j^{\rm E})\bm E_\jpf + M(\bm U_j^{\rm W})\bm E_\jmf - M(\bm U_\jmf)\bm E_j^{\rm W}} \\
		& \hspace{1.4cm} + 384\br{M(\bm U_\jmf)\bm E_j^{\rm E} - M(\bm U_j^{\rm E})\bm E_\jmf + M(\bm U_j^{\rm W})\bm E_\jpf - M(\bm U_\jpf)\bm E_j^{\rm W}}\Bigg).
	\end{aligned}
\end{equation}
The last step of the high-order approach is the same as the second-order method, where we compute $\widehat{\bm U}_j^{\rm E,W}$ from solving the slightly-adjusted nonlinear system \eqref{e:hatnl}.
After obtaining these values, we can then compute the FV numerical fluxes in \eqref{e:fvcu_flux}, in turn giving the correction terms in the A-WENO flux \eqref{e:corrections}.

\begin{rmk}
	Both the second-order and fifth-order schemes presented in this section are NOT positivity preserving. 
	We discuss the adjustments to ensure positivity spatially in \S\ref{sec4}.
\end{rmk}

\subsection{Application to the shallow water system}\label{sec2.3}
We now present the fifth-order A-WENO WB PCCU scheme applied to the SW equations, noting again that the positivity-preserving adjustments will be presented in \S\ref{sec4}. 
The 1-D Saint-Venant system 
\eqref{e:sw_vec} admits steady state solutions satisfying \eqref{e:ss} with
\begin{equation}\label{e:MU_E}
	M(\bm U) = \begin{pmatrix}
		1 & 0 & 0\\
		u & h & 0 \\ 
		0 & 0 & 1
	\end{pmatrix},\qquad
	\bm E(\bm U) = \begin{pmatrix}
		q\\ \mathcal{E} \\ 0
	\end{pmatrix}, \qquad
	\mathcal{E} = \frac{q^2}{2h^2} + g (h+Z).
\end{equation} 
First, the system \eqref{e:noncons} with \eqref{e:sw_vec} is put into quasi-conservative form using the setup presented in \eqref{e:k_r}:
\begin{equation*}\label{e:quasi_sw}
	\bm K(\bm U) = \bm F(\bm U) - \bm R(\bm U), \qquad
	\bm R(\bm U) = (0, R)^\top, \qquad
	R:= -\int\limits_{\widehat x}^x gh(\xi, t) Z(\xi)_\xi\ {\rm d}\xi.
\end{equation*}
At time $t$, assume that we have the numerical solution realized at point values $\bm U_j$.
To obtain a WB scheme for this system, we first convert from conservative variables $\bm U$ to equilibrium variables $\bm E(\bm U)$ in \eqref{e:MU_E}; that is, we compute
\begin{equation}\label{e:Ej}
	\mathcal{E}_j : = \frac{1}{2}\p{\frac{q_j}{h_j}}^2 + g (h_j + Z_j).
\end{equation}
Now having the values $q_j$ and $\mathcal{E}_j$, we then convert to the local characteristic equilibrium variables. 
Notice that since the last entry of $\bm E_j = 0$, the conversion  only needs to be done for $q_j$ and $\mathcal{E}_j$.
Thus, consider the  reduced version of \eqref{e:MU_E}:
\begin{equation*}\label{e:M_Ereduced}
	\widetilde{\bm U} = \begin{pmatrix}
		h\\ q 
	\end{pmatrix}, \qquad 
	\widetilde{M}(\widetilde{\bm U}) = \begin{pmatrix}
		1 & 0 \\
		u & h  
	\end{pmatrix},\qquad
	\widetilde{\bm E}(\widetilde{\bm U}) = \begin{pmatrix}
		q\\ \mathcal{E} 
	\end{pmatrix},
\end{equation*} 
where $\mathcal{E}$ is as defined in \eqref{e:MU_E} and $u = q/h$.
We can then obtain the constant matrix $C_j = \pd{\widetilde{\bm E}}{\widetilde{\bm U}} \widetilde{M}(\widetilde{\bm U})\big|_{x = x_j}$ and its corresponding eigenvector matrices:
\begin{equation}\label{e:CQ}
	C_j = \begin{pmatrix}
		u_j & h_j \\
		g & u_j  
	\end{pmatrix},\qquad
	Q_j = \begin{pmatrix}
		\sqrt{h_j} & \sqrt{h_j} \\
		-\sqrt{g} & \sqrt{g} 
	\end{pmatrix},\qquad
	Q_j^{-1} = \frac{1}{2\sqrt{gh_j}}
	\begin{pmatrix}
		\sqrt{g} & -\sqrt{h_j} \\
		\sqrt{g} & \sqrt{h_j} 
	\end{pmatrix}.
\end{equation}
We then (i) use the appropriate eigenvector matrix to obtain the local-characteristic equilibrium variables $\bb \Gamma_j$ at point $x_j$, as shown in \eqref{e:lcv}; 
(ii) interpolate via the affine-invariant WENO-Z scheme (see Appendix A) to obtain the point values $\bb \Gamma_j^{\rm E,W}$ and $\bb \Gamma_{j \pm \frac{1}{4}}$; and (iii) convert back to equilibrium variables using \eqref{e:Epv}, yielding the point values $(\bm E_1)_j^{\rm E,W} = (q_j^{\rm E,W}, \mc E_j^{\rm E,W})^\top$ and $(\bm E_1)_{j \pm \frac{1}{4}} = (q_{j \pm \frac{1}{4}},\mc E_{j \pm \frac{1}{4}})^\top$. Note also that, at the initial time step, the bottom topography $Z_j$ is also interpolated using the affine-invariant WENO-Z approach.

To obtain the conservative variables such that they are consistent with this WB reconstruction, we must solve for $h$ from the newly obtained values $\mc E$ at $x_{j\pm \frac{1}{2}}$ and $x_{j\pm \frac{1}{4}}$. 
For the left interface at $x_\jph$, this is done by solving the nonlinear equation 
\begin{equation}\label{e:EjE}
	\mc E_j^{\rm E} = \frac{1}{2}\p{\frac{q_j^{\rm E}}{h_j^{\rm E}}}^2 + g (h_j^{\rm E} + Z_j^{\rm E}),
\end{equation}
for $h_j^{\rm E}$---we do this using a slight adjustment of the cubic solve presented in, e.g., \cite{Cheng19ME}. 
Computing $h_j^{\rm W}$, $h_\jmf$ and $h_\jpf$ are done analogously; we thus refrain from presenting these nonlinear equations for the sake of brevity. 

Once equipped with the values $h_j^{\rm E,W}$ and $h_{j \pm \frac{1}{4}}$, we compute the global fluxes using \eqref{e:approxK}.
To do so, we need first the quadrature for $\bm R(\bm U_j^{\rm E,W})$; we thus use the recursive relation in \eqref{e:R_recurs} combined with the discrete definitions of $\bm B_{\bb \Psi,\jph}$ and $\bm B_j$ in \eqref{e:Bpsi_HO} and \eqref{e:Bj_HO}, respectively. 

For the last missing piece in the FV numerical fluxes \eqref{e:fvcu_flux}, we must obtain $\widehat{\bm U}_j^{\rm E,W}$ from \eqref{e:hatnl}.
For the SW equations, we do this by taking $\widehat q_j^{\rm E,W} = q_j^{\rm E,W}$, and find $\widehat h_j^{\rm E}$ and $\widehat h_{j+1}^{\rm W}$ by solving the following nonlinear equations:
\begin{equation*}\label{e:Ejhat}
	\mathcal{E}_j^{\rm E} = \frac{1}{2}\p{\frac{q_j^{\rm E}}{\widehat h_j^{\rm E}}}^2 + g (\widehat h_j^{\rm E} + \widehat Z_\jph),\qquad 
	\mathcal{E}_{j+1}^{\rm W} = \frac{1}{2}\p{\frac{q_{j+1}^{\rm W}}{\widehat h_{j+1}^{\rm W}}}^2 + g (\widehat h_{j+1}^{\rm W} + \widehat Z_\jph),
\end{equation*}
where the only difference within these nonlinear equations and \eqref{e:EjE} is $\widehat Z_\jph = \frac{1}{2}(Z_j^{\rm E} + Z_{j+1}^{\rm W})$, which is used to ensure the desired property that $\widehat h_{j+1}^{\rm W} = \widehat h_{j}^{\rm E}$ when $ \mathcal{E}_{j+1}^{\rm W} = \mathcal{E}_{j}^{\rm E}$. 

Finally, we can compute the FV numerical fluxes \eqref{e:fvcu_flux}, and in turn the fifth-order accurate A-WENO numerical fluxes \eqref{e:aweno_fluxes}. 
Hence, we have the complete semi-discrete scheme, and now need to select an appropriate time evolution.

\section{Modified Patankar DeC method}\label{sec3}
The \textit{modified Patankar Deferred Correction (mPDeC)} method, introduced in~\cite{torlo2020}, is a high-order, conservative, and unconditionally positivity-preserving time integration scheme tailored for production-destruction systems (PDS). 
It extends the classical deferred correction (DeC) framework \cite{dutt2000,ong2020,veiga2021} by incorporating a dynamic correction mechanism inspired by the modified Patankar trick \cite{meister2019,burchard2003}, ensuring robustness and physical consistency, which is especially critical in the numerical simulation of hyperbolic conservation laws such as the SW equations \cite{ciallella2025,bender2025, arXiv:2602.14392, ciallella2022}.
In this section, we briefly review the time evolution DeC and mPDeC schemes, as both are implemented for the time evolution. 

\subsection{Deferred correction scheme}\label{sec3.1}
To present the DeC time evolution, we follow the framework presented in \cite{Abgrall17HO,micalizzi2024efficient}, since it presents the DeC approach using a compact operator notation.
The DeC time evolution is based on the Picard iteration to minimize error, and mimics this process at the discrete level to improve accuracy over several iterative steps. 

Consider the following time-dependent initial value problem:
\begin{equation}\label{e:ODE}
	\ddt{}\bm c(t) = \bm f(\bm c(t)), \qquad \bm c(0) = \bm c_0,
\end{equation}
where $\bm c(t):\R\rightarrow\R^S$ and $\bm f: \R^S \rightarrow\R^S$  result from the semi-discretization \eqref{e:fd_semi_disc}.
The DeC approach can be written as a combination of two operators: (i) a low-order ``easy-to-solve" time evolution operator $\mc L^1$, such as, e.g., explicit Euler; and (ii) a high-order operator $\mc L^2$ that may present practical difficulties, such as, e.g., an implicit RK method.

Let us focus on one timestep $[t_n,t_{n+1}]$, and let $\{t^\ell\}_{\ell=0}^r$ be $r+1$ nodes subdividing the time step such that $t^\ell = t_n + \eta^\ell \dt$, with $0 = \eta^1 <  \cdots  < \eta^r = 1$ and $0 \leq \ell \leq r$.
In this work, we choose $\eta^\ell$ to be the Gauss-Lobatto subnodes; see, e.g., Figure \ref{fig:time-substeps-gl}; as the order of the $\mc L^2$ operator depends on the subinterval distribution---in selecting these nodes, the $\mc L^2$ order of accuracy is $\mc O(\dt^{2r})$.
Then, on each subinterval $[t^0, t^\ell]$, we mimic the Picard iteration; for the $\mc L^1$ operator, this is done using a left Riemann sum approximation (i.e., the forward Euler discretization), and for the $\mc L^2$ operator the integral is approximated using an appropriate quadrature:
\begin{equation}\label{e:L1_L2}
	\mc L^1(\underline{\bb c}) = \begin{pmatrix}
		\bb c^1 - \bb c^0 - \eta^1\dt\bm f(\bb c^0) \\
		\vdots\\
		\bb c^r - \bb c^0 - \eta^r\dt\bm f(\bb c^0)
	\end{pmatrix},\qquad
	\mc L^2(\underline{\bb c}) = \begin{pmatrix}
		\displaystyle \bb c^1 - \bb c^0 - \dt\sum_{k = 0}^r \theta_{k, 1} \bm f(\bb c^k) \\
		\vdots\\
		\displaystyle \bb c^r - \bb c^0 - \dt\sum_{k = 0}^r \theta_{k, r}  \bm f(\bb c^k)
	\end{pmatrix},
\end{equation}
where $\underline{\bm c} = (\bm c^0, \dots, \bm c^r)$ is used to abbreviate the dependence of the two operators on all different subinterval solutions $\bm c^\ell$, and $\theta_{k, \ell}$ denotes the high-order quadrature weights on the subinterval $[t^0, t^\ell]$.

\begin{figure}[t]
	\centering
	\begin{tikzpicture}[>=stealth]
		
		\draw[->, thick] (0,0) -- (10.5,0)
		node[right] {$t$};
		
		\foreach \x/\label in {
			0.8/$t^0=t_n$,
			2.05/$t^1$,
			4.15/$t^2$,
			6.85/$t^3$,
			9.7/$t^r=t_{n+1}$
		}{
			\fill (\x,0) circle (2pt);
			\draw (\x,0.12) -- (\x,-0.12);
			\node[below=5pt] at (\x,0) {\label};
		}
		
		\draw[->] (0.8,-0.75) -- (9.7,-0.75);
		\node[below] at (5.25,-0.75) {$\Delta t$};
		
		\draw[->] (0.8,0.65) -- (2.05,0.65);
		\draw[->] (2.05,0.65) -- (4.15,0.65);
		\draw[->] (4.15,0.65) -- (6.85,0.65);
		\draw[->] (6.85,0.65) -- (9.7,0.65);
		
		\node[above=9pt] at (0.8,0.65) {$\eta^0=0$};
		\node[above=9pt] at (2.05,0.65) {$\eta^1$};
		\node[above=9pt] at (4.15,0.65) {$\eta^2$};
		\node[above=9pt] at (6.85,0.65) {$\eta^3$};
		\node[above=9pt] at (9.7,0.65) {$\eta^r=1$};
		
	\end{tikzpicture}
	\caption{ \sf
		Subdivision of the time step $[t_n,t_{n+1}]$ using
		non-equidistant Gauss-Lobatto nodes.
	}
	\label{fig:time-substeps-gl}
\end{figure}
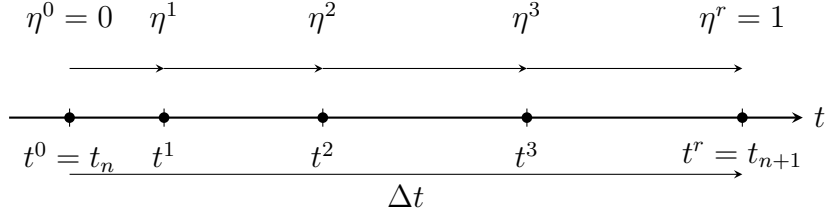

The aim of the DeC algorithm is then to iteratively approximate the numerical solution $\underline{\bm c}^*$ of the $\mc L^2(\underline{\bm c}^*) = \bm 0$  scheme without using the inherently nonlinear implicit evolution within the $\mc L^2$ operator. 
Let the successive iterations be denoted by $0 \leq m \leq M$, with $M$ is the total number of iterations. 
Then the DeC algorithm is defined as follows:\\
\textbf{DeC Algorithm}
\begin{enumerate}
	\item Set $\bm c^{0,(m)} : = \bm c(t_n)$ for $ 0 \leq m \leq M,\quad $ and $\quad \bm c^{\ell,(0)} : = \bm c(t_n) $ for $ 0 \leq \ell \leq r.$
	\item Iterate: $\displaystyle \mc L^1( \underline{\bm c}^{(m)}) = \mc L^1(\underline{\bm c}^{(m-1)}) - \mc L^2(\underline{\bm c}^{(m-1)}) $ for $1 \leq m \leq M.$ 
	\item Set $\bm c^{n+1} = \bm c^{r, (M)}$.
\end{enumerate}
The DeC algorithm described maintains exact conservation and is arbitrarily high-order accurate (with order $\mc O(\dt^{\min\set{M+1, 2r}})$; where $M$ is the number of iterations and $r+1$ is the number of Gauss-Lobatto sub-nodes), but is not inherently positivity-preserving when evolving in time. 
Thus, we require an appropriate adjustment to ensure positivity of select quantities, such as the fluid height $h$ in the SW system \eqref{e:sw_vec}.

\subsection{The modified Patankar adjustment to DeC}
In the case the initial value problem \eqref{e:ODE} can be written as a PDS---that is, if $\bm f(\bm c)$ can be deconstructed in a way such that 
\begin{equation}\label{e:pds_rhs}
	f_\alpha (\bm c) = \sum_{\beta = 1}^S \p{p_{\alpha,\beta}(\bm c) - d_{\alpha,\beta}(\bm c)},
\end{equation}
where $f_\alpha$ is the $\alpha$-th component of $\bm f$, and $p_{\alpha,\beta},\ d_{\alpha,\beta} \geq 0$ are the production and destruction terms, respectively from component $c_\beta$ to $c_\alpha$---then one can employ the modified Patankar time evolution approach. 
Initial value problems of type \eqref{e:ODE} with \eqref{e:pds_rhs} describe many applications, such as chemical or biological reactions, see, e.g., \cite{burchard2003}; but the semi-discretization of hyperbolic conservation/balance laws can also be interpreted as a PDS, see, e.g., \cite{Huang19PP,Huang19third,Meister14PP}.
Given the ODE contexts of PDS in chemical and biological reactions, it was vital to ensure that mass was conserved (if the PDS is conservative: $p_{\alpha,\beta}(\bm c) = d_{\beta,\alpha}(\bm c)$ for all $\alpha, \beta, \bm c$) and the system remained unconditionally positive for all time. 
The modified Patankar (mP) approach addresses both these issues simultaneously. 

The idea behind mP is as follows. 
From literature (see, e.g., \cite{burchard2003}), it is well known that the explicit Euler time evolution is conservative, but positivity is conditional on the time step restriction. 
However, the implicit Euler method is both conservative and positivity-preserving without dependence on a time step restriction (i.e., \textit{unconditionally positive})---but then the issue becomes about solving a nonlinear system, which is computationally costly. 
Patankar addressed this cost issue and developed an unconditionally positive numerical scheme for PDS systems by first employing the so-called Patankar weights to the destruction terms in the explicit Euler discretization \cite{Patankar80Numerical}: 
\begin{equation*}\label{e:Patankar}
	c_\alpha^{n+1} = c_\alpha^n + \dt \p{\sum_{\beta = 1}^S p_{\alpha,\beta}(\bb c^n) - \sum_{\beta = 1}^S d_{\alpha,\beta}(\bb c^n)\frac{c_\alpha^{n+1}}{c_\alpha^n}},\qquad
	1 \leq \alpha \leq S.
\end{equation*}
This approach is indeed unconditionally positive and now computationally cost-efficient given that it is linearly implicit, but it now violates the conservative property.
Leveraging this foundation, Burchard et al. \cite{burchard2003} further extended the Patankar approach to then be conservative by additionally weighting the production terms:
\begin{equation}\label{e:mP}
	c_\alpha^{n+1} = c_\alpha^n + \dt \p{\sum_{\beta = 1}^S p_{\alpha,\beta}(\bb c^n)\frac{c_\beta^{n+1}}{c_\beta^n} - \sum_{\beta = 1}^S d_{\alpha,\beta}(\bb c^n)\frac{c_\alpha^{n+1}}{c_\alpha^n}}\qquad
	1 \leq \alpha \leq S,
\end{equation}
and hence was named the mP scheme.
Their proposed method provably maintained the unconditionally positive property, ensured conservation as long as the PDS was conservative, and remains in a cost-efficient linearly implicit discretization---however, the method is first-order. 

The mP scheme has subsequently been extended to higher order, of which the mPDeC approach---introduced in \cite{torlo2020}---is \textit{arbitrary} order of accuracy. 
The mPDeC approach follows the same algorithm as DeC, except the mP weights are additionally applied to $\mc L^2$ in \eqref{e:L1_L2}; thus, the mPDeC algorithm is the same as DeC, except step 2 now iterates over the linearly implicit setup
\begin{equation*}\label{e:iterate_mP}
	\mc L^1( \underline{\bm c}^{(m)}) = \mc L^1(\underline{\bm c}^{(m-1)}) - \mc L^2(\underline{\bm c}^{(m-1)},\underline{\bm c}^{(m)}),\qquad 1 \leq m \leq M.
\end{equation*}
Equivalently, this iteration step of the mPDeC algorithm can be rewritten as 
\begin{equation*}\label{e:mPDeC}
	c^{\ell,(m)}_\alpha = c^n_\alpha + \Delta t \sum_{k=0}^r \theta_{k,\ell} \br{
	\sum_{\beta = 1}^S \p{
	p_{\alpha,\beta}(\bm c^{k,(m-1)}) \cdot \frac{c^{\ell,(m)}_{\gamma(\beta,\alpha,\theta_{k,\ell})}}{c^{\ell,(m-1)}_{\gamma(\beta,\alpha,\theta_{k,\ell})}} \\
	 - d_{\alpha,\beta}(\bm c^{k,(m-1)}) \cdot \frac{c^{\ell,(m)}_{\gamma(\alpha,\beta,\theta_{\ell,m})}}{c^{\ell,(m-1)}_{\gamma(\alpha,\beta,\theta_{\ell,m})}}
	}}, 
\end{equation*}
for all $1 \leq \ell \leq r$,\ $1 \leq m \leq M$ and $1 \leq \alpha \leq S$, where $\gamma(\alpha, \beta, \theta) = \alpha $ if $\theta \geq 0$ and $\gamma(\alpha, \beta, \theta) = \beta $ if $\theta < 0$.
In other words, \eqref{e:mPDeC} is the linearly implicit system $\mathbb M \bm c^{\ell, (m)} = \bm c^n$, where the mass matrix $\mathbb M$ reads
\begin{equation}\label{e:massmat}
	\mathbb M(\bm c^{\ell, (m-1)}) = 
	\begin{cases}
		\displaystyle 1 + \dt \sum_{k = 0}^r \sum_{\beta = 0}^S \frac{\theta_{k,\ell}}{c^{\ell,(m-1)}_\alpha} \br{
		d_{\alpha,\beta}(\bm c^{k,(m-1)}) \mathbbold{1}_{\{\theta_{k,\ell} > 0\}} - p_{\alpha,\beta}(\bm c^{k,(m-1)}) \mathbbold{1}_{\{\theta_{k,\ell} < 0\}}
		}, & \textrm{for } \alpha = \beta, \\
		\displaystyle -\dt \sum_{k = 0}^r  \frac{\theta_{k,\ell}}{c^{\ell,(m-1)}_\beta} \br{
			p_{\alpha,\beta}(\bm c^{k,(m-1)}) \mathbbold{1}_{\{\theta_{k,\ell} > 0\}} - d_{\alpha,\beta}(\bm c^{k,(m-1)}) \mathbbold{1}_{\{\theta_{k,\ell} < 0\}}
		}, & \textrm{for } \alpha \neq \beta,
	\end{cases}
\end{equation}
where $\mathbbold 1$ is the indicator function.

The final time update mPDeC is again taken using $\bm c^{n+1} = \bm c^{r,(M)}$. 
This approach is arbitrarily high-order accurate with order $\mc O(\dt^{\min\set{M+1, 2r}})$, linearly implicit and thus computationally efficient, preserves conservation of the PDS, and guarantees unconditional positivity.

\begin{rmk}(Division by a small $c_\alpha$) 
	In the case that $c_\alpha$ is zero or near zero, the divisions within the mass matrix in \eqref{e:massmat} cause it to become ill-conditioned. 
	Hence, to avoid this division by a small quantity, we follow the suggestion of, e.g., \cite{torlo2020} and apply a desingularization; for numerator $n$ and denominator $d$ we take
	
	$$
	\frac{n}{d} \approx
	\begin{cases}
		0, & \text{if } d < 10^{-8}, \\
		\displaystyle \frac{2d \cdot n}{d^2 + \max\{d^2, 10^{-8}\}}, & \text{otherwise}.
	\end{cases}
	$$
	This ensures numerical stability while preserving the physical behavior; see, e.g., \cite{ciallella2022} for more details.
\end{rmk}

\section{The unconditionally positive WB scheme using mPDeC}\label{sec4}

We now present how to ensure positivity, while keeping the WB property, of the semi-discretization presented in \S\ref{sec2}. 
One requirement to combine the semi-discrete form with the unconditionally positive mPDeC time evolution is to put the semi-discretization in PDS form \eqref{e:pds_rhs}. 
In systems with sources, such as that in \eqref{e:noncons}, it is unclear whether we combine the source with one of the production or destruction terms of mPDeC, or treat it as its own term. 

However, the combination with a flux-globalization-based approach allows us to follow previous work directly. 
Previous work, such as that in, e.g., \cite{ciallella2022, ciallella2025}, present one way to set up the production and destruction matrices within \eqref{e:pds_rhs} for conservation laws. 
Since flux-globalization techniques (including that we present in the semi-discretization \eqref{e:fd_semi_disc}) use the quasi-conservative form, we can directly follow the previous setup of these matrices, now just on the global numerical fluxes. 
Let $P(\bm U),\ D(\bm U) \in \R^{N\times N}$ denote the production and destruction matrices, respectively, with corresponding positive entries $p_{\alpha,\beta}$ and $d_{\alpha,\beta}$, $1 \leq \alpha, \beta \leq S$, which in this usage, $S = N$. 
Then for the semi-discretization in \eqref{e:fd_semi_disc}, we take
\begin{equation}\label{e:pd_entries}
	\begin{aligned}
		&p_{j,j-1} = \frac{1}{\dx} \bb{\mc K}^{+}_\jmh, 
		\qquad p_{j,j+1} = \frac{1}{\dx} \bb{\mc K}^{-}_\jph,\\
		&d_{j,j-1} = \frac{1}{\dx} \bb{\mc K}^{-}_\jmh, 
		\qquad d_{j,j+1} = \frac{1}{\dx} \bb{\mc K}^{+}_\jph,
	\end{aligned}
\end{equation}
where $\bb{\mc K}^{+}_\jph = \max\{\bb{\mc K}_\jph,0\}$ and $\bb{\mc K}^{-}_\jph = -\min\{\bb{\mc K}_\jph,0\}$, and $\bb{\mc K}_\jph$ is the global flux discretization defined as in \eqref{e:aweno_fluxes}. 
The set up \eqref{e:pd_entries} meets all requirements for unconditional positivity---as long as the quantity of interest is positive at time $t^n$.

\begin{rmk}[Practical aspects for shallow water] 
	For the SW system \eqref{e:sw_vec}, we only apply mPDeC to the $h$ evolution.
	It is physically realistic that the quantity $q$ becomes negative from positive values; thus, $q$ is evolved numerically with the classical DeC.
\end{rmk}

\begin{rmk}[Linear solver]
	In one spatial dimension, \eqref{e:pd_entries} results in the mass matrix $\mathbb M$ in equation \eqref{e:massmat} being a tri-diagonal matrix, barring boundary conditions. 
	In this work, we consider zero-extrapolation, Dirichlet, and periodic boundary conditions; of which, only periodic conditions result in a non-tri-diagonal matrix. 
	In the case where $\mathbb M$ is tri-diagonal, we solve the system with Thomas algorithm.
	In the periodic case, the resulting mass matrix is tri-diagonal with the off-diagonal corners being non-zero.
	In this case, we use a Sherman-Morrison decomposition of the matrix $\mathbb M$ to obtain the solution to $\mathbb M \bm c^{\ell, (m)} = \bm c^n$ explicitly via two separate tri-diagonal solves. 
\end{rmk}

To ensure positivity of water height $h$ in the SW system \eqref{e:sw_vec}, we must ensure the construction of point values $h_j^{\rm E,W},$ $\widehat{h}_j^{\rm E,W}$, and $h_{j\pm \frac{1}{4}}$ remains positive and does not result in possible numerical blow up for values near zero. 
For point values $h_j^{\rm E,W}$ and $h_{j\pm \frac{1}{4}}$, we ensure positivity through the Zhang-Shu bound-preserving limiter seen in, e.g., \cite{zhang2011maximum}. 
Note that we apply this limiter to conservative variable point values $h_j$, $h_j^{\rm E,W},$ and $h_{j\pm \frac{1}{4}}$ instead of those for the equilibrium variable $E$ in \eqref{e:Ej}.
We do this since $h_j$, $h_j^{\rm E,W},$ and $h_{j\pm \frac{1}{4}}$ only arise in the discretization of the $M(\bm U)$ pieces of the $\bm B_{\bm \Psi, \jph}$ and $\bm B_j$ computations in \eqref{e:Bpsi_HO} and \eqref{e:Bj_HO}, respectively. 
Moreover, the terms dependent on $M(\bm U)$ vanish in the case $\bm E_j = \bm E_j^{\rm E,W} = \bm E_{j \pm \frac{1}{4}}$, thus implying that adjusting the $h$ point values via the Zhang-Shu limiter does not destroy the WB property of the scheme. 
We additionally do not have to worry about a moving equilibria in combination with dry regions, as they cannot occur simultaneously.

The same cannot be said for point values $\widehat{h}_j^{\rm E,W}$, as applying the Zhang-Shu limiter here will indeed destroy the preservation of steady states. 
This is due to the desired property $\widehat h_{j+1}^{\rm W} = \widehat h_{j}^{\rm E}$ when $\mc E_{j+1}^{\rm W} = \mc E_{j}^{\rm E}$, which is needed to ensure the WB property of the scheme. 
Thus, in the case of a dry region, we must make a different adjustment in $\widehat{h}_j^{\rm E,W}$ computation. 
In this work, we simply set a lower bound that neither $\widehat h_{j+1}^{\rm W} $ or $ \widehat h_{j}^{\rm E}$ can exceed, and if the lower threshold is exceeded we adjust both; that is, 
\begin{equation*}
	h_{j}^{\rm E},\ h_{j+1}^{\rm W} = 
	\begin{cases}
		10^{-16}, & \textrm{if } \min\{h_{j}^{\rm E},h_{j+1}^{\rm W}\} \leq 10^{-15}, \\
		h_{j}^{\rm E},\ h_{j+1}^{\rm W},  & \textrm{otherwise.}
	\end{cases}
\end{equation*}
Since we adjust both point values $h_{j}^{\rm E}, $ and $h_{j+1}^{\rm W}$ simultaneously in the case of either is negative, it does not destroy conservation, as we can see this adjustment would result in a cancellation of the second term of the FV numerical flux \eqref{e:fvcu_flux}.

We avoid numerical blow up for values near zero in three necessary ways. 
The first of which was described in Remark 3.1, where we correct for small $h$ values in the mass matrix via a desingularization. 
The second follows similarly: throughout the algorithm, we require the computation of $u = q/h$ at varying points---we therefore desingularize using the approach seen in, e.g., \cite{Kurganov07second}:
\begin{equation*}
	u = \frac{\sqrt{2} hq}{\sqrt{h^4 + \max\{h^4,\eps\}}},
\end{equation*}
where all possible subscripts and superscripts are omitted for the sake of simplicity, and in this work we take $\eps = \min\{10^{-8}, \dx^4\}$.
We refer the reader to \cite{Kurganov07second} for the advantages of using this different regularization for the computation of velocity.

The final algorithmic adjustment we make is to the local characteristic decomposition in \S \ref{sec2.3} (see also \cite{chu24local}). 
The adjustment for this must allow a stable interpolation of quantities consistent with the equilibrium variables so that the scheme remains WB. 
Furthermore, this is not as simple as applying a desingularization, as equation \eqref{e:lcv} suggests $Q_j^{-1}$ in \eqref{e:CQ} multiplies $\bm E_\ell$ for $j-2 \leq \ell \leq j+2$.
This implies that if $h_j$ is small, while $\bm E_\ell$ is within a wet region, there would be a substantial large values within the decomposition. 
Therefore, we forego the local characteristic decomposition in dry regions altogether; that is, if 
$$
\min_{j-2 \leq \ell \leq j+2}{h_\ell} < \eps_1, 
$$
then we use the component-wise A-WENO interpolation on $\bm E_j$, as opposed to the local characteristic equilibrium variables $\bb \Gamma_j = Q_j^{-1}\bm E_j$. 
This approach is still consistent with the equilibrium variables, and thus additionally WB.
Note in this work, we take $\eps_1 = 10^{-3}$. 
Having these adjustments to ensure positivity while remaining WB, we complete the description of the numerical scheme presented in this study.

\begin{rmk}[Zhang-Shu limiter time-step restriction]
	Although we use the Zhang-Shu limiter to spatially ensure positivity, the use of the linearly implicit mPDeC for the time evolution of $h$ allows us to forego its additional strong CFL restriction. In practice, we run the simulations with classical explicit CFL due to the stability of the underlying FD method. 
\end{rmk}

We now present the main the theoretical preservation of the WB property in the proposed scheme. 
The proof for positivity follows analogously to that presented in, e.g, \cite{torlo2020}.

\begin{thm}[WB property of the mPDeC scheme]
	Given a well-balanced semi-discretization of \eqref{e:fd_semi_disc}, the mPDeC scheme with the semi-discretization in \eqref{e:fd_semi_disc} is well-balanced for the same equilibrium states.	
\end{thm}
\begin{proof}
	Given an equilibrium state $\bm U^{\text{eq}}$ and the corresponding equilibrium variables $\bm E^{\text{eq}}$, we have that the semi-discretization in \eqref{e:fd_semi_disc} is WB, and thus $\bb{\mc K}_{j+\frac12}=\bb{\mc K}_{j-\frac12}=\bb{\mc{K}}^{\text{eq}}$ for all $j$.
	Now, let us prove that, for one mP-Euler time step \eqref{e:mP}, the equilibrium state is preserved at the next time step. The extension to mPDeC is straightforward.
	Since $\bb{\mc K}_{j+\frac12}=\bb{\mc K}_{j-\frac12}$, suppose, w.l.o.g., that ${\mc K}^h_{j+\frac12}={\mc K}^h_{j-\frac12}>0$, where ${\mc K}^h_{j+\frac12}$ denotes the first component of $\bb{\mc K}_{j+\frac12}$, corresponding to the water height $h$, the variable we want to keep positive. 
	Then, by \eqref{e:pd_entries}, we have that $p_{j,j+1} = d_{j,j-1} = 0$ and $p_{j,j-1} = d_{j,j+1} = \frac{1}{\dx}{\mc K}^h_{j+\frac12}$. 
	Since, the mP-Euler system has a unique solution \cite{burchard2003}, we prove that $\bm U^{n+1} = \bm U^n = \bm U^{\text{eq}}$ is a solution of the mP-Euler system.
	Indeed, the update formula for the first component $h$ reads
	\begin{align*}
		h^{n+1}_j &= h^n_j + \dt\left(  
		 \underbrace{p_{j,j-1}(\bm U^n)}_{={\mc K}^{\text{eq},h}}\underbrace{\frac{h_{j-1}^{n+1}}{h_{j-1}^{n}}}_{=1}
		+\underbrace{p_{j,j-1}(\bm U^n)}_{=0} \underbrace{\frac{h_{j+1}^{n+1}}{h_{j+1}^{n}}}_{=1} 
		-\underbrace{d_{j,j+1}(\bm U^n)}_{={\mc K}^{\text{eq},h}} \underbrace{\frac{h_{j}^{n+1}}{h_{j}^{n}}}_{=1}  
		-\underbrace{d_{j,j-1}(\bm U^n)}_{=0} \underbrace{\frac{h_{j}^{n+1}}{h_{j}^{n}}}_{=1} \right)=h_j^n
	\end{align*}
	that is clearly satisfied by the equilibrium. The case with negative ${\mc K}^h_{j+\frac12}$ is analogous.
\end{proof}

\section{Numerical results}\label{sec5}
Here we test the expected performance of the presented mPDeC WB positivity-preserving A-WENO scheme on several numerical examples, and compare with the results of the corresponding
fifth-order DeC A-WENO scheme that also uses an analogous local characteristic decomposition in the interpolation. For the sake of brevity, we will refer to these as the WB and non-well-balanced (NWB) schemes, respectively. 
To ensure fifth-order accuracy in the mPDeC and DeC time evolutions, we use four Gauss-Lobatto nodes in mPDeC ($r = 3$) and four total iterations ($M = 4$).
In all of the examples, the mesh is uniform, the CFL number is taken as $\nu = 0.45$,  and $g = 9.812$ unless otherwise stated. 


\subsubsection*{Example 1---Convergence study}
In the first example, we test the order of accuracy for both the WB and NWB schemes. 
We follow the example from, e.g., \cite{Cheng19ME, Noelle07HOWB}, with $g = 1$, which has the following smooth initial data and bottom topography:
$$
	h(x,0) = 5+ e^{\cos(2\pi x)}, \quad
	q(x,0) = \sin(\cos(2\pi x)), \quad
	Z(x) = \sin^2(\pi x) , \qquad 
	x \in [0,1]
$$
with periodic boundary conditions. 
Since the exact solution for this initial data is not explicitly known, we form a reference solution from the NWB scheme on a mesh of $N = 12800$, and use this to compute numerical errors. 
We present the $L^1$ errors and convergence rates for both the WB and NWB schemes at the final time $t = 0.1$ in Table \ref{t:e1rates}, where it is clear that we obtain the expected experimental fifth order of accuracy for both approaches.

\begin{table}[ht]
	\begin{center}
		\caption{\sf Example 1: $L^1$ errors and numerical convergences rates at $t = 0.1$.}
	\begin{tabular}[c]{l c c c c c}
		\multicolumn{6}{l}{Well-Balanced Method}\\ \hline
		N & $\norm{h(\cdot,t) - h_{\rm ref}}_1$ & $h$ Order & & $\norm{q(\cdot,t) - q_{\rm ref}}_1$ & $q$ Order   \\ \hline
		25 &	 6.13e-04 &	 -- &	& 2.17e-03 &	 -- \\
		50 & 	 3.32e-05 &	 4.21 &	& 8.43e-05 &	 4.69 \\
		100 &	 1.17e-06 &	 4.83 &	& 2.93e-06 &	 4.85 \\
		200 &	 3.95e-08 &	 4.88 &	& 1.01e-07 &	 4.85 \\
		400 &	 1.13e-09 &	 5.13 &	& 2.97e-09 &	 5.10 \\
		800 &	 3.82e-11 &	 4.89 &	& 1.02e-10 &	 4.87 \\ \hline
		& & & & & \\ 
		\multicolumn{6}{l}{Non-Well-Balanced Method}\\ \hline
		25 &	 5.24e-04 &	 --   &	& 2.33e-03 & -- \\
		50 &	 2.97e-05 &	 4.14 &	& 8.80e-05 & 4.73 \\
		100 &	 9.47e-07 &	 4.97 &	& 2.63e-06 & 5.06 \\ 
		200 &	 3.10e-08 &	 4.93 &	& 8.46e-08 & 4.96 \\
		400 &	 1.13e-09 &	 4.78 &	& 3.05e-09 & 4.79 \\
		800 &	 4.59e-11 &	 4.62 &	& 1.23e-10 & 4.63 \\ \hline
	\end{tabular}\label{t:e1rates}
	\end{center}
\end{table}

\subsubsection*{Example 2---Moving water equilibria}
In this example, we explore the evolution of different nontrivial moving water equilbria---those corresponding with subcritical, transcritical, and supercritical flow over a continuous bump bottom topography that reads 
\begin{equation} \label{e:Zbump}
	Z(x) = \begin{cases}
		0.2 - 0.05 (x - 10)^2, & \textrm{if } 8 \leq x \leq 12, \\
		0, & \textrm{otherwise.}
	\end{cases}
\end{equation}
These steady state problems are classical tests for the three different speed-regime flows converging to a steady state, see, e.g., \cite{Noelle07HOWB, Vazquez99Improved}. Since our purpose is to maintain these moving equilibria solutions exactly, we consider the following three sets of initial and boundary data on the domain [0,25]:
\begin{itemize}
	\item[(a)] \textit{Subcritical flow} with 
	\begin{equation}\label{e:subICs}
		\begin{aligned}
			&\mc E_{\rm eq}(x) = 22.06605, \qquad &&q_{\rm eq}(x) = 4.42, \\
			& q(0,t)= 4.42, \qquad &&h(25, t) = 2.
		\end{aligned}
	\end{equation}
	
	\item[(b)] \textit{Transcritical flow} with 
	\begin{equation}\label{e:transICs}
		\begin{aligned}
			&\mc E_{\rm eq}(x) = \frac{3}{2} (1.53g)^{2/3}+0.2g, \qquad &&q_{\rm eq}(x) = 1.53, \\
			& q(0,t)= 1.53, \qquad &&\mc E(25, t) = \mc E_{\rm eq}(25).
		\end{aligned}
	\end{equation}
	
	\item[(c)] \textit{Supercritical flow} with 
	\begin{equation}\label{e:superICs}
		\begin{aligned}
			&\mc E_{\rm eq}(x) = 91.624, \qquad &&q_{\rm eq}(x) = 24, \\
			& q(0,t)= 24, \qquad &&h(25, t) = 2.
		\end{aligned}
	\end{equation}
\end{itemize}
Here, the equilibrium conditions for $h_{\rm eq}(x)$ are found by solving the cubic equation for $\mc E_{\rm eq}$ in  \eqref{e:Ej}. We plot the initial steady state conditions for $h_{\rm eq}(x)$ in Figure \ref{f:ME}.
Note that for the transcritical initial conditions, we select the subcritical root of this cubic before $x = 10$ (the peak of our bottom topography), and select the supercritical root after $x = 10$.
	
We compute the solution using a mesh with $N = 20$ and $N = 200$ points up to a final time of $t = 20$, and present the $L^1$ and $L^\infty$ numerical errors of the WB scheme for all 3 test cases in Table \ref{t:me_errors}. We see that, for all three cases, the proposed approach preserves the steady state conditions up to numerical round-off error. 


\begin{figure}[ht!]
	\centerline{
		\includegraphics[trim=1.9cm 0.7cm 1.8cm 1.0cm, clip, width=0.33\textwidth]{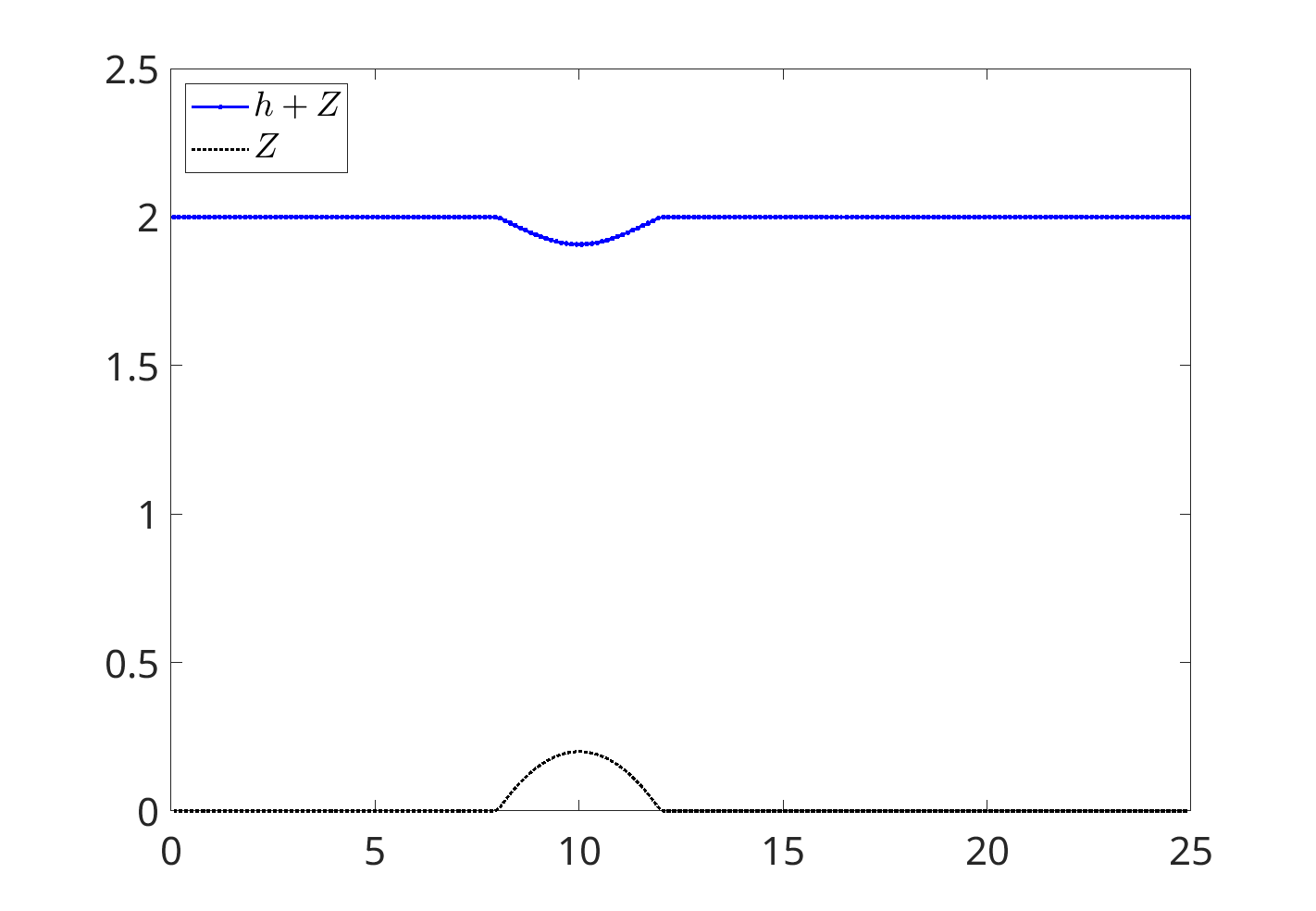}
		\includegraphics[trim=1.9cm 0.7cm 1.8cm 1.0cm, clip, width=0.33\textwidth]{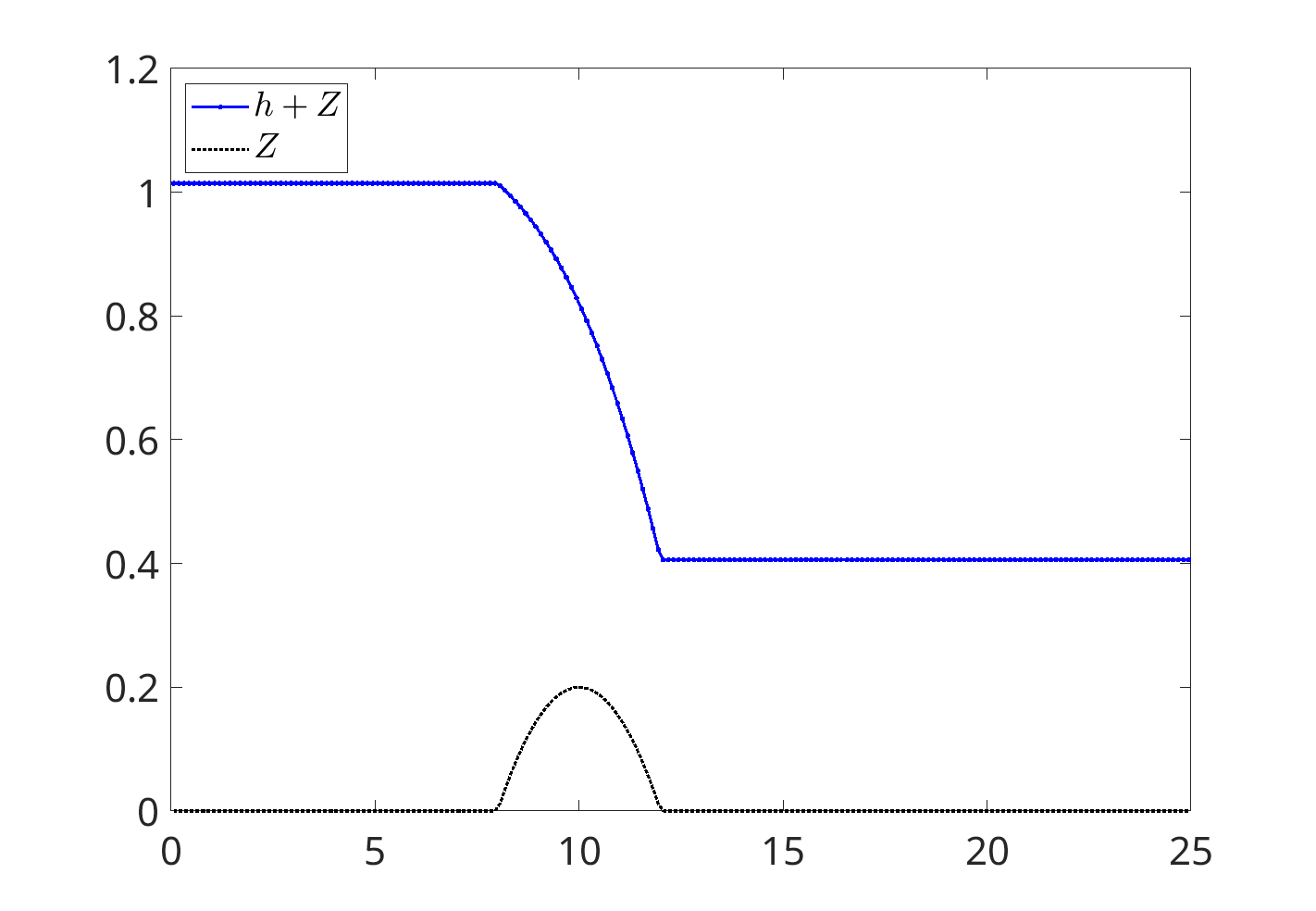}
		\includegraphics[trim=1.9cm 0.7cm 1.8cm 1.0cm, clip, width=0.33\textwidth]{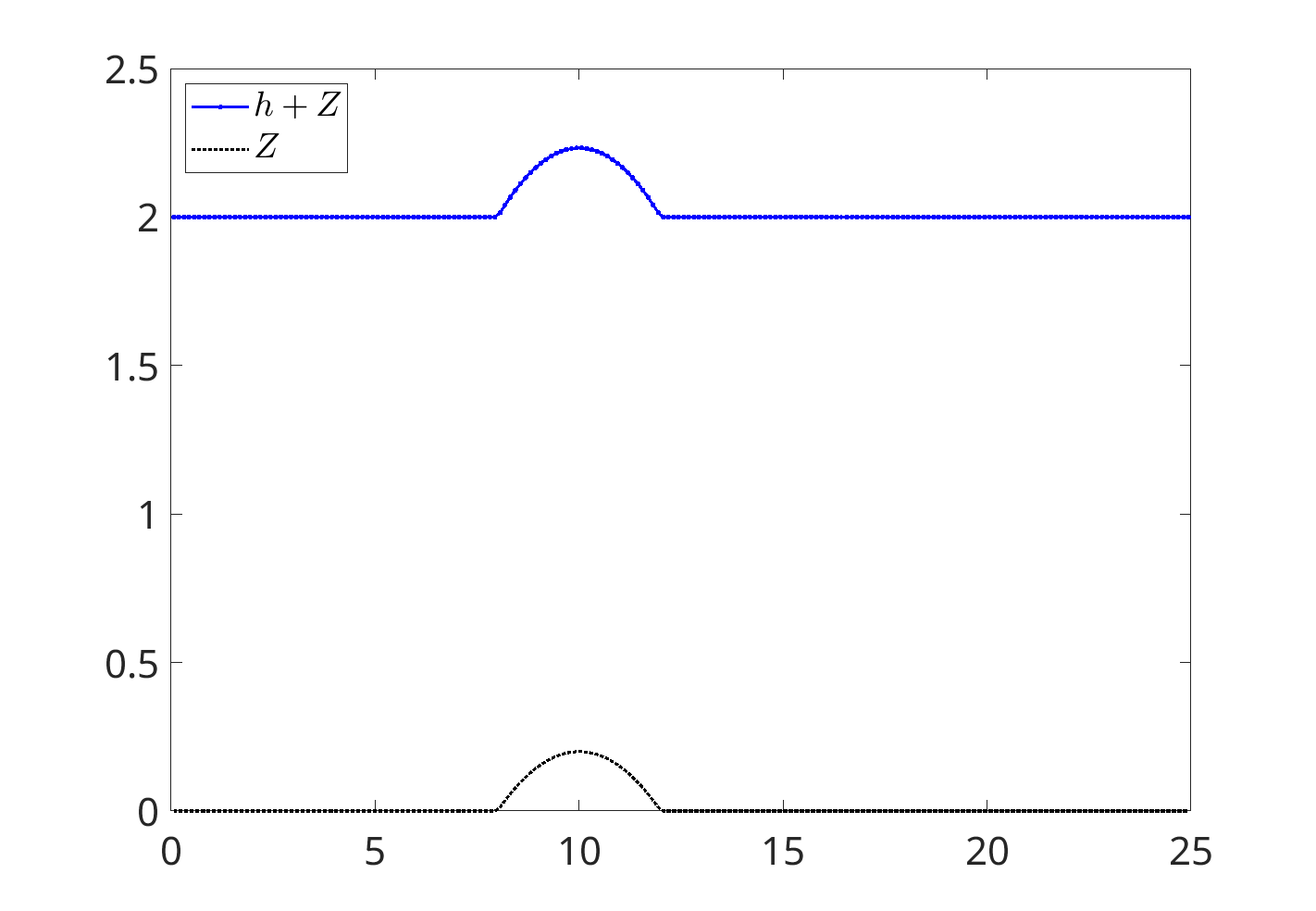}}
	\caption{\sf Example 2: (Left) subcritical, (middle) transcritical, and (right) supercritical moving water steady states. Notice the difference in vertical scales.}
	\label{f:ME} 
\end{figure}

\begin{table}[ht]
	\begin{center}
		\caption{\sf Example 2: $L^1$ and $L^\infty $ errors at $t = 20$ for the subcritical, transcritical, and supercritical flow equilibria.}
		\begin{tabular}[c]{l c c c c c}
			\multicolumn{6}{l}{(a) Subcritical Flow}\\ \hline
			N & $\norm{h(\cdot,t) - h_{\rm eq}}_1$ & $\norm{q(\cdot,t) - q_{\rm eq}}_1$ & & $\norm{h(\cdot,t) - h_{\rm eq}}_\infty$ & $\norm{q(\cdot,t) - q_{\rm eq}}_\infty$  \\ \hline
			20 &	1.94e-13 &	 3.40e-13 & & 1.71e-14 & 3.02e-14  \\
			200 &	5.70e-13 &	 2.13e-12 &	& 6.35e-14 & 2.13e-13 \\ \hline
			& & & & & \\ 
			\multicolumn{6}{l}{(b) Transcritical Flow}\\ \hline
			20 &	4.19e-14 &	 4.88e-14 & & 5.55e-15 & 3.55e-15  \\
			200 &	2.29e-13 &	 5.89e-13 & & 3.66e-14 & 5.35e-14 \\ \hline
			& & & & & \\ 
			\multicolumn{6}{l}{(c) Supercritical Flow}\\ \hline
			20 &	1.34e-13 &	 9.41e-13 & & 1.02e-14 & 7.82e-14  \\
			200 &	4.33e-13 &	 2.86e-12 & & 5.84e-14 & 4.87e-13 \\ \hline
		\end{tabular}\label{t:me_errors}
	\end{center}
\end{table}

\subsubsection*{Example 3---Perturbation of moving equilibria}
Next, we perturb the initial steady state conditions presented in Example 2 to examine the performance of the proposed WB scheme on a small perturbation problem. 
Namely, we introduce a perturbation to the initial fluid height data seen in \eqref{e:subICs}, \eqref{e:transICs}, and \eqref{e:superICs}:
\begin{equation*}\label{e:h_pert}
	h(x,0) = h_{\rm eq}(x) + \begin{cases}
		\eps_{\rm pert} & \textrm{if } 4.5 \leq x \leq 5.5, \\
		0 & \textrm{otherwise},
	\end{cases}
\end{equation*}
where $\eps_{\rm pert}$ denotes the size of the perturbation.

We first take $\eps_{\rm pert} = 10^{-3}$, as they did in, e.g., \cite{Cheng19ME}, to compare to the solution of the NWB scheme. 
For each of the three cases, we compute the WB and NWB solution using a mesh with $N = 100$ and $N = 1000$ points to a final time $t = 1$ for the subcritical and supercritical cases, and to a final time $t = 1.5$ for the transcritical flow example. 
We plot the perturbation from the moving equilibrium solutions $h(x,t) - h_{\rm eq}(x)$ for the subcritical, transcritical, and supercritical flow in Figures \ref{f:Sub_Pert}, \ref{f:Trans_Pert}, and \ref{f:Super_Pert}, respectively. 

For the subcritical and transcritical case, the NWB scheme fails to capture the correct solution on the course $N = 100$ mesh. 
Furthermore, even on the finer $N = 1000$ mesh, the NWB scheme produces oscillations that are outside the confines of the correct perturbation.
On the other hand, the WB solution is oscillation free and captures the correct structure of the solution, even on the coarse mesh.
For the supercritical case, both the NWB scheme and WB scheme capture the solution well for both meshes; we do see some sight oscillations in the NWB result, but nothing as significant as the previous two flow cases. 
The local characteristic decomposition plays a big role in reducing the oscillations in the NWB scheme to obtain a sound result.


To further demonstrate the capabilities of the WB scheme, we now reduce the perturbation size to $\eps_{\rm pert} = 10^{-5}$. 
We run the three flow cases to the same final times, and plot the perturbation from the corresponding steady states $h(x,t) - h_{\rm eq}(x)$ in Figure \ref{f:ME_Small_Pert}.
One can see that, even on a much smaller perturbation size, the WB scheme captures the correct solution structure on the coarse $N = 100$ mesh, and additionally remains essentially non-oscillatory when refining the grid.

\begin{figure}[ht!]
	\centerline{
		\includegraphics[trim=0.5cm 0.6cm 1.8cm 0.2cm, clip, width=0.49\textwidth]{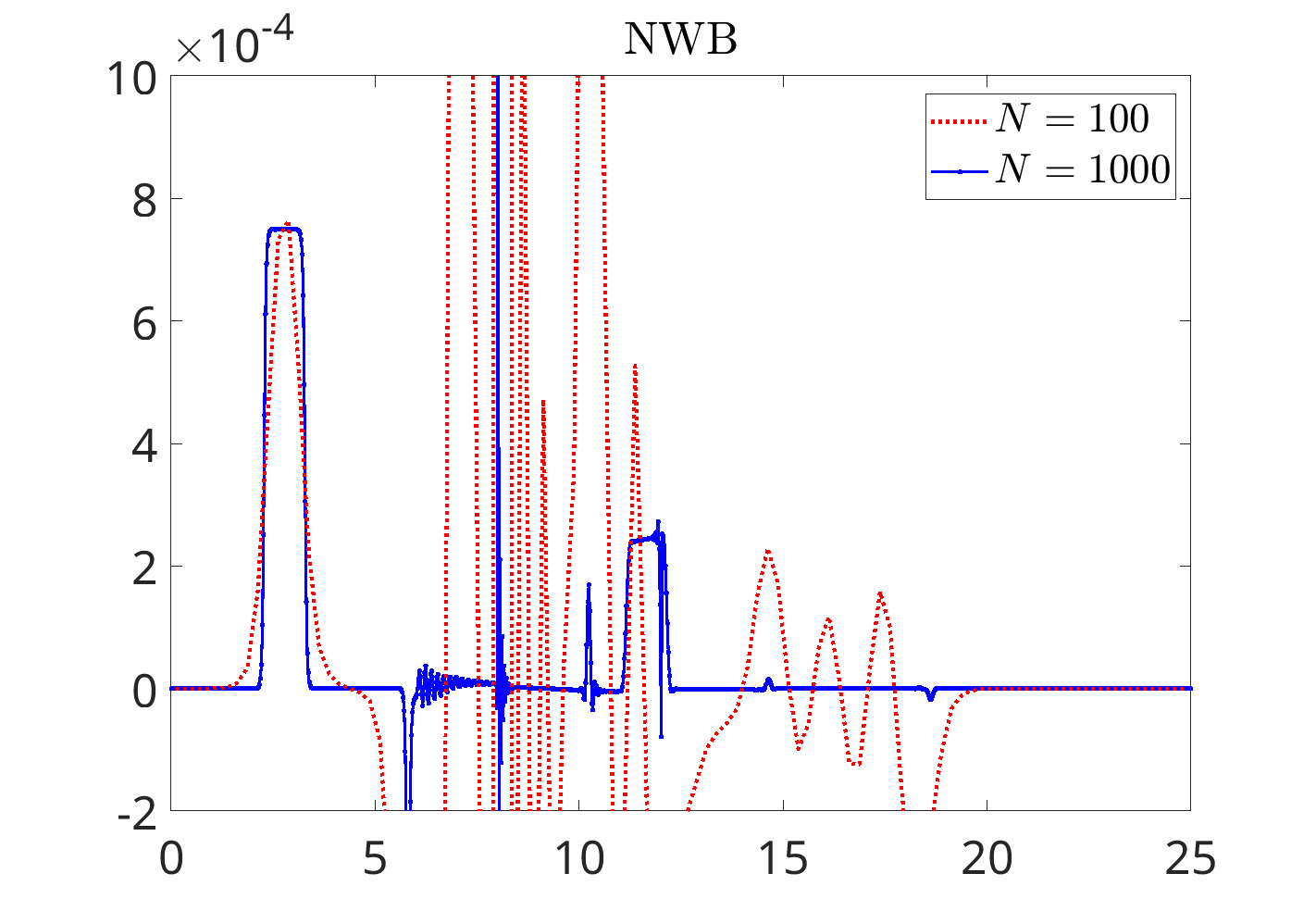}
		\includegraphics[trim=0.5cm 0.6cm 1.8cm 0.2cm, clip, width=0.49\textwidth]{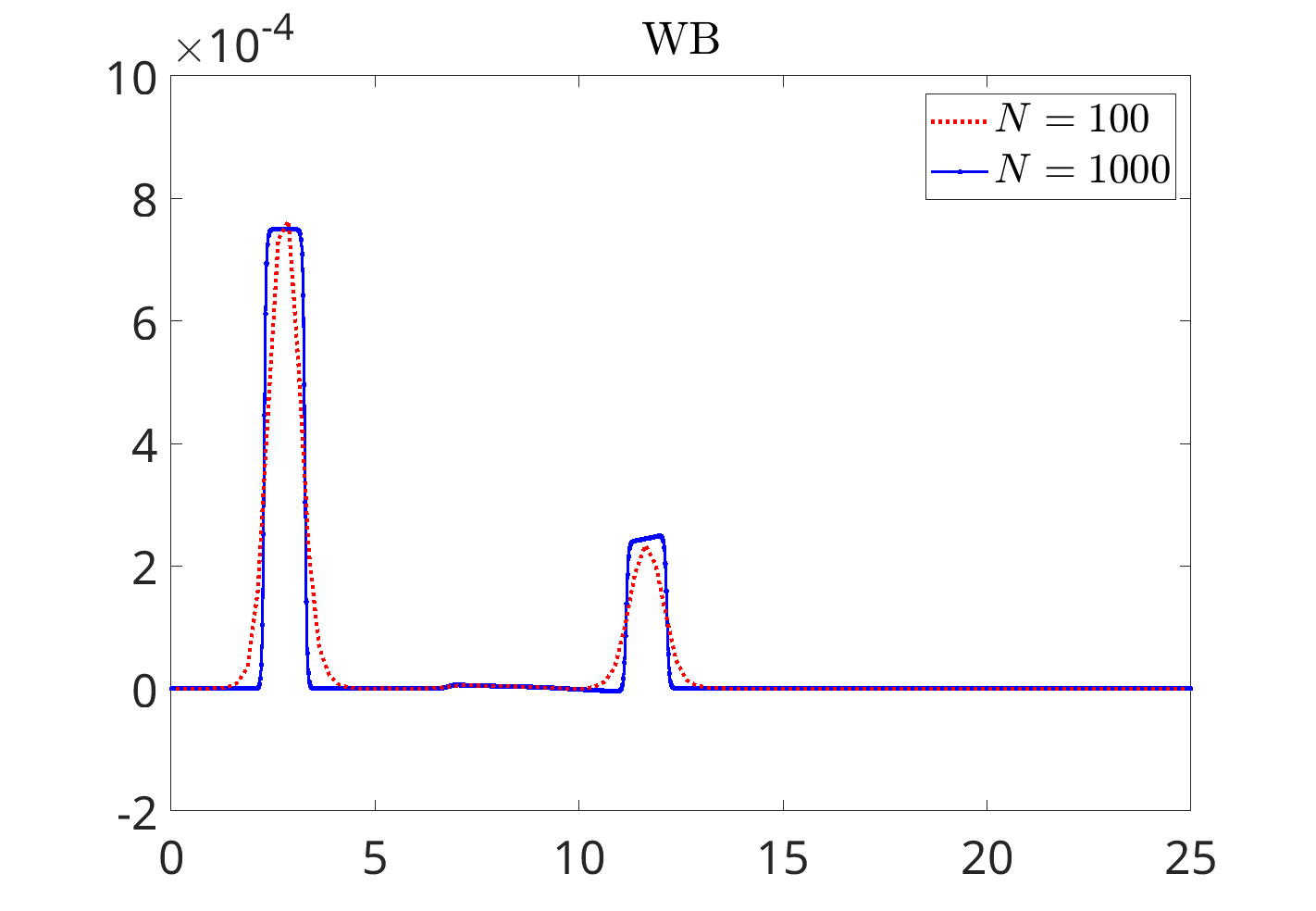}}
	\caption{\sf Example 3(a): The evolved perturbation $h(x,t) - h_{\rm eq}(x)$ of the subcritical flow steady state at time $t = 1$ for (left) the NWB scheme and (right) the WB scheme for $\eps_{\rm pert} = 10^{-3}$.}
	\label{f:Sub_Pert} 
\end{figure}

\begin{figure}[ht!]
	\centerline{
		\includegraphics[trim=0.5cm 0.6cm 1.8cm 0.2cm, clip, width=0.49\textwidth]{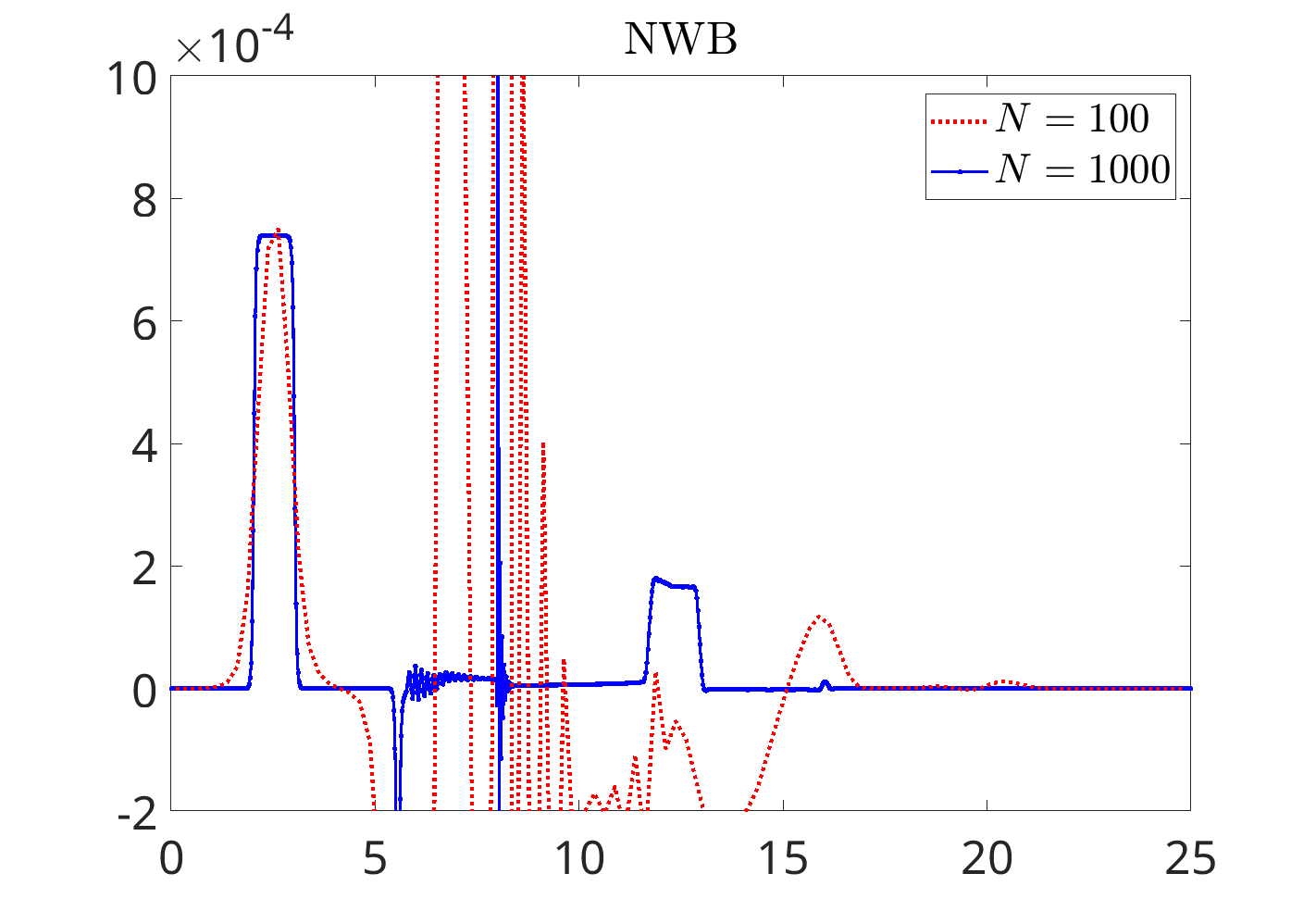}
		\includegraphics[trim=0.5cm 0.6cm 1.8cm 0.2cm, clip, width=0.49\textwidth]{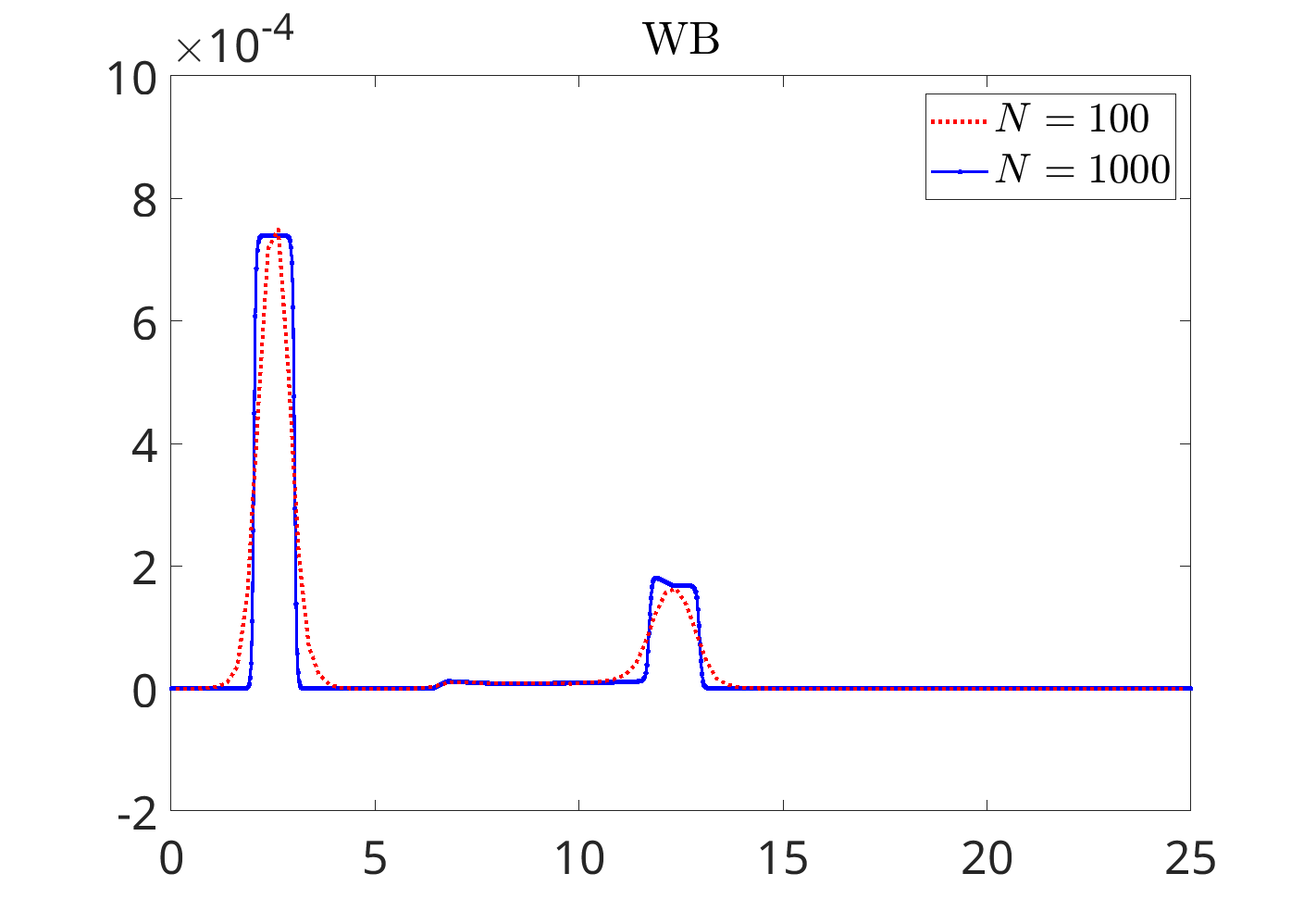}}
	\caption{\sf Example 3(b): The evolved perturbation $h(x,t) - h_{\rm eq}(x)$ of the transcritical flow steady state at time $t = 1.5$ for (left) the NWB scheme and (right) the WB scheme for $\eps_{\rm pert} = 10^{-3}$.}
	\label{f:Trans_Pert} 
\end{figure}

\begin{figure}[ht!]
	\centerline{
		\includegraphics[trim=0.5cm 0.6cm 1.8cm 0.2cm, clip, width=0.49\textwidth]{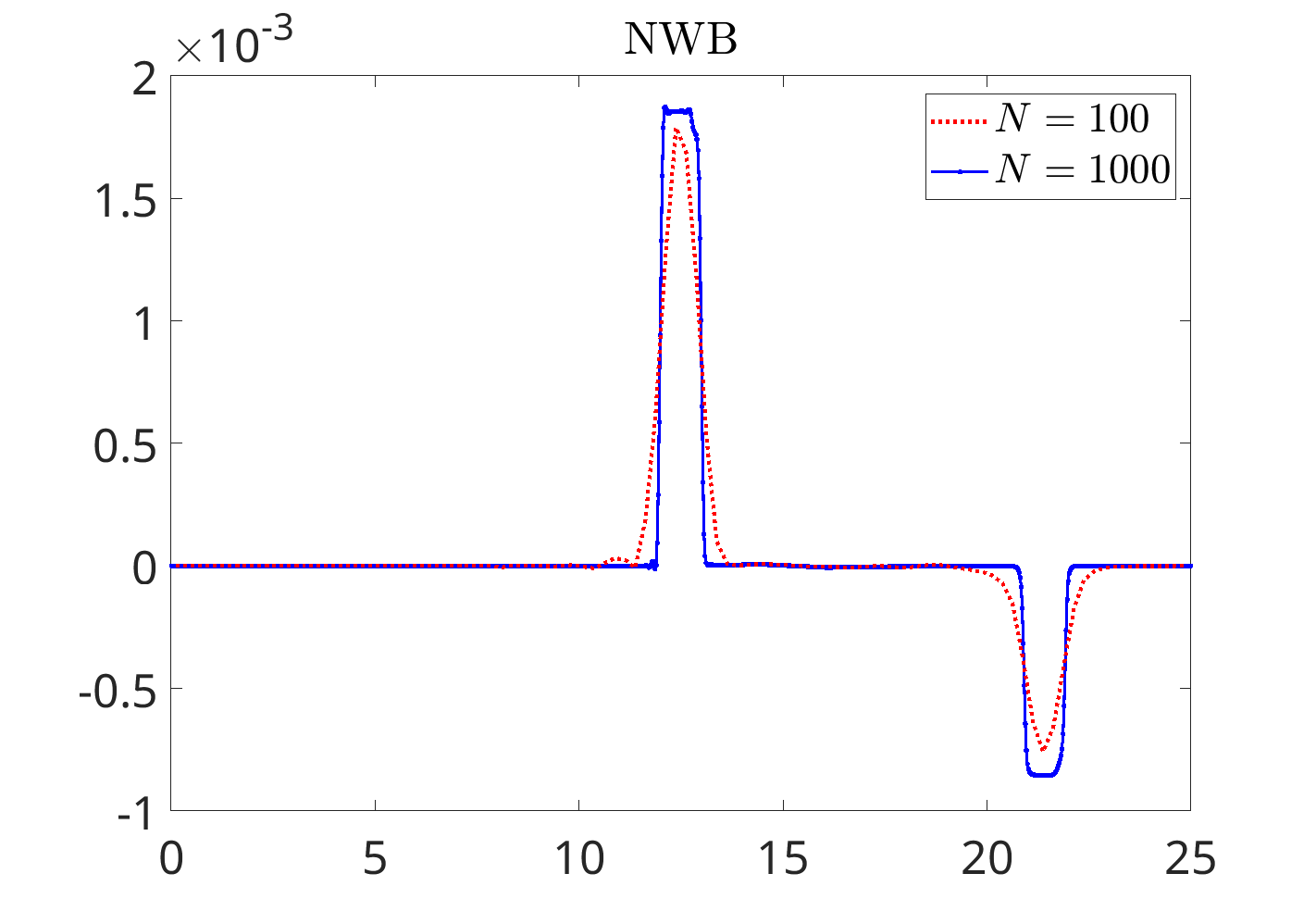}
		\includegraphics[trim=0.5cm 0.6cm 1.8cm 0.2cm, clip, width=0.49\textwidth]{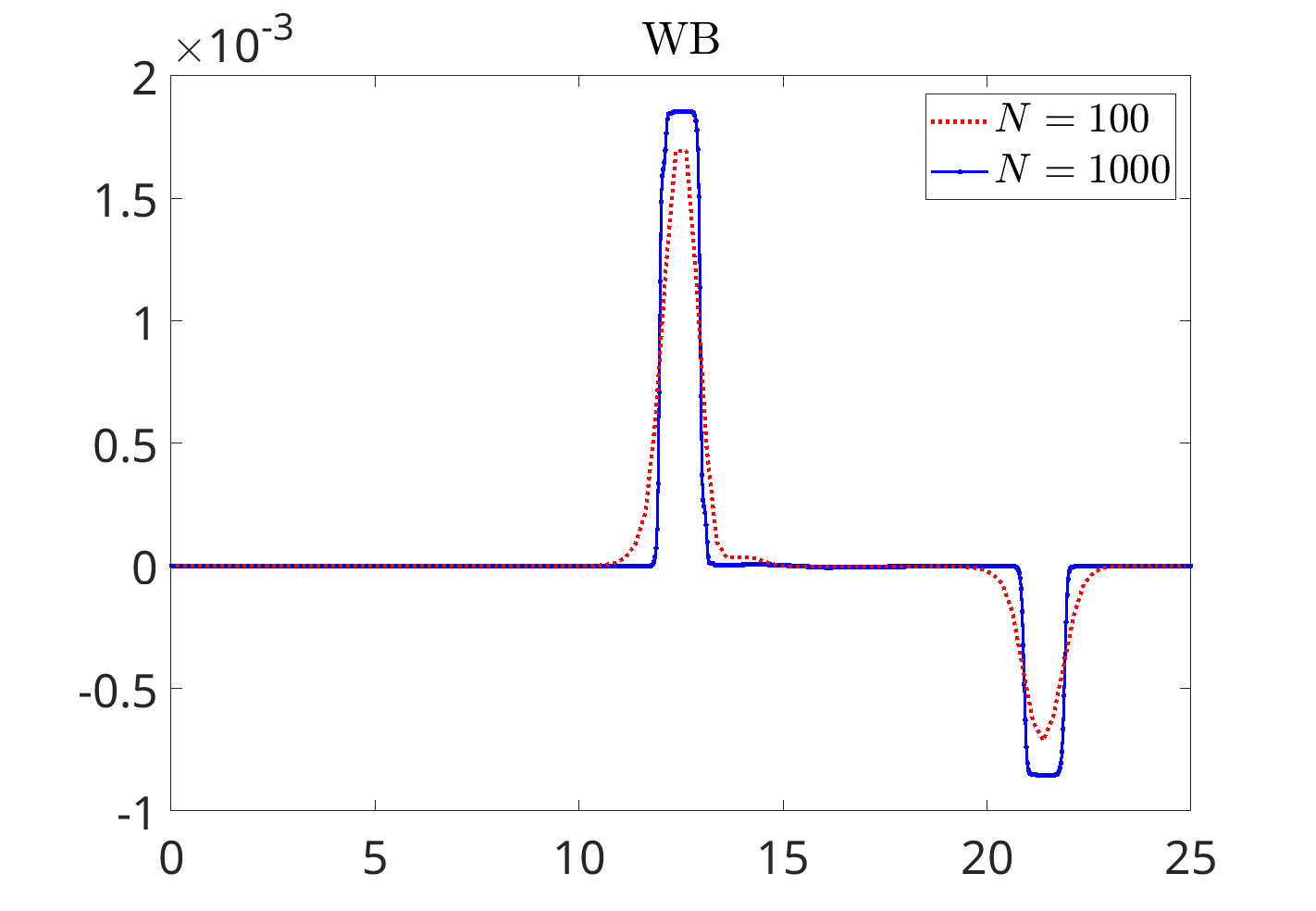}}
	\caption{\sf Example 3(c): The evolved perturbation $h(x,t) - h_{\rm eq}(x)$ of the supercritical flow steady state at time $t = 1$ for (left) the NWB scheme and (right) the WB scheme for $\eps_{\rm pert} = 10^{-3}$.}
	\label{f:Super_Pert} 
\end{figure}

\begin{figure}[ht!]
	\centerline{
		\includegraphics[trim=0.5cm 0.6cm 1.8cm 0.2cm, clip, width=0.49\textwidth]{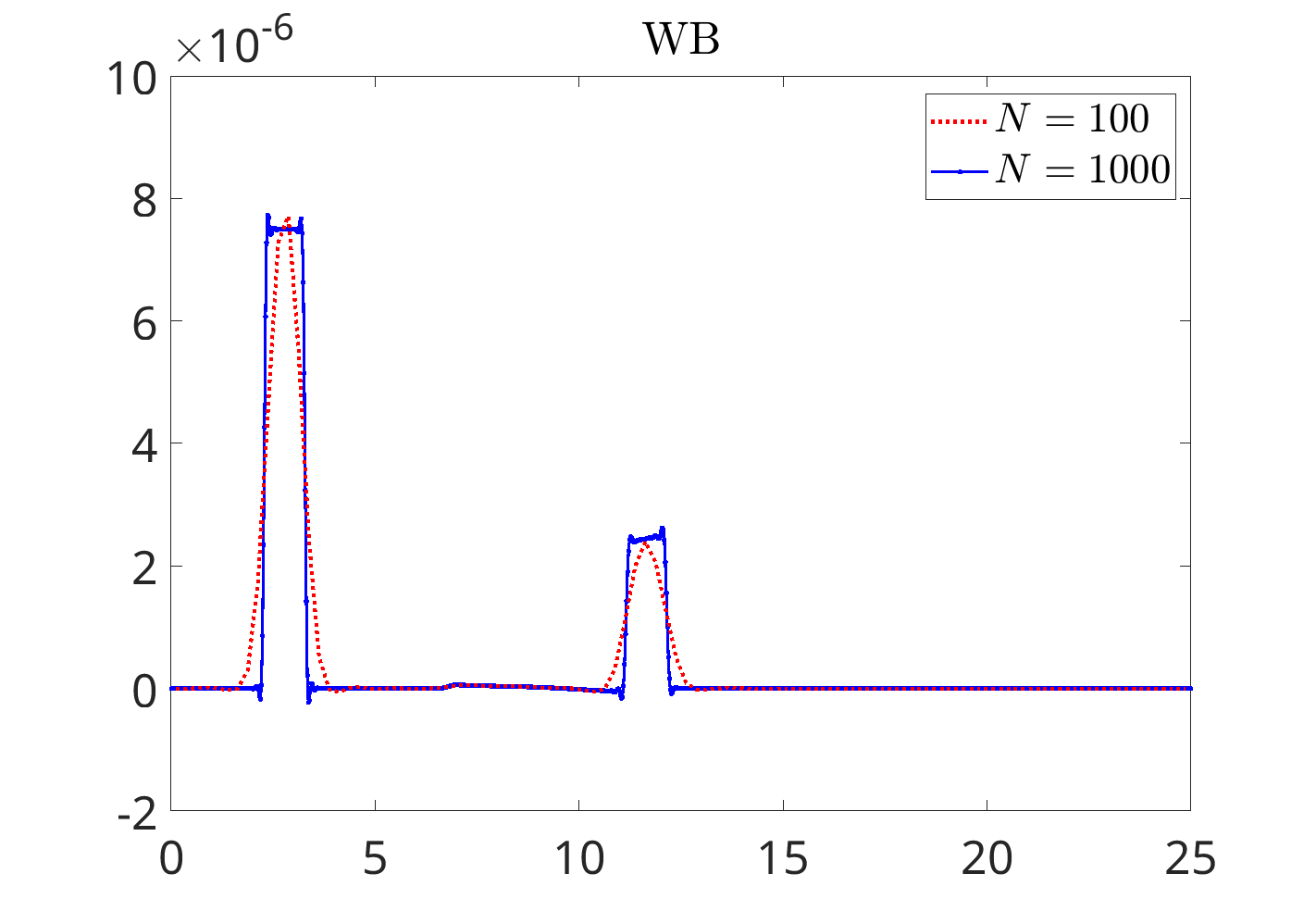}
		\includegraphics[trim=0.5cm 0.6cm 1.8cm 0.2cm, clip, width=0.49\textwidth]{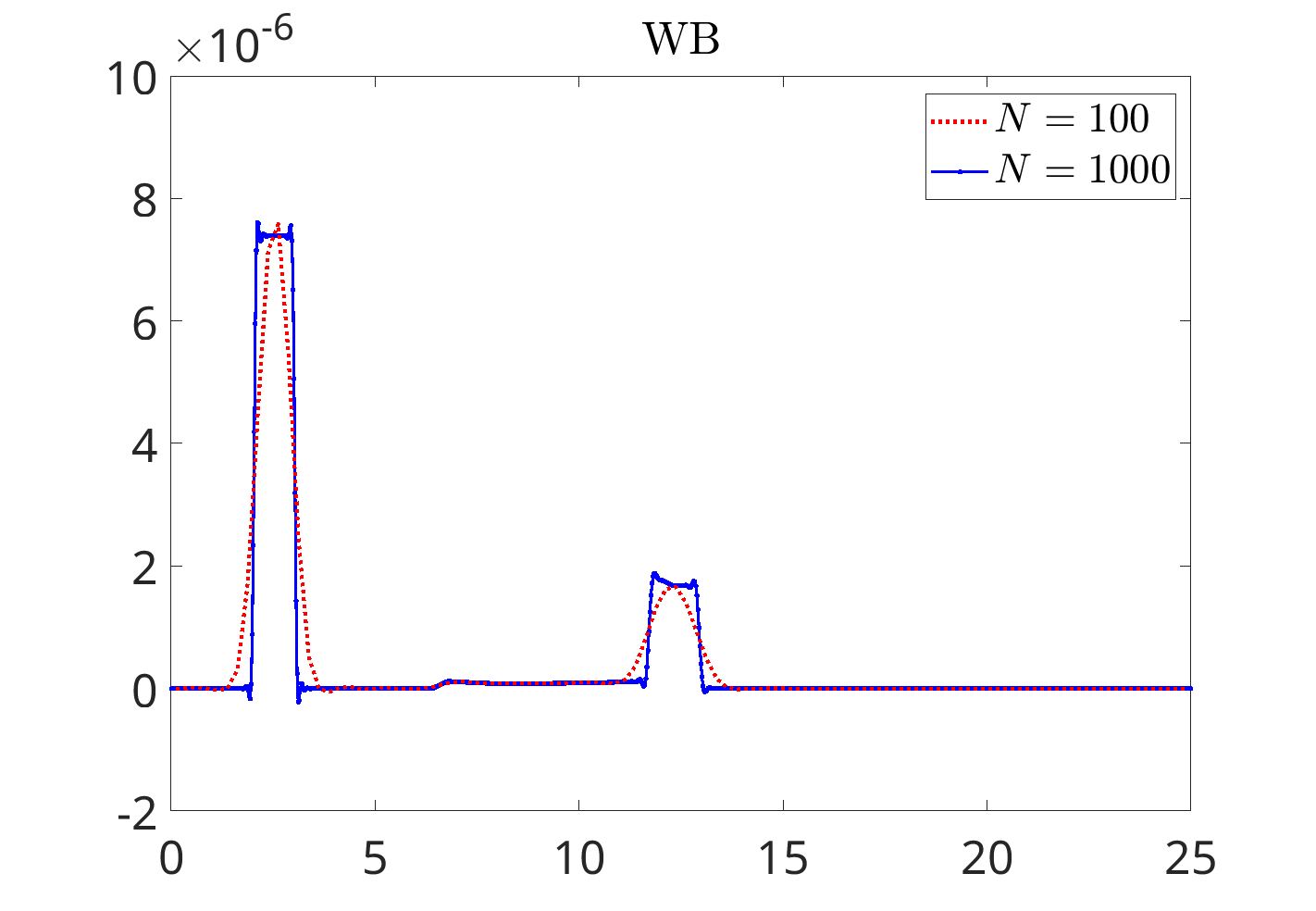}}
		\vspace*{7pt}
	\centerline{
		\includegraphics[trim=0.5cm 0.6cm 1.8cm 0.2cm, clip, width=0.49\textwidth]{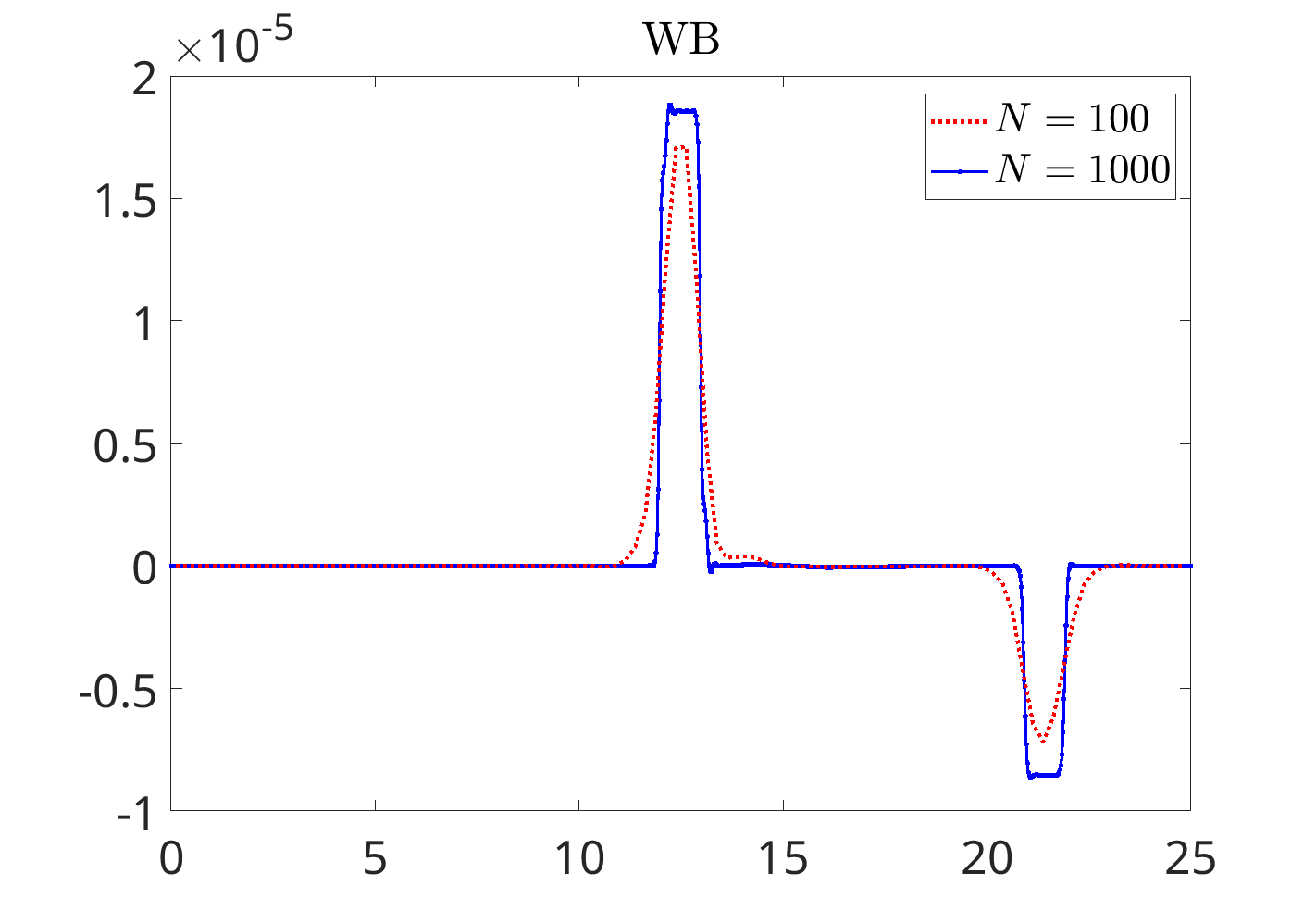}}
	\caption{\sf Example 3: The WB evolved perturbation $h(x,t) - h_{\rm eq}(x)$ of the (left) subcritical, (right) transcritical, and (bottom) supercritical flow at $t = 1$ and initial perturbation $\eps_{\rm pert} = 10^{-5}$.}
	\label{f:ME_Small_Pert} 
\end{figure}

\subsubsection*{Example 4---Wet/dry lake-at-rest steady state}
In this example, we continue our exploration of steady state solutions, now on a wet/dry lake-at-rest equilibria. This example, seen in e.g., \cite{Chertock22WB}, combines a lake-at-rest and ``dry lake" steady states, testing both the WB and positivity-preserving properties of the scheme. 
We consider the following initial data on the domain [0,20]:
\begin{equation}\label{e:LaR}
	h_{\rm eq}(x) = \max\set{0, 0.4-Z(x)}, \qquad u_{\rm eq}(x) = 0,
\end{equation}
with a two-hump bottom topography that reads 
\begin{equation} \label{e:Z2bump}
	Z(x) = \begin{cases}
		0.2 - 0.05 (x - 7)^2 & \textrm{if } 6 \leq x \leq 8, \\
		0.48 - 0.12 (x - 17)^2 & \textrm{if } 15\leq x \leq 19, \\
		0 & \textrm{otherwise.}
	\end{cases}
\end{equation}
Both the steady state \eqref{e:LaR} and the bottom topography \eqref{e:Z2bump} are presented in Figure \ref{f:LaR}.

We advance this steady state using $N = 100$ and $N = 200$ points to test a proper capturing of this lake-at-rest steady state up to machine precision. We present the $L^1$ and $L^\infty$ errors for $h$ and $q$ in Table \ref{t:lar_errors}, where it is clear that the WB scheme captures the steady state solution to round-off error accuracy. 
\begin{figure}[ht!]
	\centerline{
		\includegraphics[trim=1.9cm 0.7cm 1.8cm 1.0cm, clip, width=0.49\textwidth]{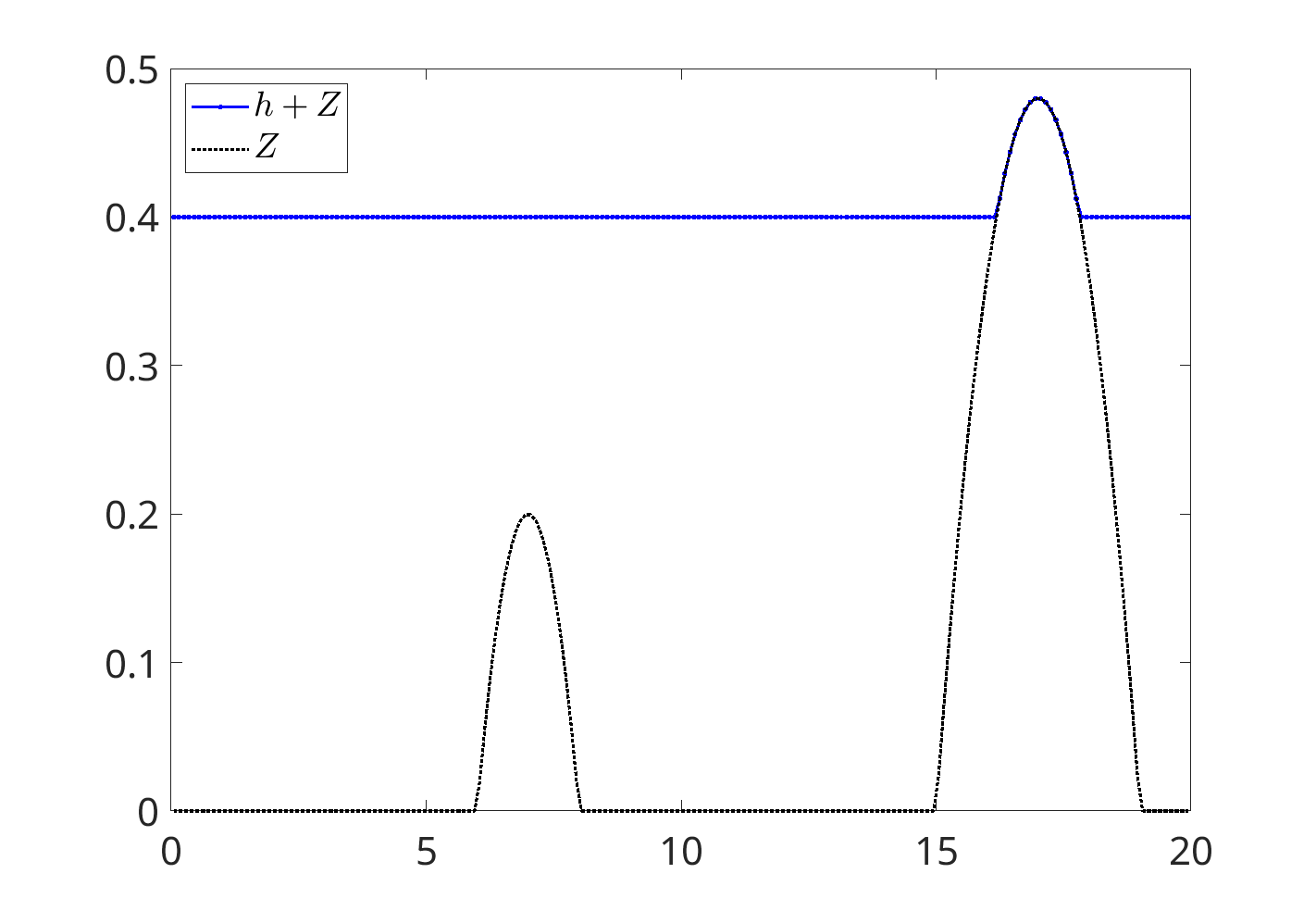}}
	\caption{\sf Example 4: The wet/dry lake-at-rest steady state.}
	\label{f:LaR} 
\end{figure}

\begin{table}[ht]
	\begin{center}
		\caption{\sf Example 4: $L^1$ and $L^\infty $ errors at $t = 20$ for the subcritical, transcritical, and supercritical flow equilibria.}
		\begin{tabular}[c]{l c c c c c}
			\multicolumn{6}{l}{Supercritical Inflow}\\ \hline
			N & $\norm{h(\cdot,t) - h_{\rm eq}}_1$ & $\norm{q(\cdot,t) - q_{\rm eq}}_1$ & & $\norm{h(\cdot,t) - h_{\rm eq}}_\infty$ & $\norm{q(\cdot,t) - q_{\rm eq}}_\infty$  \\ \hline
			100 &	4.15e-14 &  1.19e-13  & & 4.16e-15  & 5.87e-14  \\
			200 &	5.86e-14 &	1.19e-13 &	& 1.47e-14  & 1.04e-13 \\ \hline
		\end{tabular}\label{t:lar_errors}
	\end{center}
\end{table}

\subsubsection*{Example 5---Perturbation of wet/dry lake-at-rest}
As we did for the moving equilibria, we will also perturb the wet/dry lake-at-rest steady state and evaluate the evolution of the small traveling waves. We introduce a central perturbation to $h_{\rm eq}$ using 
\begin{equation*}\label{e:h_pert_lar}
	h(x,0) = h_{\rm eq}(x) + \begin{cases}
		\eps_{\rm pert} & \textrm{if } 9.5 \leq x \leq 10.5, \\
		0 & \textrm{otherwise},
	\end{cases}
\end{equation*}
where in this example, we follow, e.g., \cite{Chertock22WB}, and take an initial perturbation $\eps_{pert} = 10^{-4}$.

The lake-at-rest small perturbation is evolved on a mesh with $N = 400$ points to a final time $t = 4.8$---long enough to where the right-traveling wave interacts with the dry portion of the steady state. 
We present time snapshots of the evolution of $h(x,t) - h_{\rm eq}(x)$ in Figure \ref{f:LaR_pert}. 
The initial perturbation splits into two traveling waves, as seen in the $t = 0.6$ snapshot. 
At $t = 1.8$, the left-traveling wave is positioned over the smaller bump in the bottom topography, resulting in a slight magnitude increase and a small reflected wave traveling right. 
The left-moving wave returns to the original wave size of $5\times 10^{-5}$ after traveling over the smaller bump, as seen at $t = 3$.
At the same time, the right-traveling wave arrives at the larger bump of the bottom topography located at $x \in [15,19]$. 
After time $t = 3$, the left-traveling wave keeps its shape and changes minimally for the remainder of the simulation. 
The opposite can be said for the right-traveling wave, as this is when it greatly interacts with the larger bump of the bottom topography. 
At time $t = 3$, the right-going wave travels over the submerged portion of this larger bump, thus increasing the magnitude of the perturbation and releasing a small left-traveling reflected wave. 
At time $t = 3.6$, the right traveling wave reaches the dry portion of the domain at $x \approx 16.2$, resulting in a substantial increase of the perturbation magnitude. 
The water then flows down the originally-dry portion of the larger bump, now resulting in a left-propagating wave at time $t = 4.2$.
Finally, at $t = 4.8$, we see the reflected wave from the small bump and the water flowing off the large bump interacting, resulting temporarily in a visually smooth wave front at $x \approx 15$.

\begin{figure}[ht!]
	\centerline{
		\includegraphics[trim=1.9cm 0.7cm 1.8cm 0.2cm, clip, width=0.33\textwidth]{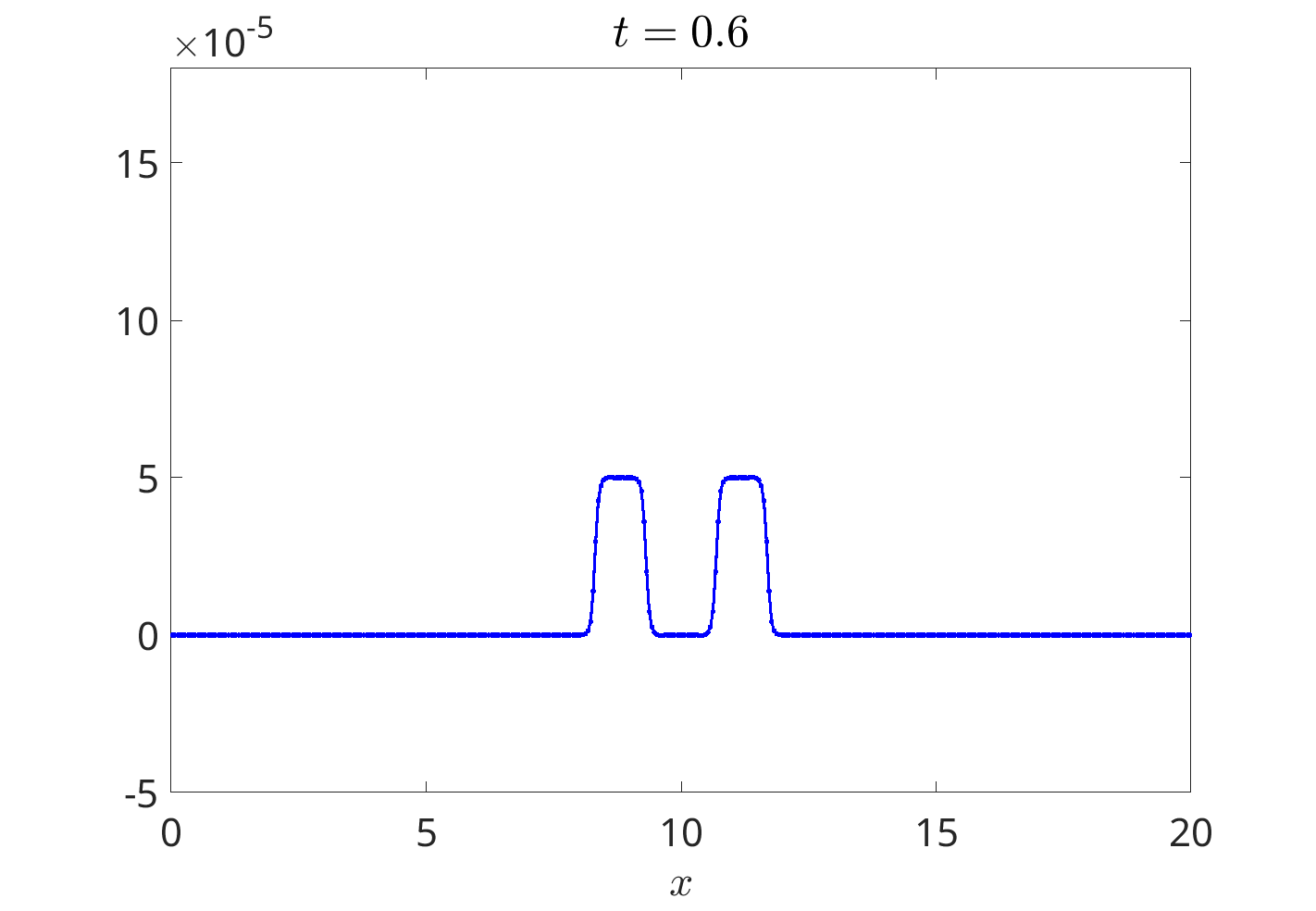}
		\includegraphics[trim=1.9cm 0.7cm 1.8cm 0.2cm, clip, width=0.33\textwidth]{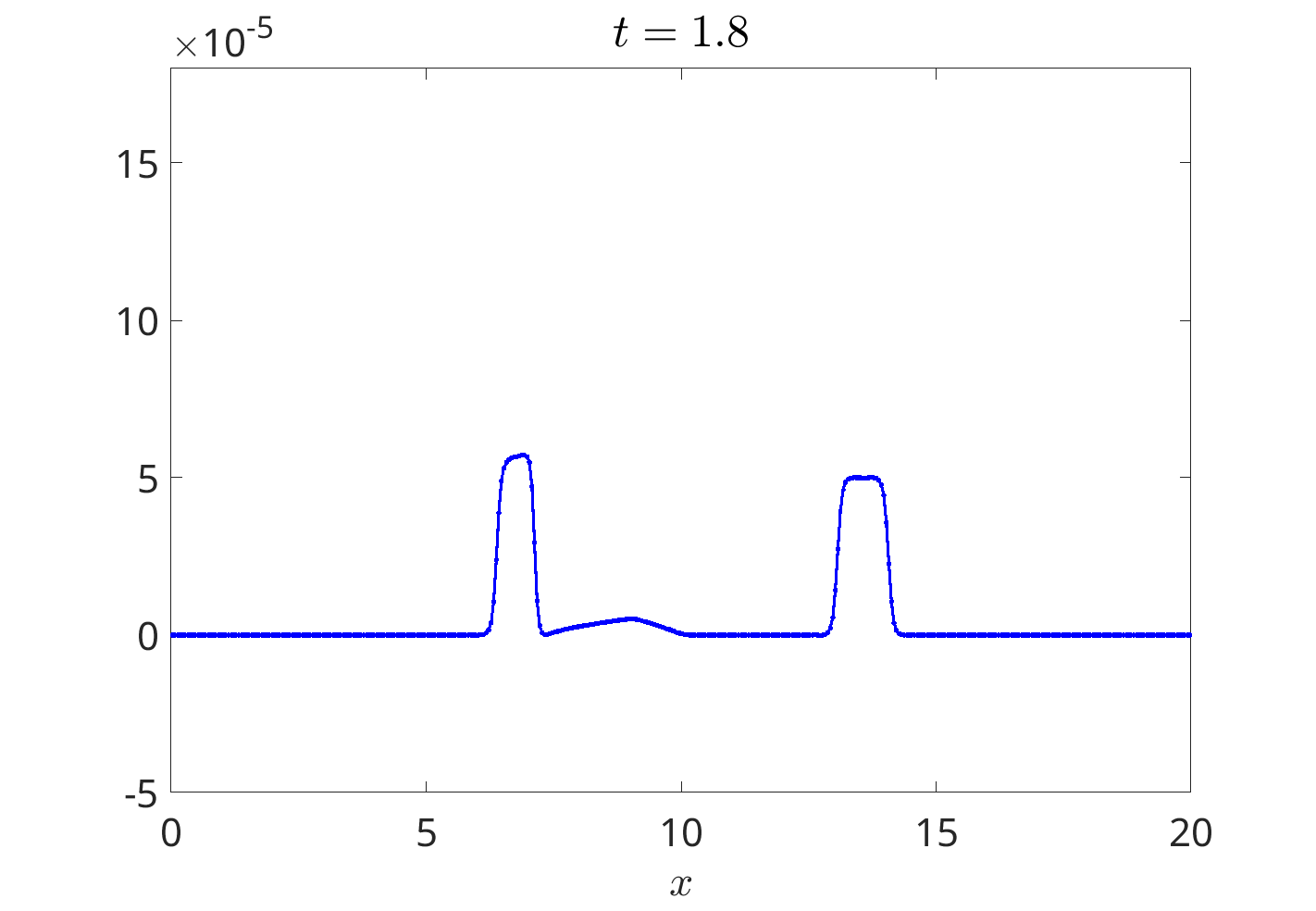}
		\includegraphics[trim=1.9cm 0.7cm 1.8cm 0.2cm, clip, width=0.33\textwidth]{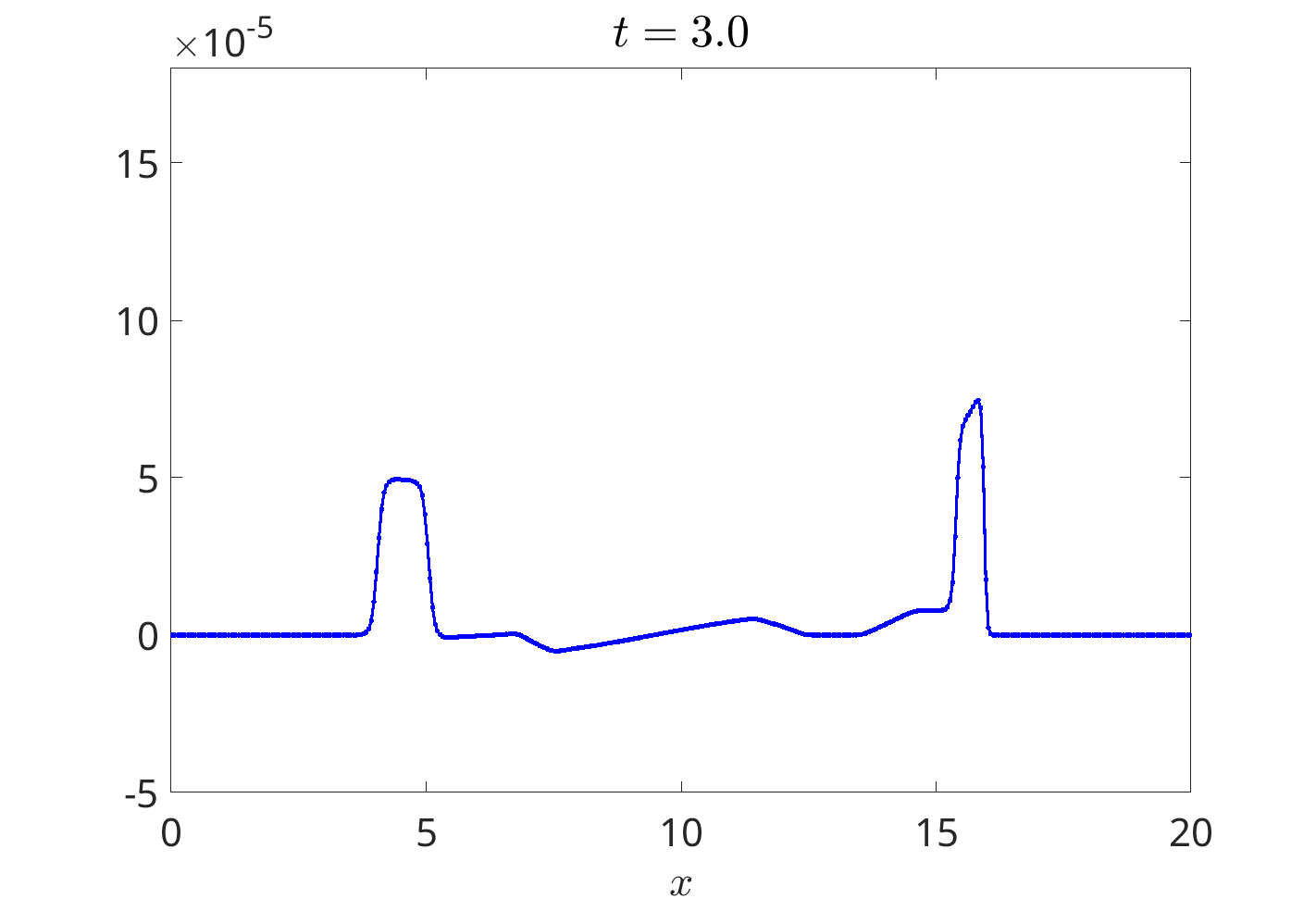}}
		
		\vspace*{7pt}
		\centerline{
		\includegraphics[trim=1.9cm 0.7cm 1.8cm 0.2cm, clip, width=0.33\textwidth]{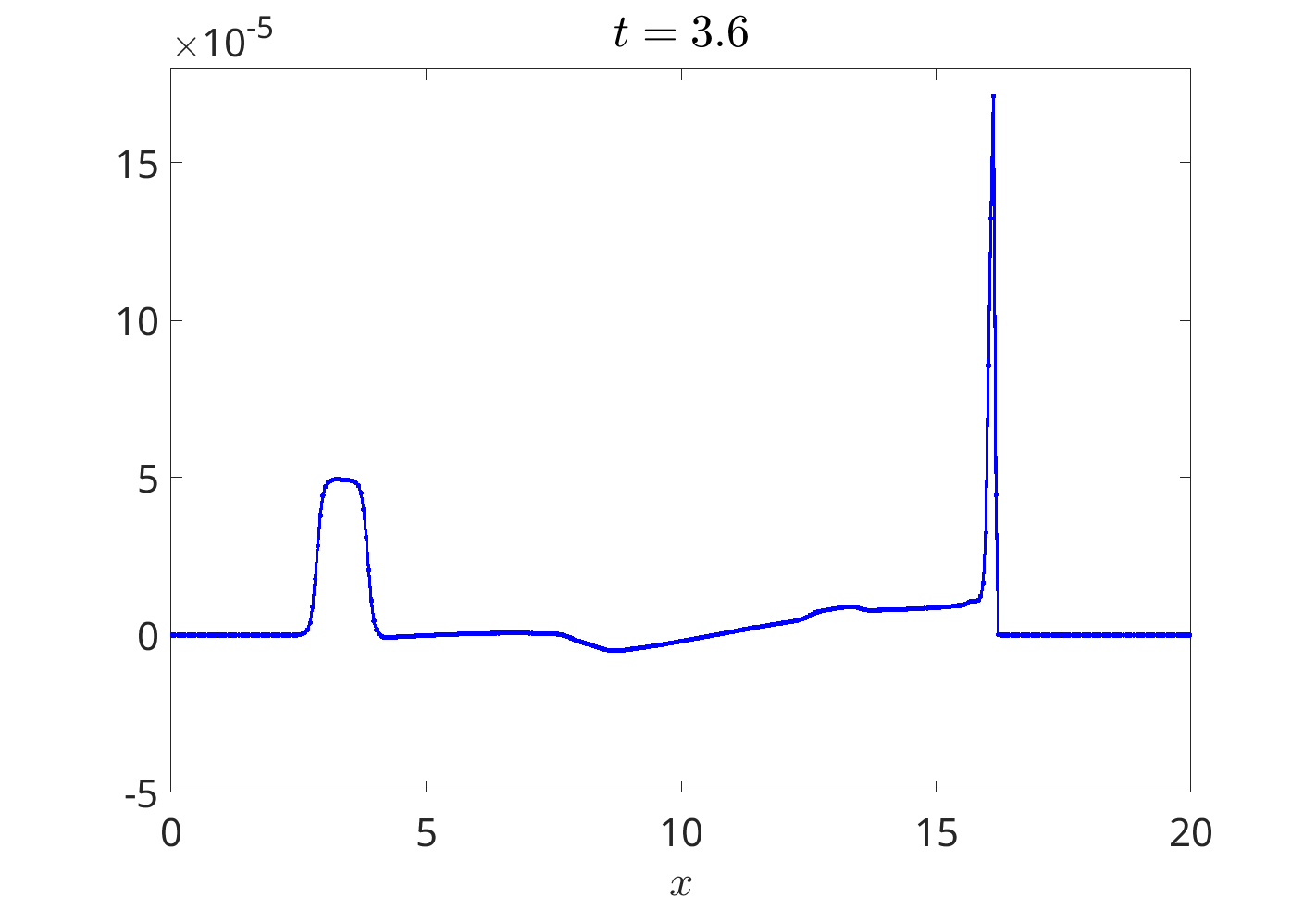}
		\includegraphics[trim=1.9cm 0.7cm 1.8cm 0.2cm, clip, width=0.33\textwidth]{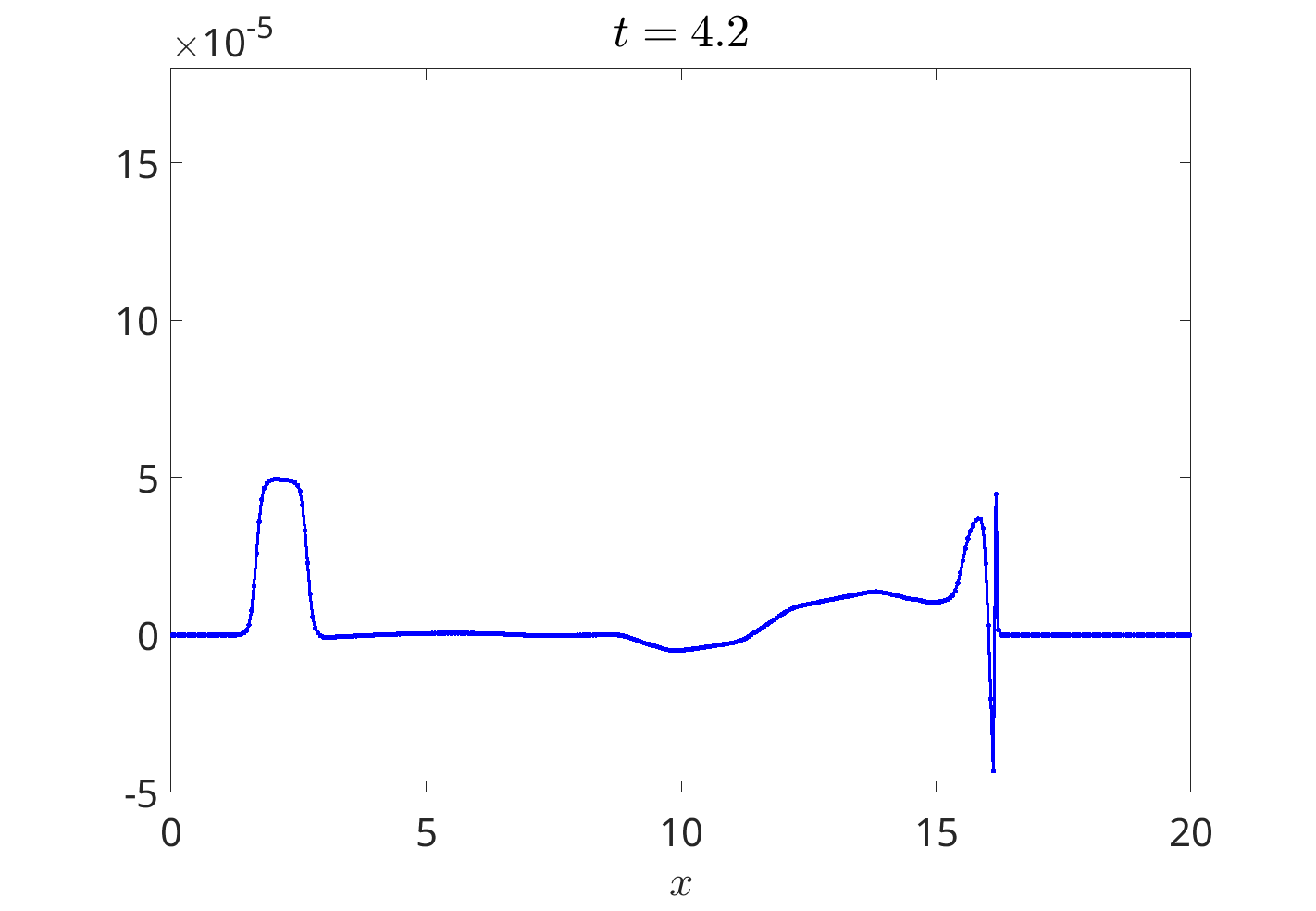}
		\includegraphics[trim=1.9cm 0.7cm 1.8cm 0.2cm, clip, width=0.33\textwidth]{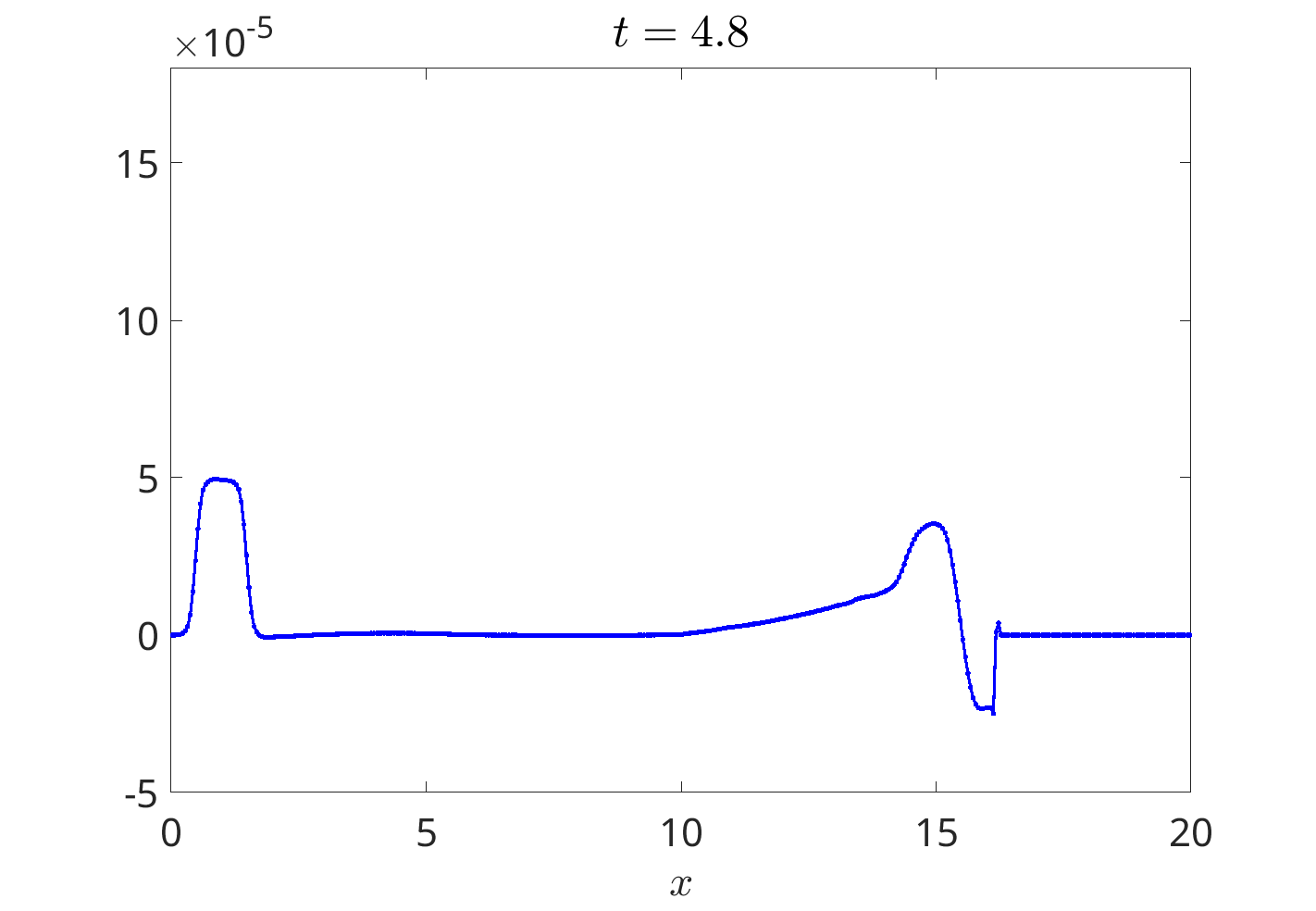}}
	\caption{\sf Example 5: The evolved perturbation $h(x,t) - h_{\rm eq}(x)$ of the wet/dry lake-at-rest steady state at various time slices. }
	\label{f:LaR_pert} 
\end{figure}

\subsubsection*{Example 6---Dam-break problem with discontinuous bottom topography}
In this example, we test a dam-break problem over a discontinuous rectangular bottom topography. 
We use a variant of the problem seen in, e.g., \cite{Noelle07HOWB,Xing05HO}; that is, we consider the following initial conditions on the domain $[0,10]$ with zero-extrapolation boundary conditions:

\begin{equation*}\label{e:h_rp}
	u(x, 0) = 0 \quad \textrm{and} \quad h(x,0) = \begin{cases}
		20-Z(x) & \textrm{if } x \leq 10, \\
		15-Z(x) & \textrm{if } x > 10,
	\end{cases}
\end{equation*}
where the bottom topography is defined as
\begin{equation*}\label{e:h_rp}
	Z(x) = \begin{cases}
		8 & \textrm{if } 3 \leq x \leq 7, \\
		0 & \textrm{otherwise}.
	\end{cases}
\end{equation*}

We evolve this Riemann problem over a discontinuous bathymetry using $N = 100$ and $ N = 1000$ points until a final time $t = 0.3$. 
We present time snapshots of the $h+Z$ solution at times $t = 0.15$ and $t = 0.3$ in Figure \ref{f:RP_DiscZ}. 
At time $t = 0.15$, the outward traveling waves have not yet reached the discontinuity fronts of the bottom topography, so the solution appears as a standard dam-break problem over a flat bathymetry. 
Shortly after, when the waves are crossing these spatial locations, each wave splits---one part continues to be outwardly transmitted, another part is reflected back inward, and the last part becomes a standing wave. 
By time $t = 0.3$, we see six different waves have formed. 
Notice that (i) the presented WB scheme gives well-resolved, non-oscillatory solutions to this discontinuous problem; (ii) the $N = 100$ solution matches well to the $N = 1000$ solution; and  (iii) we do not see unphysical numerical artifacts at the bottom topography discontinuities.

\begin{figure}[ht!]
	\centerline{
		\includegraphics[trim=0.5cm 0.6cm 1.8cm 0.2cm, clip, width=0.45\textwidth]{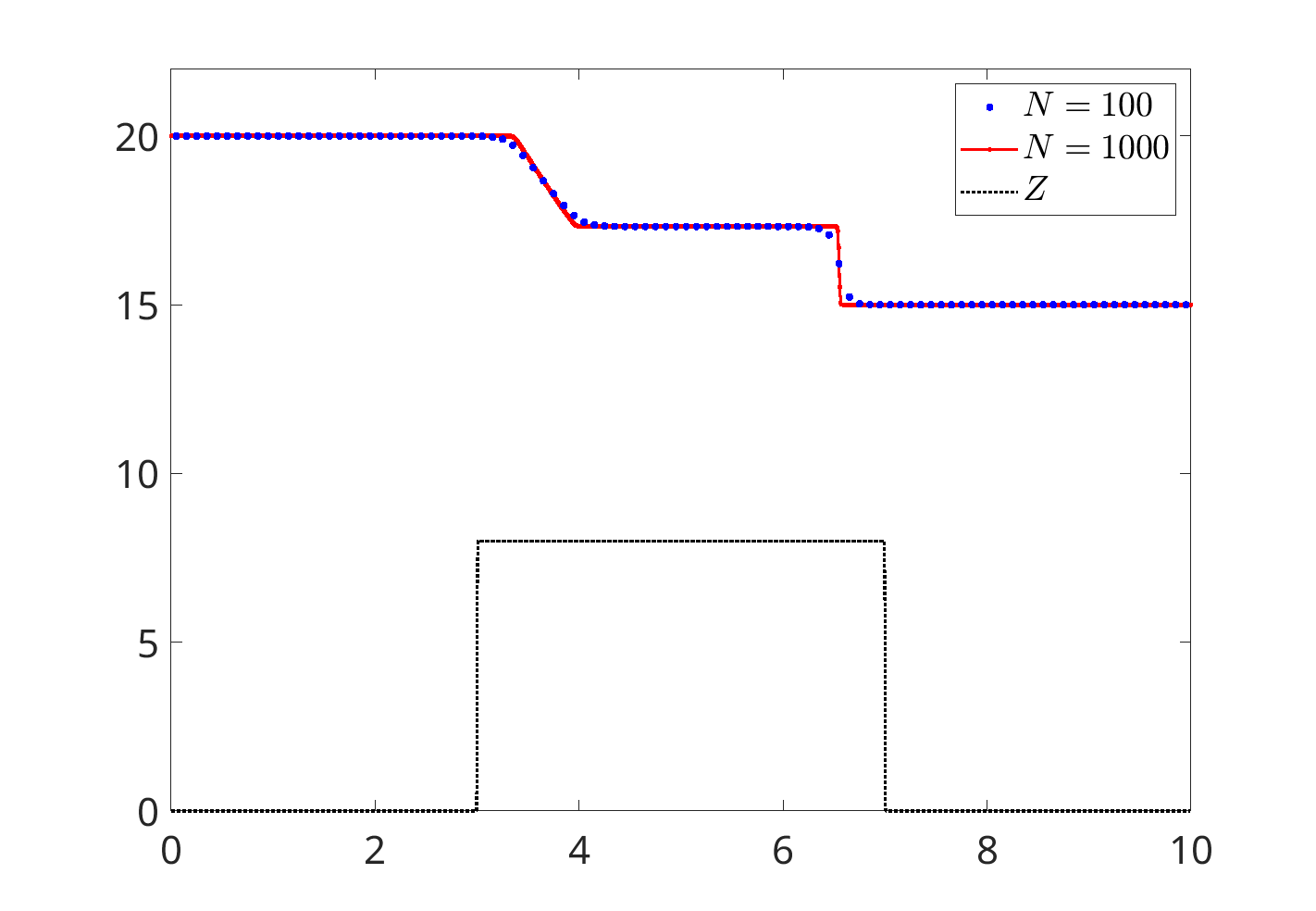}
		\includegraphics[trim=0.5cm 0.6cm 1.8cm 0.2cm, clip, width=0.45\textwidth]{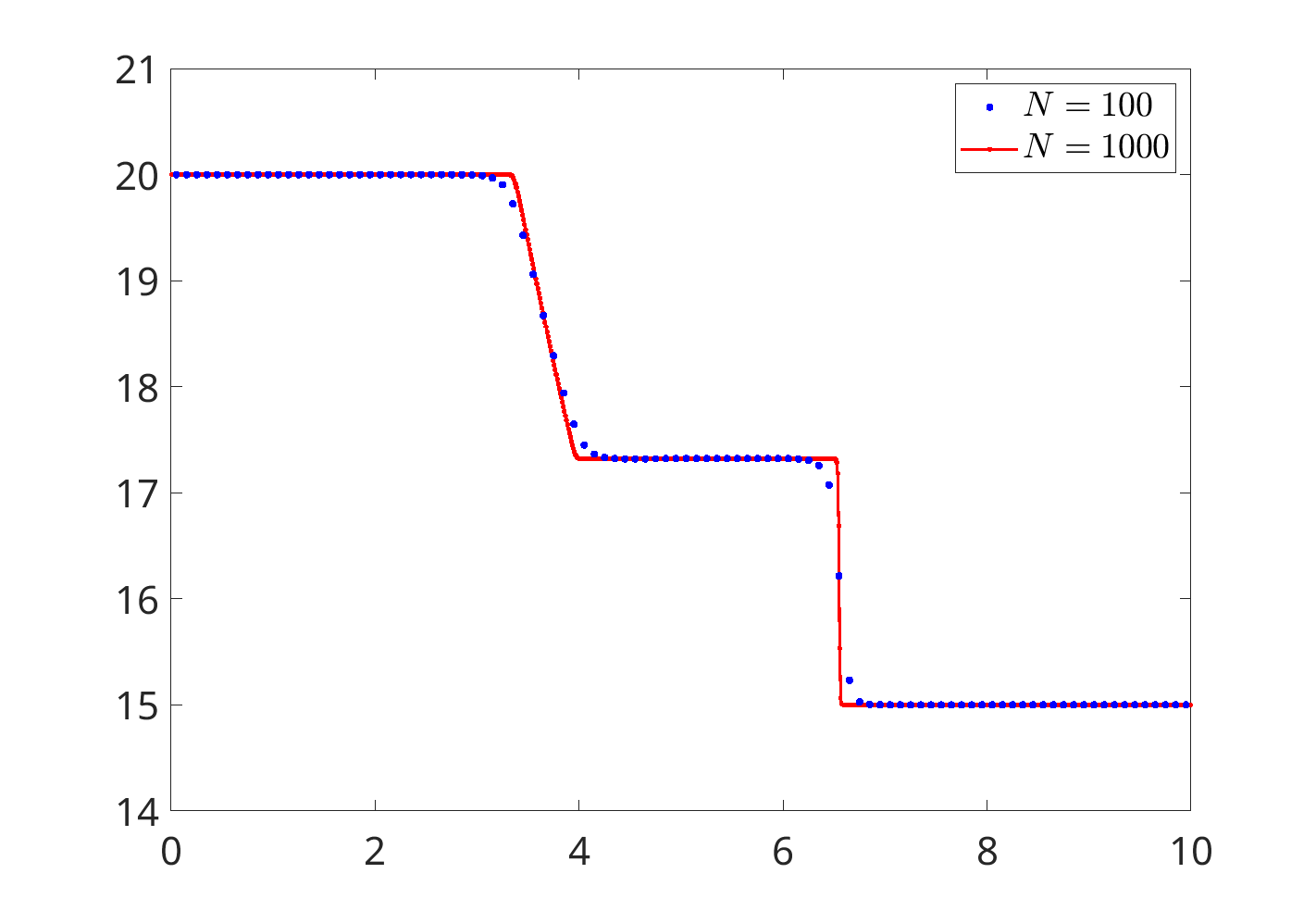}}
		
	\vspace*{7pt}
	\centerline{
		\includegraphics[trim=0.5cm 0.6cm 1.8cm 0.2cm, clip, width=0.45\textwidth]{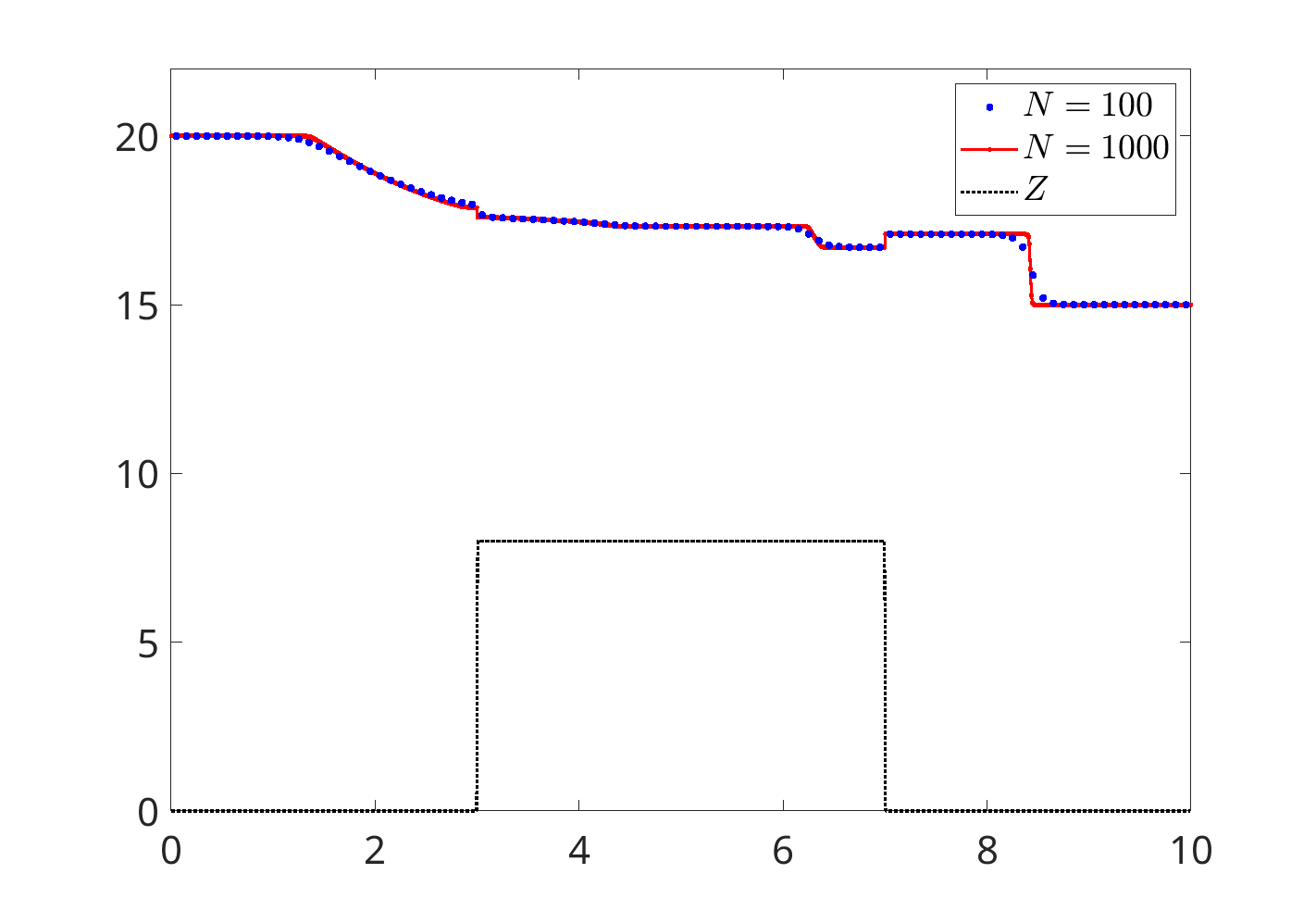}
		\includegraphics[trim=0.5cm 0.6cm 1.8cm 0.2cm, clip, width=0.45\textwidth]{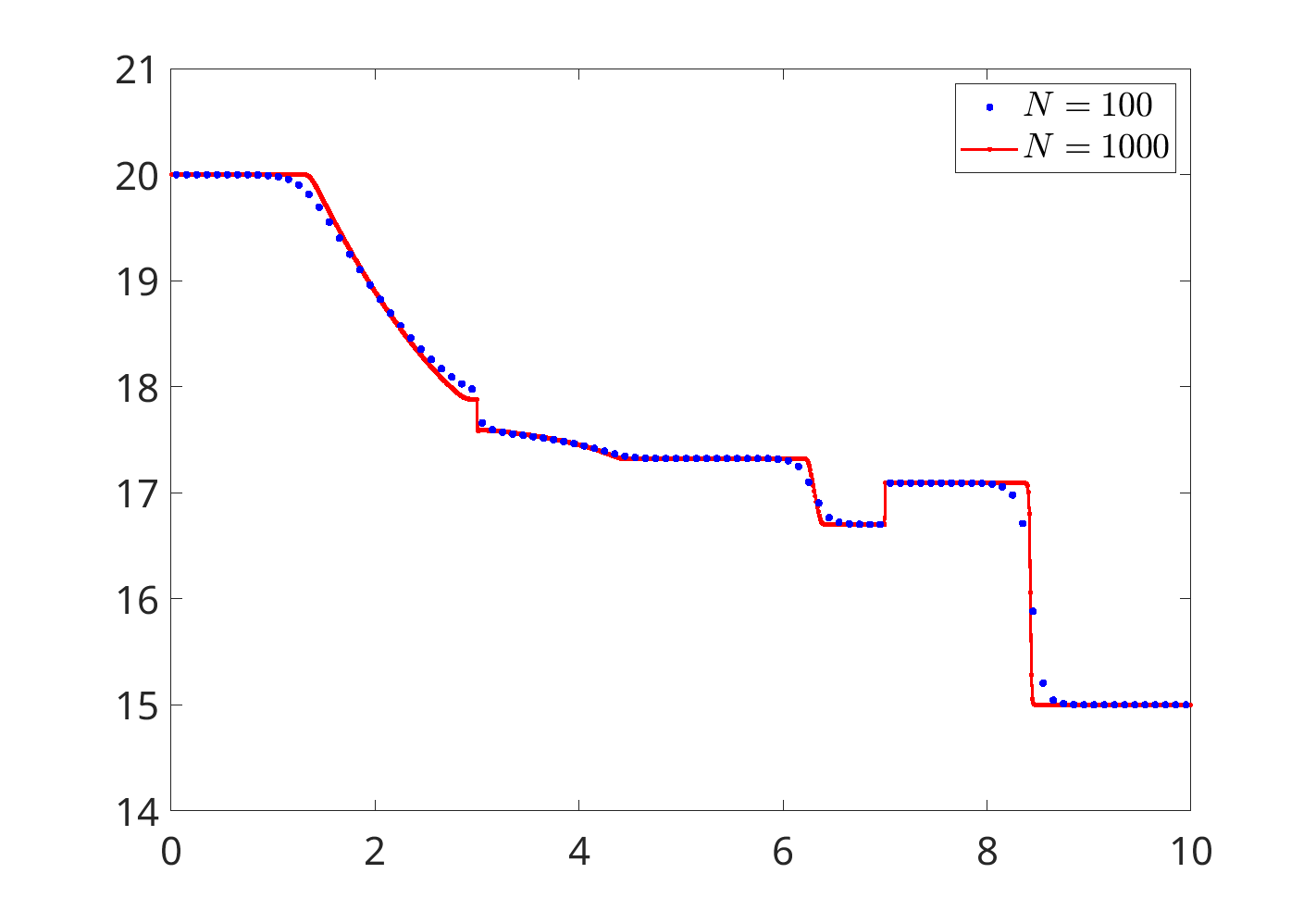}}
	\caption{\sf Example 6: the fluid height $h+Z$ at times $t = 0.15$ (top row) and $t = 0.3$ (bottom row). The left column includes a depiction of the bottom topography, while the right column zooms in on solely $h+Z$.}
	\label{f:RP_DiscZ} 
\end{figure}

\subsubsection*{Example 7---Wet/dry Riemann problem, convergence to moving equilibria}
In the final test, we look at the evolution of a Riemann problem over a dry front to (i) further test the positivity-preserving property; and (ii) present that the proposed method does indeed preserve a moving equilibria state once reached. Following \cite{Cheng19ME}, we modify the initial data in \eqref{e:superICs} (i.e., the supercritical moving equilibria state) to instead be the following Riemann problem on the domain [0,25]: 
\begin{equation*}
	h(x,0) = \begin{cases}
		2 & \textrm{if } x < 5, \\
		0 & \textrm{otherwise,}
	\end{cases}, \qquad 
	q(x,0) = \begin{cases}
		24 & \textrm{if } x < 5, \\
		0 & \textrm{otherwise.}
	\end{cases}
\end{equation*}
Here, we take the same bottom topography \eqref{e:Zbump} as the moving equilibria examples, and take inflow and outflow boundary conditions on the left and right, respectively.

First, we evolve this problem to time $t = 1$ on a mesh with $N = 100$ and $N = 1000$ to view the evolution of the Riemann problem; these results are presented in Figure \ref{f:wetdry_RP}. 
We see the wave front, which at this time, appears as a shock meeting a rarefaction wave. 
We can also see that the first half of the domain matches exactly that from the supercritical moving equilibria, as seen in Figure \ref{f:ME}. 
We additionally run this example and obtain solution snapshots at $t = 5$ and $t = 10$. 
By $t = 5$, the solutions matches that of the supercritical moving equilibria, and thus $t = 10$ should remain at this state as well. 
We present the errors between the $t = 5$ and $t = 10$ solutions in Table \ref{t:eq_conv_errors} for $N = 100$ and $N = 200$, which demonstrates that we are preserving the newly formed equilibrium state, to round-off error, with the proposed WB scheme.
\begin{figure}[ht!]
	\centerline{
		\includegraphics[trim=0.5cm 0.6cm 1.8cm 0.2cm, clip, width=0.49\textwidth]{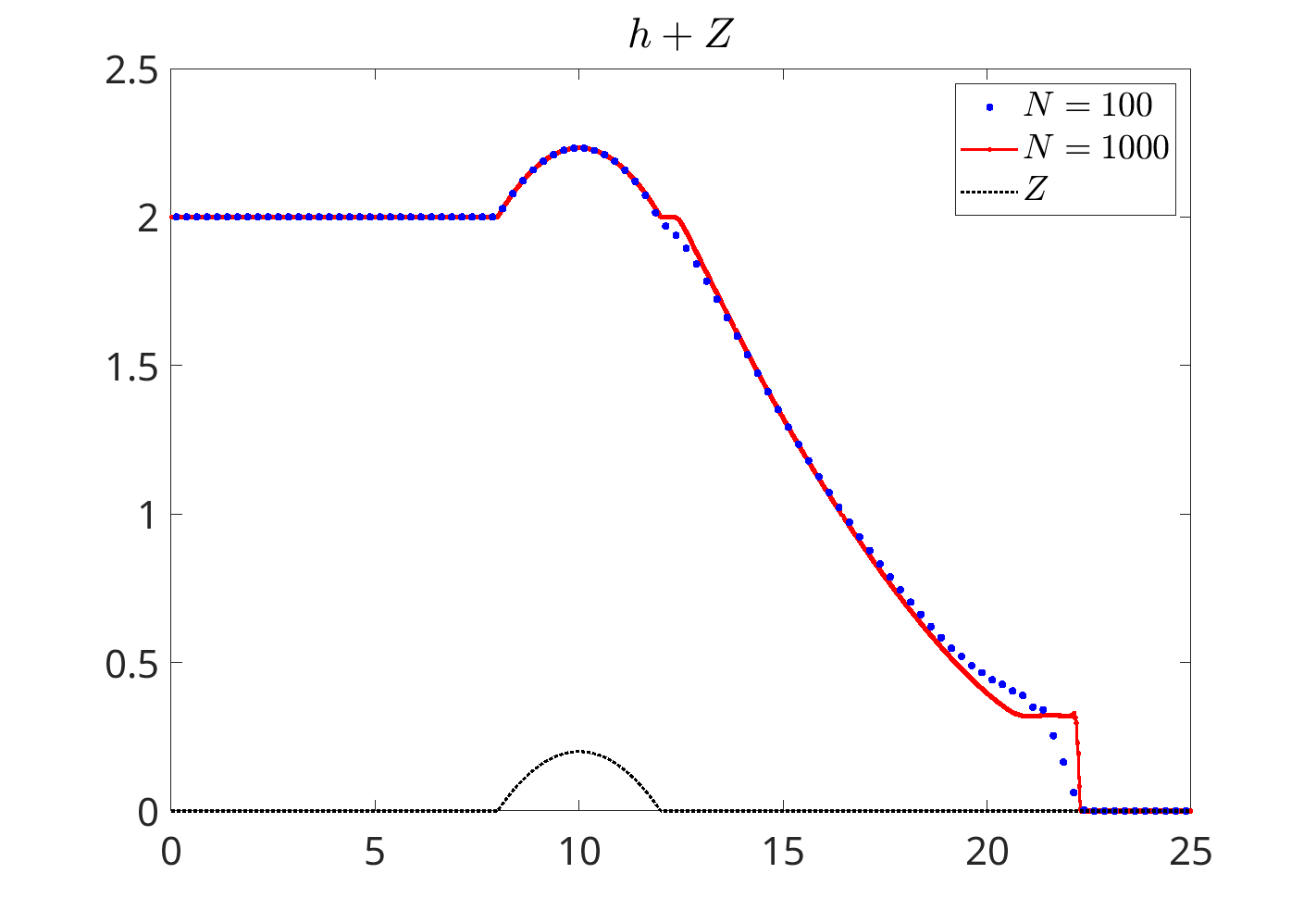}
		\includegraphics[trim=0.5cm 0.6cm 1.8cm 0.2cm, clip, width=0.49\textwidth]{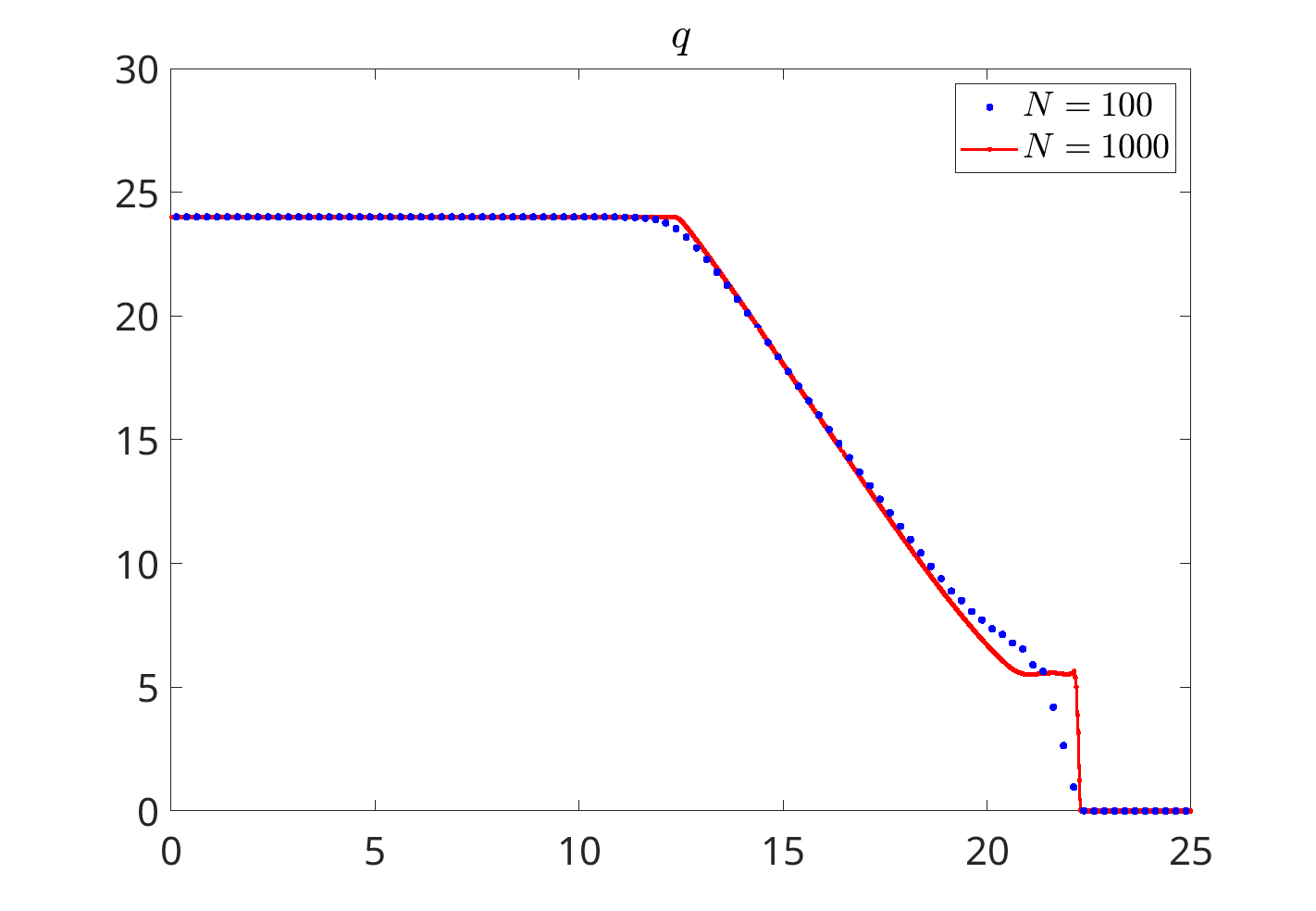}}
	\caption{\sf Example 7: Fluid level $h+Z$ (left) and momentum $q$ (right) at $t = 1$.}
	\label{f:wetdry_RP} 
\end{figure}


\begin{table}[ht]
	\begin{center}
		\caption{\sf Example 7: $L^1$ and $L^\infty $ errors between the solution at $t = 10$ and $t = 5$.}
		\begin{tabular}[c]{l c c c c c}
			\multicolumn{6}{l}{Wet/Dry Lake-at-Rest}\\ \hline
			N & $\norm{h(\cdot,10) - h(\cdot,5) }_1$ & $\norm{q(\cdot,10)  - q(\cdot,5) }_1$ & & $\norm{h(\cdot,10)  - h(\cdot,5) }_\infty$ & $\norm{q(\cdot,10) - q(\cdot,5) }_\infty$  \\ \hline
			100 &	 7.85e-13 &	 4.48e-12 & &	 6.44e-14 &	 4.44e-13 \\
			200 &	 8.92e-13 &	 6.25e-12 & &	 8.68e-14 &	 6.82e-13 \\ \hline
		\end{tabular}\label{t:eq_conv_errors}
	\end{center}
\end{table}

\section{Discussion} \label{sec6}

In this work, we presented a new high-order positivity-preserving fully WB scheme for 1-D  shallow water flow. 
The proposed method uses a flux-globalization approach to be WB in the sense that it preserves the 1-D wet/dry and moving water equilibria of the system.
We reduced oscillations within the high-order A-WENO interpolation using the recently introduced local characteristic equilibrium variable decomposition \cite{chu24local}. 
We ensured positivity of the fluid thickness through the mPDeC time-stepping.
This evolution approach uses nonlinear weighting of terms within a production-destruction framework to form a linearly implicit update---consequently ensuring positivity without any additional restriction on the time step (as long as the fluid thickness is positive within the spatial update). 
With the help of the quasi-conservative form used within the flux-globalization approach, we are able to format the semi-discretization of the shallow water system in this production-destruction framework to use this time update.
To preserve positivity in the spatial discretization, we use the Zhang-Shu limiter.
To additionally ensure robustness of the local characteristic decomposition in dry regions, we introduce a switching indicator; that is, we use A-WENO interpolation of the local characteristic equilibrium variables in wet regions, and interpolate just the equilibrium variables if a near-dry region is detected. 
Finally, we present several numerical results demonstrating robustness, high-order accuracy, the WB property, and the positivity-preserving property---all of which remain essentially non-oscillatory.

While we present this specific WB A-WENO positivity-preserving scheme on the shallow water equations using the mPDeC time evolution, we expect this method to widely generalize in two ways. 
First, while we did not present different systems in this work, the entire method in \S\ref{sec2} is written for general nonconservative hyperbolic systems of balance laws, and can be applied analogously to such systems {when the positive quantities can be framed into a PDS}.
Second, we see that as long as the global flux is foundationally used within the semi-discretization (as it is in flux-globalization-based approaches), then one can use the mPDeC time evolution to ensure unconditional positivity. 
Thus, one can now apply---to general balance laws---the Zhang-Shu positivity limiter without the strong time step restriction. \\

\noindent
\textbf{Acknowledgments:}
Philipp \"Offner was supported by the German Research Foundation (DFG)  within SPP 2410, project 
525866748 (OE 661/5-1) and under the personal grant  520756621 (OE 661/4-1). D.T.
is a member of the INdAM GNCS group in Italy and was supported by the Ateneo Sapienza
project 2024 “Advanced Computational Methods for Real-World Applications: Data-Driven
Models, Hyperbolic Equations and Optimal Control”. \\

\noindent
\textbf{Authorship contribution statement:}
\textbf{Philipp \"Offner:} Conceptualization, Methodology, Funding acquisition, Project administration, Writing---original draft, Writing---review and editing.
\textbf{Michael Redle:} Conceptualization, Methodology, Formal Analysis, Software, Validation, Visualization,  Writing---original draft, Writing---review and editing.
\textbf{Davide Torlo:} Conceptualization, Methodology, Project administration, Writing---review and editing.

\appendix
\section{1-D Affine-invariant WENO-Z interpolation}\label{secA}
We review the Affine-invariant WENO-Z interpolation, introduced in, e.g., \cite{Don22SI_WENO,Li23SI_WENO,Wang22AI_WENO}, in this appendix. 
Consider a scalar function $\Gamma(x)$ whose point values $\Gamma_j $ at $x = x_j$ are available. 
We present how to calculate the interpolated values of $\Gamma$ at $x = x_\jpf$ and the left-sided value at $x = x_\jph$, denoted by $\Gamma_\jpf$ and $\Gamma_j^{\rm E}$, respectively.
One can then obtain $\Gamma_\jmf$ using the mirror-symmetric stencil of $\Gamma_\jpf$; likewise, one can obtain $\Gamma_j^{\rm W}$ with the mirror-symmetric stencil of $\Gamma_j^{\rm E}$. 

For a fifth-order scheme, $\Gamma_j^{\rm E}$ and $\Gamma_\jpf$ are calculated using a weighted average of three parabolic interpolants $\mc P_{j,0}(x)$, $\mc P_{j,1}(x)$, and $\mc P_{j,2}(x)$ constructed on the substencils $[x_{j-2}, x_{j-1}, x_{j}]$, $[x_{j-1}, x_{j}, x_{j+1}]$, $[x_{j}, x_{j+1}, x_{j+2}]$, respectively:
\begin{equation}\label{e:weighted_avg}
	\Gamma(x) \approx \sum_{\ell = 0}^{2} \omega_{j,\ell} \mc P_{j,\ell}(x), \qquad x \in [x_\jmh, x_\jph].
\end{equation} 
To obtain the value $\Gamma_j^{\rm E}$, we evaluate \eqref{e:weighted_avg} at $x = x_\jph$; in this case, we have 
\begin{align*}
	{\mc P}_{j,0}(x_\jph) & = \frac{3}{8} \Gamma_{j - 2} - \frac{5}{4} \Gamma_{j - 1} + \frac{15}{8} \Gamma_{j}, \\
	{\mc P}_{j,1}(x_\jph) & = -\frac{1}{8} \Gamma_{j - 1} +\frac{3}{4} \Gamma_{j} + \frac{3}{8} \Gamma_{j+1}, \\
	{\mc P}_{j,2}(x_\jph) & = \frac{3}{8} \Gamma_{j} + \frac{3}{4} \Gamma_{j + 1} - \frac{1}{8} \Gamma_{j+2}.
\end{align*}
Likewise, we obtain the value $\Gamma_\jpf$ by evaluating \eqref{e:weighted_avg} at $x = x_\jpf$; in this case, the polynomial interpolant values read
\begin{align*}
	{\mc P}_{j,0}(x_\jpf) & = \frac{5}{32} \Gamma_{j - 2} - \frac{9}{16} \Gamma_{j - 1} + \frac{45}{32} \Gamma_{j}, \\
	{\mc P}_{j,1}(x_\jpf) & = -\frac{3}{32} \Gamma_{j - 1} +\frac{15}{16} \Gamma_{j} + \frac{5}{32} \Gamma_{j+1}, \\
	{\mc P}_{j,2}(x_\jpf) & = \frac{21}{32} \Gamma_{j} + \frac{7}{16} \Gamma_{j + 1} - \frac{3}{32} \Gamma_{j+2}.
\end{align*}

The nonlinear weights for $\Gamma_\jpf$ and $\Gamma_j^{\rm E}$ are computed nearly identically---except for different linear weights $d_\ell$; that is, 
\begin{equation}\label{eq:nl_weights}
	\omega_{j,\ell} = \frac{\alpha_{j,\ell}}{ \alpha_{j,0} + \alpha_{j,1} + \alpha_{j,2}}, \qquad
	\alpha_{j,\ell} = d_\ell \br{1 + \p{\frac{\tau_{5}}{\beta_{j,\ell} + \eps}}^p}, \qquad
	\ell = 0,1,2.
\end{equation}
Here, the linear weights for $\Gamma_j^{\rm E}$ read $d_0 = \frac{1}{16}$, $d_1 = \frac{5}{8}$, $d_2 = \frac{5}{16}$, and for $\Gamma_\jpf$ read $d_0 = \frac{7}{64}$, $d_1 = \frac{21}{32}$, $d_2 = \frac{15}{64}$.
The smoothness indicators $\beta_{j,\ell}$ for the corresponding polynomial interpolants ${\mc P}_{j,\ell}(x)$ are computed using 
\begin{align*}
	\beta_{j,0} & = \frac{13}{12} (\Gamma_{j - 2} - 2\Gamma_{j - 1} + \Gamma_{j} )^2 + \frac{1}{4}(\Gamma_{j - 2} -4 \Gamma_{j - 1}+ 3\Gamma_{j})^2, \\
	\beta_{j,1} & = \frac{13}{12} (\Gamma_{j - 1} - 2\Gamma_{j} + \Gamma_{j+1} )^2 + \frac{1}{4}(\Gamma_{j - 1} -\Gamma_{j+1})^2, \\
	\beta_{j,2} & = \frac{13}{12} (\Gamma_{j } - 2\Gamma_{j + 1} + \Gamma_{j+2} )^2 + \frac{1}{4}(3\Gamma_{j} -4 \Gamma_{j + 1} + \Gamma_{j+2})^2.
\end{align*}
Finally, $\tau_5 = \abs{\beta_{j,2}  - \beta_{j,0} }, $ and in this work, we take $p = 2$ and $\eps = 10^{-12}$ within \eqref{eq:nl_weights} for all numerical examples.


\bibliographystyle{abbrv}
\bibliography{biblio}

@article {Abgrall17HO,
	AUTHOR = {Abgrall, R.},
	TITLE = {High order schemes for hyperbolic problems using globally
	continuous approximation and avoiding mass matrices},
	JOURNAL = {J. Sci. Comput.},
	FJOURNAL = {Journal of Scientific Computing},
	VOLUME = {73},
	YEAR = {2017},
	NUMBER = {2-3},
	PAGES = {461--494},
	ISSN = {0885-7474,1573-7691},
	MRCLASS = {65M60 (65M22)},
	MRNUMBER = {3719595},
	MRREVIEWER = {Beny\ Neta},
	DOI = {10.1007/s10915-017-0498-4},
	URL = {https://doi-org.prox.lib.ncsu.edu/10.1007/s10915-017-0498-4},
}

@article {Caselles09Flux,
	AUTHOR = {Caselles, Vicent and Donat, Rosa and Haro, Gloria},
	TITLE = {Flux-gradient and source-term balancing for certain high
	resolution shock-capturing schemes},
	JOURNAL = {Comput. \& Fluids},
	FJOURNAL = {Computers \& Fluids. An International Journal},
	VOLUME = {38},
	YEAR = {2009},
	NUMBER = {1},
	PAGES = {16--36},
	ISSN = {0045-7930,1879-0747},
	MRCLASS = {65M06 (76M20)},
	MRNUMBER = {2645613},
	MRREVIEWER = {Laurent\ E.\ Gosse},
	DOI = {10.1016/j.compfluid.2007.07.023},
	URL = {https://doi-org.prox.lib.ncsu.edu/10.1016/j.compfluid.2007.07.023},
}

@article {Cheng19ME,
	AUTHOR = {Cheng, Yuanzhen and Chertock, Alina and Herty, Michael and
	Kurganov, Alexander and Wu, Tong},
	TITLE = {A new approach for designing moving-water equilibria
	preserving schemes for the shallow water equations},
	JOURNAL = {J. Sci. Comput.},
	FJOURNAL = {Journal of Scientific Computing},
	VOLUME = {80},
	YEAR = {2019},
	NUMBER = {1},
	PAGES = {538--554},
	ISSN = {0885-7474,1573-7691},
	MRCLASS = {65M08 (35L60 35L65 76M12 86A05)},
	MRNUMBER = {3954453},
	DOI = {10.1007/s10915-019-00947-w},
	URL = {https://doi-org.prox.lib.ncsu.edu/10.1007/s10915-019-00947-w},
}

@incollection {Chertock18WB,
	AUTHOR = {Chertock, Alina and Herty, Michael and \"Ozcan, \c Seyma Nur},
	TITLE = {Well-balanced central-upwind schemes for {$2\times2$} systems
	of balance laws},
	BOOKTITLE = {Theory, numerics and applications of hyperbolic problems. {I}},
	SERIES = {Springer Proc. Math. Stat.},
	VOLUME = {236},
	PAGES = {345--361},
	PUBLISHER = {Springer, Cham},
	YEAR = {2018},
	ISBN = {978-3-319-91545-6; 978-3-319-91544-9},
	MRCLASS = {65M08 (35L60)},
	MRNUMBER = {3828226},
	DOI = {10.1007/978-3-319-91545-6\_2},
	URL = {https://doi-org.prox.lib.ncsu.edu/10.1007/978-3-319-91545-6_2},
}

@article {Chertock22WB,
	AUTHOR = {Chertock, Alina and Kurganov, Alexander and Liu, Xin and Liu,
	Yongle and Wu, Tong},
	TITLE = {Well-balancing via flux globalization: applications to shallow
	water equations with wet/dry fronts},
	JOURNAL = {J. Sci. Comput.},
	FJOURNAL = {Journal of Scientific Computing},
	VOLUME = {90},
	YEAR = {2022},
	NUMBER = {1},
	PAGES = {Paper No. 9, 21},
	ISSN = {0885-7474,1573-7691},
	MRCLASS = {65M08 (35L65 76B15 76M12)},
	MRNUMBER = {4342672},
	MRREVIEWER = {Hans-Joachim\ Bungartz},
	DOI = {10.1007/s10915-021-01680-z},
	URL = {https://doi-org.prox.lib.ncsu.edu/10.1007/s10915-021-01680-z},
}

@misc{chu24local,
author = {Shaoshuai Chu and Alexander Kurganov and Mingye Na and Bao Shan Wang and Ruixiao Xin},
title = {Local {Characteristic} {Decomposition} of {Equilibrium} {Variables} for {Hyperbolic} {Systems} of {Balance} {Laws}},
year = {2026},
howpublished = {Preprint, {arXiv}:2412.19791 [math.{NA}] (2026)},
url = {https://arxiv.org/abs/2412.19791},
arXiv = {arXiv:2412.19791}
}

@misc{ziggaf2026arbitrary,
  title={Arbitrary order stationarity preserving stabilized finite elements for multidimensional nonlinear hyperbolic problems. Application to the {Euler} equations with gravity},
  author={Ziggaf, Moussa and Torlo, Davide and Ricchiuto, Mario},
  howpublished = {Preprint, {arXiv}:2603.23185 [math.{NA}] (2026)},
  url = {https://arxiv.org/abs/2603.23185},
  arXiv = {arXiv:2603.23185},
  year={2026}
}

@article {Chu25WB5,
	AUTHOR = {Chu, Shaoshuai and Kurganov, Alexander and Xin, Ruixiao},
	TITLE = {A well-balanced fifth-order {A}-{WENO} scheme based on flux
	globalization},
	JOURNAL = {Beijing J. Pure Appl. Math.},
	FJOURNAL = {Beijing Journal of Pure and Applied Mathematics},
	VOLUME = {2},
	YEAR = {2025},
	NUMBER = {1},
	PAGES = {87--113},
	ISSN = {2994-8444,2994-8452},
	MRCLASS = {99-06},
	MRNUMBER = {5083960},
	DOI = {10.4310/bpam.250414230324},
	URL = {https://doi-org.prox.lib.ncsu.edu/10.4310/bpam.250414230324},
}

@article {Chu25Efficient,
	AUTHOR = {Chu, Shaoshuai and Kurganov, Alexander and Xin, Ruixiao},
	TITLE = {New more efficient {A}-{WENO} schemes},
	JOURNAL = {J. Sci. Comput.},
	FJOURNAL = {Journal of Scientific Computing},
	VOLUME = {104},
	YEAR = {2025},
	NUMBER = {2},
	PAGES = {Paper No. 53, 24},
	ISSN = {0885-7474,1573-7691},
	MRCLASS = {65M06 (35L65 65M08)},
	MRNUMBER = {4923663},
	DOI = {10.1007/s10915-025-02969-z},
	URL = {https://doi-org.prox.lib.ncsu.edu/10.1007/s10915-025-02969-z},
}

@article {Chu22Fifth,
	AUTHOR = {Chu, Shaoshuai and Kurganov, Alexander and Na, Mingye},
	TITLE = {Fifth-order {A}-{WENO} schemes based on the path-conservative
	central-upwind method},
	JOURNAL = {J. Comput. Phys.},
	FJOURNAL = {Journal of Computational Physics},
	VOLUME = {469},
	YEAR = {2022},
	PAGES = {Paper No. 111508, 22},
	ISSN = {0021-9991,1090-2716},
	MRCLASS = {65M08 (76-10)},
	MRNUMBER = {4471537},
	DOI = {10.1016/j.jcp.2022.111508},
	URL = {https://doi-org.prox.lib.ncsu.edu/10.1016/j.jcp.2022.111508},
}

@article {Ciallella23WB,
	AUTHOR = {Ciallella, Mirco and Torlo, Davide and Ricchiuto, Mario},
	TITLE = {Arbitrary high order {WENO} finite volume scheme with flux
	globalization for moving equilibria preservation},
	JOURNAL = {J. Sci. Comput.},
	FJOURNAL = {Journal of Scientific Computing},
	VOLUME = {96},
	YEAR = {2023},
	NUMBER = {2},
	PAGES = {Paper No. 53, 28},
	ISSN = {0885-7474,1573-7691},
	MRCLASS = {65M08},
	MRNUMBER = {4610951},
	DOI = {10.1007/s10915-023-02280-9},
	URL = {https://doi-org.prox.lib.ncsu.edu/10.1007/s10915-023-02280-9},
}

@article {Don22SI_WENO,
	AUTHOR = {Don, Wai Sun and Li, Run and Wang, Bao-Shan and Wang, Yinghua},
	TITLE = {A novel and robust scale-invariant {WENO} scheme for
	hyperbolic conservation laws},
	JOURNAL = {J. Comput. Phys.},
	FJOURNAL = {Journal of Computational Physics},
	VOLUME = {448},
	YEAR = {2022},
	PAGES = {Paper No. 110724, 23},
	ISSN = {0021-9991,1090-2716},
	MRCLASS = {65M08 (76M60)},
	MRNUMBER = {4319348},
	DOI = {10.1016/j.jcp.2021.110724},
	URL = {https://doi-org.prox.lib.ncsu.edu/10.1016/j.jcp.2021.110724},
}

@article {Donat11Hybrid,
	AUTHOR = {Donat, R. and Martinez-Gavara, A.},
	TITLE = {Hybrid second order schemes for scalar balance laws},
	JOURNAL = {J. Sci. Comput.},
	FJOURNAL = {Journal of Scientific Computing},
	VOLUME = {48},
	YEAR = {2011},
	NUMBER = {1-3},
	PAGES = {52--69},
	ISSN = {0885-7474,1573-7691},
	MRCLASS = {65M06 (65M12 76M20)},
	MRNUMBER = {2811690},
	MRREVIEWER = {Tore\ Fl\aa tten},
	DOI = {10.1007/s10915-010-9404-z},
	URL = {https://doi-org.prox.lib.ncsu.edu/10.1007/s10915-010-9404-z},
}

@article {Gascon01Const,
	AUTHOR = {Gasc\'on, Ll. and Corber\'an, J. M.},
	TITLE = {Construction of second-order {TVD} schemes for nonhomogeneous
	hyperbolic conservation laws},
	JOURNAL = {J. Comput. Phys.},
	FJOURNAL = {Journal of Computational Physics},
	VOLUME = {172},
	YEAR = {2001},
	NUMBER = {1},
	PAGES = {261--297},
	ISSN = {0021-9991,1090-2716},
	MRCLASS = {65M06 (35L65)},
	MRNUMBER = {1852329},
	MRREVIEWER = {Mohammad\ Asadzadeh},
	DOI = {10.1006/jcph.2001.6823},
	URL = {https://doi.org/10.1006/jcph.2001.6823},
}

@article {Guerrero22WB,
	AUTHOR = {Guerrero Fern\'andez, E. and Castro D\'iaz, M. J. and Dumbser,
	M. and Morales de Luna, T.},
	TITLE = {An arbitrary high order well-balanced {ADER}-{DG} numerical
	scheme for the multilayer shallow-water model with variable
	density},
	JOURNAL = {J. Sci. Comput.},
	FJOURNAL = {Journal of Scientific Computing},
	VOLUME = {90},
	YEAR = {2022},
	NUMBER = {1},
	PAGES = {Paper No. 52, 38},
	ISSN = {0885-7474,1573-7691},
	MRCLASS = {65M60 (76B15 76B70)},
	MRNUMBER = {4353498},
	MRREVIEWER = {Zhicheng\ Hu},
	DOI = {10.1007/s10915-021-01734-2},
	URL = {https://doi-org.prox.lib.ncsu.edu/10.1007/s10915-021-01734-2},
}

@article {Huang19PP,
	AUTHOR = {Huang, Juntao and Shu, Chi-Wang},
	TITLE = {Positivity-preserving time discretizations for
	production-destruction equations with applications to
	non-equilibrium flows},
	JOURNAL = {J. Sci. Comput.},
	FJOURNAL = {Journal of Scientific Computing},
	VOLUME = {78},
	YEAR = {2019},
	NUMBER = {3},
	PAGES = {1811--1839},
	ISSN = {0885-7474,1573-7691},
	MRCLASS = {65M06 (65L06 65M20 76V05)},
	MRNUMBER = {3934688},
	MRREVIEWER = {Qifeng\ Zhang},
	DOI = {10.1007/s10915-018-0852-1},
	URL = {https://doi-org.prox.lib.ncsu.edu/10.1007/s10915-018-0852-1},
}

@article {Huang19third,
	AUTHOR = {Huang, Juntao and Zhao, Weifeng and Shu, Chi-Wang},
	TITLE = {A third-order unconditionally positivity-preserving scheme for
	production-destruction equations with applications to
	non-equilibrium flows},
	JOURNAL = {J. Sci. Comput.},
	FJOURNAL = {Journal of Scientific Computing},
	VOLUME = {79},
	YEAR = {2019},
	NUMBER = {2},
	PAGES = {1015--1056},
	ISSN = {0885-7474,1573-7691},
	MRCLASS = {65M06 (76Nxx 76V05)},
	MRNUMBER = {3969000},
	DOI = {10.1007/s10915-018-0881-9},
	URL = {https://doi-org.prox.lib.ncsu.edu/10.1007/s10915-018-0881-9},
}

@article {Kurganov07second,
	AUTHOR = {Kurganov, Alexander and Petrova, Guergana},
	TITLE = {A second-order well-balanced positivity preserving
	central-upwind scheme for the {S}aint-{V}enant system},
	JOURNAL = {Commun. Math. Sci.},
	FJOURNAL = {Communications in Mathematical Sciences},
	VOLUME = {5},
	YEAR = {2007},
	NUMBER = {1},
	PAGES = {133--160},
	ISSN = {1539-6746,1945-0796},
	MRCLASS = {76B15 (35L60 35Q53 65M06 76M25)},
	MRNUMBER = {2310637},
	MRREVIEWER = {Henrik\ Kalisch},
	DOI = {10.4310/cms.2007.v5.n1.a6},
	URL = {https://doi-org.prox.lib.ncsu.edu/10.4310/cms.2007.v5.n1.a6},
}

@article {Kurganov23WBPCCU,
	AUTHOR = {Kurganov, Alexander and Liu, Yongle and Xin, Ruixiao},
	TITLE = {Well-balanced path-conservative central-upwind schemes based
	on flux globalization},
	JOURNAL = {J. Comput. Phys.},
	FJOURNAL = {Journal of Computational Physics},
	VOLUME = {474},
	YEAR = {2023},
	PAGES = {Paper No. 111773, 32},
	ISSN = {0021-9991,1090-2716},
	MRCLASS = {65M08 (35L02)},
	MRNUMBER = {4513801},
	MRREVIEWER = {Panters\ Rodr\'iguez-Berm\'udez},
	DOI = {10.1016/j.jcp.2022.111773},
	URL = {https://doi-org.prox.lib.ncsu.edu/10.1016/j.jcp.2022.111773},
}

@article {Li23SI_WENO,
	AUTHOR = {Li, Peng and Li, Tingting and Don, Wai-Sun and Wang, Bao-Shan},
	TITLE = {Scale-invariant multi-resolution alternative {WENO} scheme for
	the {E}uler equations},
	JOURNAL = {J. Sci. Comput.},
	FJOURNAL = {Journal of Scientific Computing},
	VOLUME = {94},
	YEAR = {2023},
	NUMBER = {1},
	PAGES = {Paper No. 15, 32},
	ISSN = {0885-7474,1573-7691},
	MRCLASS = {65M06 (35L65 35Q31 65M12 65M15)},
	MRNUMBER = {4519622},
	MRREVIEWER = {Guanghua\ Ji},
	DOI = {10.1007/s10915-022-02065-6},
	URL = {https://doi-org.prox.lib.ncsu.edu/10.1007/s10915-022-02065-6},
}

@article {Martinez11Hybrid,
	AUTHOR = {Martinez-Gavara, A. and Donat, R.},
	TITLE = {A hybrid second order scheme for shallow water flows},
	JOURNAL = {J. Sci. Comput.},
	FJOURNAL = {Journal of Scientific Computing},
	VOLUME = {48},
	YEAR = {2011},
	NUMBER = {1-3},
	PAGES = {241--257},
	ISSN = {0885-7474,1573-7691},
	MRCLASS = {65M06 (76B15 76M20)},
	MRNUMBER = {2811701},
	MRREVIEWER = {Prabir\ K.\ Daripa},
	DOI = {10.1007/s10915-010-9440-8},
	URL = {https://doi.org/10.1007/s10915-010-9440-8},
}

@article {Meister14PP,
	AUTHOR = {Meister, A. and Ortleb, S.},
	TITLE = {On unconditionally positive implicit time integration for the
	{DG} scheme applied to shallow water flows},
	JOURNAL = {Internat. J. Numer. Methods Fluids},
	FJOURNAL = {International Journal for Numerical Methods in Fluids},
	VOLUME = {76},
	YEAR = {2014},
	NUMBER = {2},
	PAGES = {69--94},
	ISSN = {0271-2091,1097-0363},
	MRCLASS = {65M60 (65M12 86A05)},
	MRNUMBER = {3247707},
	MRREVIEWER = {Juan\ Antonio\ Infante},
	DOI = {10.1002/fld.3921},
	URL = {https://doi-org.prox.lib.ncsu.edu/10.1002/fld.3921},
}

@article {Nessyahu90NonOsc,
	AUTHOR = {Nessyahu, Haim and Tadmor, Eitan},
	TITLE = {Nonoscillatory central differencing for hyperbolic
	conservation laws},
	JOURNAL = {J. Comput. Phys.},
	FJOURNAL = {Journal of Computational Physics},
	VOLUME = {87},
	YEAR = {1990},
	NUMBER = {2},
	PAGES = {408--463},
	ISSN = {0021-9991,1090-2716},
	MRCLASS = {65M12 (35L65 76L05 76M20)},
	MRNUMBER = {1047564},
	MRREVIEWER = {Anders\ Szepessy},
	DOI = {10.1016/0021-9991(90)90260-8},
	URL = {https://doi-org.prox.lib.ncsu.edu/10.1016/0021-9991(90)90260-8},
}

@article {Noelle07HOWB,
	AUTHOR = {Noelle, Sebastian and Xing, Yulong and Shu, Chi-Wang},
	TITLE = {High-order well-balanced finite volume {WENO} schemes for
	shallow water equation with moving water},
	JOURNAL = {J. Comput. Phys.},
	FJOURNAL = {Journal of Computational Physics},
	VOLUME = {226},
	YEAR = {2007},
	NUMBER = {1},
	PAGES = {29--58},
	ISSN = {0021-9991,1090-2716},
	MRCLASS = {76M12 (65M60 76B15)},
	MRNUMBER = {2356351},
	DOI = {10.1016/j.jcp.2007.03.031},
	URL = {https://doi-org.prox.lib.ncsu.edu/10.1016/j.jcp.2007.03.031},
}

@book{Patankar80Numerical,
  author    = {Patankar, Suhas V.},
  title     = {Numerical Heat Transfer and Fluid Flow},
  publisher = {Hemisphere Publishing Corporation},
  year      = {1980},
  series    = {Series in Computational Methods in Mechanics and Thermal Sciences},
  address   = {Washington, D.C.},
  isbn      = {9780891165224}
}

@article {Sweby84High,
	AUTHOR = {Sweby, P. K.},
	TITLE = {High resolution schemes using flux limiters for hyperbolic
	conservation laws},
	JOURNAL = {SIAM J. Numer. Anal.},
	FJOURNAL = {SIAM Journal on Numerical Analysis},
	VOLUME = {21},
	YEAR = {1984},
	NUMBER = {5},
	PAGES = {995--1011},
	ISSN = {0036-1429},
	MRCLASS = {65M05},
	MRNUMBER = {760628},
	MRREVIEWER = {Randall\ J.\ LeVeque},
	DOI = {10.1137/0721062},
	URL = {https://doi-org.prox.lib.ncsu.edu/10.1137/0721062},
}

@article {Vazquez99Improved,
	AUTHOR = {V\'azquez-Cend\'on, Mar\'ia Elena},
	TITLE = {Improved treatment of source terms in upwind schemes for the
	shallow water equations in channels with irregular geometry},
	JOURNAL = {J. Comput. Phys.},
	FJOURNAL = {Journal of Computational Physics},
	VOLUME = {148},
	YEAR = {1999},
	NUMBER = {2},
	PAGES = {497--526},
	ISSN = {0021-9991,1090-2716},
	MRCLASS = {76M20 (86-08)},
	MRNUMBER = {1669644},
	DOI = {10.1006/jcph.1998.6127},
	URL = {https://doi-org.prox.lib.ncsu.edu/10.1006/jcph.1998.6127},
}

@article {Wang22AI_WENO,
	AUTHOR = {Wang, Bao-Shan and Don, Wai Sun},
	TITLE = {Affine-invariant {WENO} weights and operator},
	JOURNAL = {Appl. Numer. Math.},
	FJOURNAL = {Applied Numerical Mathematics. An IMACS Journal},
	VOLUME = {181},
	YEAR = {2022},
	PAGES = {630--646},
	ISSN = {0168-9274,1873-5460},
	MRCLASS = {65M06 (35L65 65M15)},
	MRNUMBER = {4457761},
	DOI = {10.1016/j.apnum.2022.07.007},
	URL = {https://doi-org.prox.lib.ncsu.edu/10.1016/j.apnum.2022.07.007},
}

@article {Xing05HO,
	AUTHOR = {Xing, Yulong and Shu, Chi-Wang},
	TITLE = {High order finite difference {WENO} schemes with the exact
	conservation property for the shallow water equations},
	JOURNAL = {J. Comput. Phys.},
	FJOURNAL = {Journal of Computational Physics},
	VOLUME = {208},
	YEAR = {2005},
	NUMBER = {1},
	PAGES = {206--227},
	ISSN = {0021-9991,1090-2716},
	MRCLASS = {76M20 (65M06 76B15)},
	MRNUMBER = {2144699},
	DOI = {10.1016/j.jcp.2005.02.006},
	URL = {https://doi-org.prox.lib.ncsu.edu/10.1016/j.jcp.2005.02.006},
}

@article {Xing13WB,
	AUTHOR = {Xing, Yulong and Zhang, Xiangxiong},
	TITLE = {Positivity-preserving well-balanced discontinuous {G}alerkin
	methods for the shallow water equations on unstructured
	triangular meshes},
	JOURNAL = {J. Sci. Comput.},
	FJOURNAL = {Journal of Scientific Computing},
	VOLUME = {57},
	YEAR = {2013},
	NUMBER = {1},
	PAGES = {19--41},
	ISSN = {0885-7474,1573-7691},
	MRCLASS = {65M60 (76B15 86A05)},
	MRNUMBER = {3095289},
	MRREVIEWER = {B\"ulent\ Karas\"ozen},
	DOI = {10.1007/s10915-013-9695-y},
	URL = {https://doi-org.prox.lib.ncsu.edu/10.1007/s10915-013-9695-y},
}

@article{ciallella2025,
 author = {Ciallella, Mirco and Micalizzi, Lorenzo and Michel-Dansac, Victor and {\"O}ffner, Philipp and Torlo, Davide},
 title = {A high-order, fully well-balanced, unconditionally positivity-preserving finite volume framework for flood simulations},
 fjournal = {GEM - International Journal on Geomathematics},
 journal = {GEM. Int. J. Geomath.},
 issn = {1869-2672},
 volume = {16},
 pages = {33},
 note = {Id/No 6},
 year = {2025},
 language = {English},
 doi = {10.1007/s13137-025-00262-7},
 zbMATH = {8014827},
 Zbl = {1557.65161}
}

@article{veiga2021,
 author = {Han Veiga, Maria and {\"O}ffner, Philipp and Torlo, Davide},
 title = {{DeC} and {ADER}: similarities, differences and a unified framework},
 fjournal = {Journal of Scientific Computing},
 journal = {J. Sci. Comput.},
 issn = {0885-7474},
 volume = {87},
 number = {1},
 pages = {35},
 note = {Id/No 2},
 year = {2021},
 language = {English},
 doi = {10.1007/s10915-020-01397-5},
 zbMATH = {7316876},
 Zbl = {1459.76086}
}

@article{torlo2020,
 author = {{\"O}ffner, Philipp and Torlo, Davide},
 title = {Arbitrary high-order, conservative and positivity preserving {Patankar}-type deferred correction schemes},
 fjournal = {Applied Numerical Mathematics},
 journal = {Appl. Numer. Math.},
 issn = {0168-9274},
 volume = {153},
 pages = {15--34},
 year = {2020},
 language = {English},
 doi = {10.1016/j.apnum.2020.01.025},
 zbMATH = {7188131},
 Zbl = {1437.65073}
}

@article{ong2020,
 author = {Ong, Benjamin W. and Spiteri, Raymond J.},
 title = {Deferred correction methods for ordinary differential equations},
 fjournal = {Journal of Scientific Computing},
 journal = {J. Sci. Comput.},
 issn = {0885-7474},
 volume = {83},
 number = {3},
 pages = {29},
 note = {Id/No 60},
 year = {2020},
 language = {English},
 doi = {10.1007/s10915-020-01235-8},
 zbMATH = {7212386},
 Zbl = {1442.65124}
}

@article{dutt2000,
 author = {Dutt, Alok and Greengard, Leslie and Rokhlin, Vladimir},
 title = {Spectral deferred correction methods for ordinary differential equations},
 fjournal = {BIT},
 journal = {BIT},
 issn = {0006-3835},
 volume = {40},
 number = {2},
 pages = {241--266},
 year = {2000},
 language = {English},
 doi = {10.1023/A:1022338906936},
 zbMATH = {1496092},
 Zbl = {0959.65084}
}

@misc{arXiv:2602.14392,
 author = {Thomas Izgin and Andreas Meister and Chi-Wang Shu and Davide Torlo},
 title = {Flux-{Balanced} {Patankar}-type {Schemes} for the {Compressible} {Euler} {Equations}},
 year = {2026},
 howpublished = {Preprint, {arXiv}:2602.14392 [math.{NA}] (2026)},
 url = {https://arxiv.org/abs/2602.14392},
 arXiv = {arXiv:2602.14392}
 }

@article{meister2019,
 author = {Kopecz, Stefan and Meister, Andreas},
 title = {On the existence of three-stage third-order modified {Patankar}-{Runge}-{Kutta} schemes},
 fjournal = {Numerical Algorithms},
 journal = {Numer. Algorithms},
 issn = {1017-1398},
 volume = {81},
 number = {4},
 pages = {1473--1484},
 year = {2019},
 language = {English},
 doi = {10.1007/s11075-019-00680-3},
 zbMATH = {7089089},
 Zbl = {1416.65205}
}

@article{burchard2003,
 author = {Burchard, Hans and Deleersnijder, Eric and Meister, Andreas},
 title = {A high-order conservative {Patankar}-type discretisation for stiff systems of production--destruction equations},
 fjournal = {Applied Numerical Mathematics},
 journal = {Appl. Numer. Math.},
 issn = {0168-9274},
 volume = {47},
 number = {1},
 pages = {1--30},
 year = {2003},
 language = {English},
 doi = {10.1016/S0168-9274(03)00101-6},
 zbMATH = {1992879},
 Zbl = {1028.80008}
}

@article{ciallella2022,
 author = {Ciallella, M. and Micalizzi, L. and {\"O}ffner, P. and Torlo, D.},
 title = {An arbitrary high order and positivity preserving method for the shallow water equations},
 fjournal = {Computers and Fluids},
 journal = {Comput. Fluids},
 issn = {0045-7930},
 volume = {247},
 pages = {21},
 note = {Id/No 105630},
 year = {2022},
 language = {English},
 doi = {10.1016/j.compfluid.2022.105630},
 zbMATH = {7602522},
 Zbl = {1521.76408}
}

@article{kurganov2018finite,
author = {Kurganov, Alexander},
title = {Finite-volume schemes for shallow-water equations},
fjournal = {Acta Numerica},
journal = {Acta Numerica},
issn = {0962-4929},
volume = {27},
pages = {289--351},
year = {2018},
language = {English},
doi = {10.1017/S0962492918000028},
zbMATH = {7099157},
Zbl = {1430.76372}
}

@article {michel2017well,
AUTHOR = {Michel-Dansac, Victor and Berthon, Christophe and Clain,
St\'ephane and Foucher, Fran\c coise},
TITLE = {A well-balanced scheme for the shallow-water equations with
topography or {M}anning friction},
JOURNAL = {J. Comput. Phys.},
FJOURNAL = {Journal of Computational Physics},
VOLUME = {335},
YEAR = {2017},
PAGES = {115--154},
ISSN = {0021-9991,1090-2716},
MRCLASS = {65M08 (76B15 76M12)},
MRNUMBER = {3612493},
MRREVIEWER = {Pilar\ Garcia-Navarro},
DOI = {10.1016/j.jcp.2017.01.009},
URL = {https://doi-org.prox.lib.ncsu.edu/10.1016/j.jcp.2017.01.009},
}

@article{berthon2022very,
AUTHOR = {Berthon, Christophe and Bulteau, Sol\`ene and Foucher, Fran\c
coise and M'Baye, Meissa and Michel-Dansac, Victor},
TITLE = {A very easy high-order well-balanced reconstruction for
hyperbolic systems with source terms},
JOURNAL = {SIAM J. Sci. Comput.},
FJOURNAL = {SIAM Journal on Scientific Computing},
VOLUME = {44},
YEAR = {2022},
NUMBER = {4},
PAGES = {A2506--A2535},
ISSN = {1064-8275,1095-7197},
MRCLASS = {65M08 (65M12)},
MRNUMBER = {4469516},
MRREVIEWER = {Junming\ Duan},
DOI = {10.1137/21M1429230},
URL = {https://doi-org.prox.lib.ncsu.edu/10.1137/21M1429230},
}

@Article{mantri2024fully,
 Author = {Mantri, Yogiraj and {\"O}ffner, Philipp and Ricchiuto, Mario},
 Title = {Fully well-balanced entropy controlled discontinuous {Galerkin} spectral element method for shallow water flows: global flux quadrature and cell entropy correction},
 FJournal = {Journal of Computational Physics},
 Journal = {J. Comput. Phys.},
 ISSN = {0021-9991},
 Volume = {498},
 Pages = {35},
 Note = {Id/No 112673},
 Year = {2024},
 Language = {English},
 DOI = {10.1016/j.jcp.2023.112673},
 zbMATH = {7797652}
}

@article{ranocha2017shallow,
author = {Ranocha, Hendrik},
title = {Shallow water equations: split-form, entropy stable, well-balanced, and positivity preserving numerical methods},
fjournal = {GEM - International Journal on Geomathematics},
journal = {GEM. Int. J. Geomath.},
issn = {1869-2672},
volume = {8},
number = {1},
pages = {85--133},
year = {2017},
language = {English},
doi = {10.1007/s13137-016-0089-9},
zbMATH = {6853862},
Zbl = {1432.65156}
}

@article {castro-pares-jsc20,
AUTHOR = {Castro, Manuel J. and Par\'es, Carlos},
TITLE = {Well-balanced high-order finite volume methods for systems of
balance laws},
JOURNAL = {J. Sci. Comput.},
FJOURNAL = {Journal of Scientific Computing},
VOLUME = {82},
YEAR = {2020},
NUMBER = {2},
PAGES = {Paper No. 48, 48},
ISSN = {0885-7474,1573-7691},
MRCLASS = {65M08},
MRNUMBER = {4062845},
MRREVIEWER = {G\"unter\ Karl Franz B\"arwolff},
DOI = {10.1007/s10915-020-01149-5},
URL = {https://doi-org.prox.lib.ncsu.edu/10.1007/s10915-020-01149-5},
}

@article{zhang2011maximum,
  title={Maximum-principle-satisfying and positivity-preserving high-order schemes for conservation laws: survey and new developments},
  author={Zhang, Xiangxiong and Shu, Chi-Wang},
  journal={Proceedings of the Royal Society A: Mathematical, Physical and Engineering Sciences},
  volume={467},
  number={2134},
  pages={2752--2776},
  year={2011},
  publisher={The Royal Society Publishing}
}

@article{perthame1996positivity,
author = {Perthame, Benoit and Shu, Chi-Wang},
title = {On positivity preserving finite volume schemes for {Euler} equations},
fjournal = {Numerische Mathematik},
journal = {Numer. Math.},
issn = {0029-599X},
volume = {73},
number = {1},
pages = {119--130},
year = {1996},
language = {English},
doi = {10.1007/s002110050187},
zbMATH = {932056},
Zbl = {0857.76062}
}

@article {gottlieb2001strong,
AUTHOR = {Gottlieb, Sigal and Shu, Chi-Wang and Tadmor, Eitan},
TITLE = {Strong stability-preserving high-order time discretization
methods},
JOURNAL = {SIAM Rev.},
FJOURNAL = {SIAM Review},
VOLUME = {43},
YEAR = {2001},
NUMBER = {1},
PAGES = {89--112},
ISSN = {1095-7200,0036-1445},
MRCLASS = {65M20 (65L06)},
MRNUMBER = {1854647},
MRREVIEWER = {Thomas\ H.\ Sonar},
DOI = {10.1137/S003614450036757X},
URL = {https://doi-org.prox.lib.ncsu.edu/10.1137/S003614450036757X},
}

@incollection {shu1998essentially,
AUTHOR = {Shu, Chi-Wang},
TITLE = {Essentially non-oscillatory and weighted essentially
non-oscillatory schemes for hyperbolic conservation laws},
BOOKTITLE = {Advanced numerical approximation of nonlinear hyperbolic
equations ({C}etraro, 1997)},
SERIES = {Lecture Notes in Math.},
VOLUME = {1697},
PAGES = {325--432},
PUBLISHER = {Springer, Berlin},
YEAR = {1998},
ISBN = {3-540-64977-8},
MRCLASS = {65Mxx (65-02 76M20)},
MRNUMBER = {1728856},
MRREVIEWER = {Shi\ Jin},
DOI = {10.1007/BFb0096355},
URL = {https://doi-org.prox.lib.ncsu.edu/10.1007/BFb0096355},
}

@misc{hennes2022,
 author = {Hennes Hajduk and Dmitri Kuzmin},
 title = {Bound-preserving and entropy-stable algebraic flux correction schemes for the shallow water equations with topography},
 year = {2022},
 howpublished = {Preprint, {arXiv}:2207.07261 [math.{NA}] (2022)},
 url = {https://arxiv.org/abs/2207.07261},
 arXiv = {arXiv:2207.07261}
}

@incollection{abgrall_2022,
 author = {Abgrall, R{\'e}mi and Ricchiuto, Mario},
 title = {Hyperbolic balance laws: residual distribution, local and global fluxes},
 booktitle = {Numerical fluid dynamics. Methods and computations},
 isbn = {978-981-16-9664-0; 978-981-16-9665-7},
 pages = {177--222},
 year = {2022},
 publisher = {Singapore: Springer},
 language = {English},
 doi = {10.1007/978-981-16-9665-7_7},
 zbMATH = {7615423},
 Zbl = {1513.76096}
}

@article{mario1,
 author = {Ricchiuto, Mario},
 title = {An explicit residual based approach for shallow water flows},
 fjournal = {Journal of Computational Physics},
 journal = {J. Comput. Phys.},
 issn = {0021-9991},
 volume = {280},
 pages = {306--344},
 year = {2015},
 language = {English},
 doi = {10.1016/j.jcp.2014.09.027},
 zbMATH = {6660477},
 Zbl = {1349.76035}
}

@article{barsukow2024structure,
  title={{Structure preserving nodal continuous Finite Elements via Global Flux quadrature}},
  author={Barsukow, Wasilij and Ricchiuto, Mario and Torlo, Davide},
  fjournal={Numerical Methods for Partial Differential Equations},
  journal={Numer. Methods Partial Differ. Equ.},
  volume={41},
  number={1},
  pages = {e23167},
  doi = {https://doi.org/10.1002/num.23167},
  year={2025}
}

@article{barsukow2026genuinely,
title = {Genuinely multi-dimensional stationarity preserving Finite Volume formulation for nonlinear hyperbolic PDEs},
fjournal = {Journal of Computational Physics},
journal={J. Comput. Phys.},
volume = {550},
pages = {114633},
year = {2026},
issn = {0021-9991},
doi = {https://doi.org/10.1016/j.jcp.2025.114633},
url = {https://www.sciencedirect.com/science/article/pii/S0021999125009143},
author = {Wasilij Barsukow and Mirco Ciallella and Mario Ricchiuto and Davide Torlo}
}

@article{barsukow2026stationarity,
  title={Stationarity preserving nodal Finite Element methods for multi-dimensional linear hyperbolic balance laws via a Global Flux quadrature formulation},
  author={Barsukow, Wasilij and Ricchiuto, Mario and Torlo, Davide},
  fjournal={Journal of Scientific Computing},
  journal={J. Sci. Comput.},
  volume={109},
  number={1},
  pages={8},
  year={2026},
  publisher={Springer},
  doi={10.1007/s10915-026-03392-8}
}

@phdthesis{micalizzi2024efficient,
  title={Efficient Iterative Arbitrary High Order Methods: Adaptivity and Structure Preservation},
  author={Micalizzi, L},
  year={2024},
  school={University of Zurich}
}

@article{bender2025,
  title={{The Lax--Wendroff theorem for Patankar-type methods applied to hyperbolic conservation laws}},
  author={Bender, Janina and Izgin, Thomas and {\"O}ffner, Philipp and Torlo, Davide},
	JOURNAL = {Comput. \& Fluids},
	FJOURNAL = {Computers \& Fluids. An International Journal},
  pages={106885},
  year={2025},
  publisher={Elsevier},
  doi={10.1016/j.compfluid.2025.106885}

}

\end{document}